\documentclass[11pt, a4paper]{article}
\usepackage[a4paper, nomarginpar]{geometry}	
\usepackage{graphicx}
\usepackage[dvipsnames]{xcolor}
\usepackage{appendix}
\usepackage{enumerate}
\usepackage{amsmath, amsthm, amssymb, bbm, bm}
\usepackage{amsfonts}
\usepackage{mathrsfs}
\usepackage[hidelinks]{hyperref}
\usepackage{tensor}
\usepackage{tikz-cd}
\usepackage{framed}
\usepackage[backend=bibtex,style=ieee]{biblatex}
\bibliography{bib}

\hypersetup{colorlinks = true, allcolors = black, linkcolor=blue, citecolor = magenta}

\colorlet{shadecolor}{gray!25}   
\renewenvironment{leftbar}{%
  \MakeFramed {\advance\hsize-\width \FrameRestore}}%
{\endMakeFramed}

\theoremstyle{plain}
\newtheorem{lemma}{Lemma}[section]
\newtheorem{proposition}[lemma]{Proposition}
\newtheorem{theorem}[lemma]{Theorem}
\newtheorem{corollary}[lemma]{Corollary}
\theoremstyle{definition}
\newtheorem{definition}[lemma]{Definition}
\newtheorem{remark}[lemma]{Remark}
\newtheorem{exampleus}[lemma]{Example}
\newenvironment{example}[1][]
    {\begin{leftbar} \vspace{-9pt} \begin{exampleus}[#1]}
    {\end{exampleus} \vspace{-9pt}\end{leftbar}}

\newcommand{\N}{\mathbb{N}}
\newcommand{\Z}{\mathbb{Z}}
\newcommand{\R}{\mathbb{R}}

\newcommand{\id}{\mathrm{id}}

\newcommand{\gvec}{\mathsf{gVec}}
\newcommand{\gcas}{\mathsf{gcAs}}
\newcommand{\gman}{\mathsf{gMan}}
\newcommand{\Set}{\mathsf{Set}}
\newcommand{\Op}{\mathsf{Op}}

\newcommand{\scrF}{\mathscr{F}}

\newcommand{\scrS}{\mathscr{S}}
\newcommand{\scrX}{\mathscr{X}}
\newcommand{\scrP}{\mathscr{P}}
\newcommand{\scrQ}{\mathscr{Q}}
\newcommand{\scrL}{\mathscr{L}}
\newcommand{\calE}{\mathcal{E}}
\newcommand{\calF}{\mathcal{F}}

\newcommand{\calI}{\mathcal{I}}
\newcommand{\calJ}{\mathcal{J}}

\newcommand{\calM}{\mathcal{M}}
\newcommand{\calN}{\mathcal{N}}
\newcommand{\calS}{\mathcal{S}}

\newcommand{\frJ}{\mathfrak{J}}
\newcommand{\rmT}{\mathrm{T}}

\newcommand{\fP}{\mathbf{P}}
\newcommand{\fp}{\mathbf{p}}

\newcommand{\fr}{\mathbf{r}}
\newcommand{\fF}{\mathbf{F}}
\newcommand{\fE}{\mathbf{E}}
\def\f1{\mathbf{1}}
\newcommand{\fG}{\mathbf{G}}

\newcommand{\rest}[2]{\left.#1\right|_{#2}}
\newcommand{\parest}[2]{\rest{\left(#1\right)}{#2}}
\newcommand{\bld}[1]{\mathbf{#1}}
\newcommand{\uline}[1]{\underline{#1}}
\newcommand{\dd}[2]{\frac{\partial #1}{\partial
 #2}}

\newcommand{\ol}{\overline}
\newcommand{\ul}{\underline}
\newcommand{\ssm}{\smallsetminus}

\DeclareMathOperator{\gdim}{gdim}
\DeclareMathOperator{\grk}{grk}
\DeclareMathOperator{\im}{im}
\DeclareMathOperator{\Hom}{Hom}
\DeclareMathOperator{\supp}{supp}
\DeclareMathOperator{\Sff}{Sff}
\DeclareMathOperator{\sff}{sff}

\newcommand{\gR}{\mathrm{g}\mathbb{R}}
\newcommand{\add}{\Sigma}
\newcommand{\1}{\mathbbm{1}}
\newcommand{\svbun}{\mathsf{sVBun}}
\newcommand{\gvbun}{\mathsf{gVBun}}
\newcommand{\avbun}{\mathsf{aVBun}}
\newcommand{\cifty}{C^\infty}
\newcommand{\clin}{C^{\mathrm{lin}}}
\newcommand{\TM}{\mathrm{T}\mathcal{M}}

\title{Graded Vector Bundles}
\author{JV, RS}

\begin{document}
\begin{flushright}
\today
\end{flushright}
\vspace{0.7cm}
\begin{center}

\baselineskip=13pt {\Large \bf{Threefold Nature of Graded Vector Bundles}\\}
 \vskip0.5cm
 {\large{Rudolf Šmolka$^{1}$, Jan Vysoký$^{2}$}}\\
 \vskip0.6cm
$^{1}$\textit{Faculty of Nuclear Sciences and Physical Engineering, Czech Technical University in Prague\\ Břehová 7, 115 19 Prague 1, Czech Republic, smolkrud@fjfi.cvut.cz}\\ 
$^{2}$\textit{Faculty of Nuclear Sciences and Physical Engineering, Czech Technical University in Prague\\ Břehová 7, 115 19 Prague 1, Czech Republic, vysokjan@fjfi.cvut.cz}\\
\vskip0.3cm
\end{center}

\begin{abstract}
Graded vector bundles over a given $\Z$-graded manifold can be defined in three different ways: certain sheaves of graded modules over the structure sheaf of the base graded manifold, finitely generated projective graded modules over the algebra of global functions on the base graded manifold, or locally trivial graded manifolds with a suitable linear structure. We argue that all three approaches are the same. More precisely, the respective categories are proved to be equivalent.
\end{abstract}

{\textit{Keywords}: Graded vector bundles, graded manifolds, graded modules, Serre--Swan theorem,

\tableofcontents

\section{Introduction}
Vector bundles form a backbone of modern differential geometry. The notion formalizes the idea of gluing a vector space to each point of a smooth manifold, allowing to describe vector space valued quantities on manifolds, e.g. vector fields, differential forms, differential operators, etc. Vector bundles allow for a more conceptual description of connections and various characteristic classes. Vector bundles form the foundation of K-theory, an enormously successful theory (proving e.g. Bott periodicity or Atiyah-Singer index theorem). Last but not least, vector bundles are necessary in every modern description of theoretical physics.

In ordinary differential geometry, it is most natural to view vector bundles as particular surjective submersions $\pi: E \rightarrow M$, where for sufficiently small open sets $U$ in $M$, their inverse image $\pi^{-1}(U)$ becomes diffeomorphic with the Cartesian product of $U$ with a fixed vector space $V$ (together with some additional compatibility conditions). We call this the \textit{geometric approach} to vector bundles. In this perspective, local sections of $E$ are simply smooth maps from open subsets of $M$ to the total space $E$, forming a right inverse to the projection.

It turns out that sets of local sections inherit a rich algebraic structure. Altogether, they form a sheaf $\Gamma_{E}$ of \textit{local sections of $E$} on the base manifold $M$. This is a sheaf of modules over the sheaf of smooth functions on $M$, having the additional property that every section can be locally decomposed as a unique finite combination of some local frame. One says that $\Gamma_{E}$ is locally freely and finitely generated. It turns out that every such sheaf can be obtained from some vector bundle and this vector bundle is unique up to an isomorphism. One can thus choose to describe vector bundles as certain sheaves of modules, hence we call it the \textit{sheaf approach}. This brings us right on the border between differential geometry and commutative algebra. 

Global sections of a given vector bundle $\pi: E \rightarrow M$ form a module over the algebra of global smooth functions on $M$. This module is finitely generated and projective. It turns out that \textit{every} such module is isomorphic to a module of global sections of some vector bundle, unique up to an isomorphism. This is the result of the famous \textit{Serre--Swan theorem} \cite{serre1955faisceaux,swan1962vector}, and it gives a third way of viewing vector bundles, which we choose to call the \textit{algebraic approach}.

In supergeometry, vector (super)bundles are usually defined using the sheaf approach, see \cite{manin1988gauge}, \cite{deligne1999quantum}. Some references contain both sheaf and geometrical definition and comment on their equivalence \cite{bruzzo1988supermanifolds, Balduzzi_2011, bartocci2012geometry}. Using a geometric approach similar to the one in this paper, authors in \cite{bruceGrabowski2024} define vector and principal bundles in the category of $\Z^n_2$-graded manifolds, and authors in \cite{grabowskaGrabowski2024} generalize vector superbudles into so-called homogeneity supermanifolds. Recently there appeared a version of Serre--Swan theorem for compact supermanifolds \cite{morye2025serre}. Also note that the Serre--Swan theorem for general non-compact smoooth manifolds first appeared in \cite{nestruev2003smooth}.

In recent years, there was a reinvigoration of interest in $\Z$-graded manifolds with local coordinates of arbitrary degrees \cite{vysoky2022graded, kotov2024category, fairon2017introduction}, following on the non-negatively graded (or just $\N$-graded) manifolds \cite{Kontsevich:1997vb, severa2001some, Voronov:2019mav} and \cite{mehta2006supergroupoids,2011RvMaP..23..669C}. The aim of this paper is to extend these three ways of working with vector bundles into the realm of $\Z$-graded geometry. By ``graded'', we will thus exclusively mean $\Z$-graded. We will also often abbreviate the term ``graded vector bundle'' as GVB. Also let us emphasize that our notion should not be confused with the one in \cite{bruce2016introduction}.

We have defined GVBs using the sheaf approach in \cite{Vysoky:2022gm} and extensively examined this direction. We have constructed a total space graded manifold, thus pointing towards the geometrical interpretation. The main result of this paper is that, just as in the non-graded case and supermanifold case, the sheaf, geometric and algebraic approaches to GVBs are \textit{equivalent} in the sense of equivalence of the respective categories. To skip ahead to these results, see Definitions \ref{def_svbun}, \ref{def_avbun} and \ref{def_GVB2} together with Theorem \ref{thm_SerreSwan} (Graded Serre--Swan) and Theorem \ref{prop_main}. Let us point out that we prove the equivalence of categories of GVBs over \textit{all} (connected) graded manifolds. The discussion thus includes morphisms of GVBs over non-identity graded smooth maps. 

Note that the approach presented here can be easily adapted to supermanifolds and $\Z_{2}^{n}$-graded manifolds. In fact, one can in principle repeat the arguments for any theory where some version of Proposition \ref{tvrz_fiberinjsur} applies. Let us also point out that there is a general procedure for ringed spaces \cite{morye2017serre} which should be in principle adaptable to the $\Z$-graded case. However, in order to keep the text self-contained, we have decided to directly prove the graded Serre--Swan theorem. 

The text is structured as follows: Section 2 provides a quick overview of $\Z$-graded manifolds, giving a succinct summary of the basic concepts required to understand the rest of the text. For a thorough introduction to the field of $\Z$-graded manifolds the reader is invited to the introductory text \cite{Vysoky:2022gm}.

Section 3 deals with graded vector bundles in the sheaf approach. We recall the definitions from \cite{Vysoky:2022gm} and we state some observations, mostly regarding morphisms; for example, that a morphism is fully determined by the map of global sections (Proposition \ref{prop_GVBmorphisms}) or sufficient conditions for when the image or kernel of a morphism is a graded vector subbundle (Proposition \ref{prop_kerimsubbundles}).

Section 4 introduces the algebraic approach and as such is concerned mostly with properties of graded modules. Notable results include the already mentioned graded Serre--Swan Theorem (Theorem \ref{thm_SerreSwan}) and its application in a short proof of the fact that the ``naive" tensor product of GVBs does not require sheafification to be a correct tensor product (Theorem \ref{thm_tensorProd}).

Section 5 is tasked with the geometric approach --- the graded manifold side of GVBs. In \cite{grabowski2009higher} and \cite{grabowskiRotkiewicz_2012} it is shown that a vector bundle is fully described solely in terms of the homogeneous structure on the total space. It is in this spirit that we describe GVBs in terms of certain homothety maps encompassing the homogeneous structure. The main result of the section is simply the equivalence with the sheaf approach (Theorem \ref{prop_main}). Of interest may be the ``geometric" construction of the pullback bundle (Example \ref{ex_pullbackBundle}). We also revisit degree shifts of graded vector bundles and note that local sections of a GVB may be realized as right inversions to the projection from appropriately shifted total space (Theorem \ref{thm_sections_and_sections}).

\section*{Acknowledgments}
The research of RŠ an JV was supported by grant GAČR 24-10031K. JV is also grateful for a financial support from MŠMT under grant no. RVO 14000. 

This work was also supported by the Grant Agency of the Czech Technical University in Prague, grant no. SGS25/163/OHK4/3T/14.

\section{$\Z$-graded Manifolds: A Quick Guide}
This article uses the definition of $\Z$-graded manifolds put forward in \cite{Vysoky:2022gm}. Let us briefly recall the relevant notions - for a more thorough introduction, the reader is referred to the original article.

A real \textbf{graded vector space} is defined as a sequence $V \equiv (V_i)_{i \in \Z}$ of real vector spaces, each called the component vector space. In this text we will not consider vector spaces over any fields other than $\R$. By an element of a graded vector space, $v \in V$, we mean that there exists a unique integer $i \in \Z$ such that $v \in V_i$. This $i$ is called the \textbf{degree} of $v$ and we write $i =: |v|$. Addition of vectors and multiplication by a scalar is defined as in the component vector spaces; in particular, one can only sum vectors \textbf{of the same degree}, there are no inhomogeneous elements of a graded vector space. 
We say that a graded vector space $V$ is \textbf{finite dimensional} if $\sum_{i \in \Z} \dim{V_i} < + \infty$. One can also consider every ordinary vector space $V$ as a graded vector space, where the zeroth component is $V$ itself and all other components are $\{0\}$. 
Considering two graded vector spaces $V,W$, a \textbf{graded linear map} $\varphi : V \to W$ between them is defined as a collection $\varphi \equiv (\varphi_i)_{i \in \Z}$, where for every $i \in \Z$ and for some fixed $k \in \Z$, $\varphi_i : V_i \to W_{i + k}$ is a linear map. The integer $k$ is then called the \textbf{degree} of $\varphi$, denoted as $k =: |\varphi|$. Graded vector spaces, together with graded linear maps of degree zero, form a category $\gvec$.

\begin{example}
Let $(n_j)_{j \in \Z}$ be a sequence of non-negative integers. Then $\R^{(n_j)}$ will denote the graded vector space $(\R^{(n_j)})_i := \R^{n_i}$.
\end{example}
A \textbf{graded algebra} $A$ is a graded vector space with the additional bilinear operation of multiplication of two vectors $(a,b) \mapsto a \cdot b$. One can multiply vectors of any degrees, and the degree of the product needs to satisfy
\begin{equation}
    |a \cdot b| = |a| + |b|.
\end{equation}
An associative graded algebra is one where the multiplication is associative, a unital graded algebra is one which contains a unit, i.e. an element $1 \in A$ such that $1 \cdot a = a = a \cdot 1$ for every $a \in A$. Such a unit clearly needs to have a degree zero, and if it exists, it is unique. A \textbf{graded commutative} algebra is a graded algebra $A$ where the multiplication satisfies the graded commutativity condition
\begin{equation}
    a \cdot b = (-1)^{|a| |b|} \, b \cdot a,
\end{equation}
for any $a,b \in A$. Consider two graded unital algebras $A,B$ and a graded linear map of degree zero $\varphi : A \to B$. We say that $\varphi$ is a \textbf{graded algebra morphism} if  $\varphi(1) = 1$ and $\varphi(a\cdot b) = \varphi(a)\cdot \varphi(b)$, for all $a,b \in A$. In this text we only work with graded commutative, associative, unital algebras; let us denote their category as $\gcas$. Note that by excluding multiplication by scalar from the definition of a graded algebra, one obtains a graded ring.

Consider now a topological space $X$. By a \textbf{presheaf of graded vector spaces} on $X$ we mean a functor $\mathcal{F} : \Op(X)^{op} \to \gvec$, where $\Op(X)$ denotes the category of open subsets of $X$ with inclusions as arrows, and the superscript $op$ denotes the opposite category. In particular, for every open set $U \subseteq X$ there is a graded vector space $\calF(U) \in \gvec$, and for every inclusion of open sets $V \subseteq U$ there is a graded linear map of degree zero $\calF^U_V : \calF(U) \to \calF(V)$ called a \textbf{restriction}. 
As $\calF$ is a functor, there is $\calF^U_U = \id_{\calF(U)}$ and $\calF^V_W \circ \calF^U_V = \calF^U_W$ for all open $W \subseteq V \subseteq U \subseteq X$. For $v \in \calF(U)$ we usually denote $\calF^U_V(v) =: \rest{v}{V}$. We say that $\calF$ is a \textbf{sheaf} if for every open subset $U \subseteq X$, any open cover $\{U_\alpha\}_{\alpha \in I}$ of $U$ and any collection $\{v_\alpha\}_{\alpha \in I}$, $v_\alpha \in \calF(U_\alpha)$, such that $\rest{v_\alpha}{U_{\alpha \beta}} = \rest{v_\beta}{U_{\alpha \beta}}$ for all $\alpha, \beta \in I$, there exists a unique $v \in \calF(U)$, such that $\rest{v}{U_\alpha} = v_\alpha$, for every $\alpha \in I$. Here and throughout $U_{\alpha \beta} := U_\alpha \cap U_\beta$.
Similarly we may consider sheaves valued in categories other than $\gvec$, such as $\gcas$.

To motivate the first example of a graded manifold, consider how any choice of basis of the vector space $\R^n$ gives rise to global coordinates (the dual basis), making $\R^n$ into a smooth manifold. We now endeavor to make $\R^{(n_j)}$ into a graded manifold in a roughly similar manner.
\begin{example}\label{example_graded_domain}
    Consider a finite-dimensional graded vector space $\R^{(n_j)}$, and let $\{\xi_i^{(k)}\}_{i = 1}^{n_k}$ be a basis for $(\R^{(n_j)})_k$ for every $k \in \Z$. In particular, $|\xi_i^{(k)}| = k$. Let us concatenate all non-zero degree bases into $\{\xi_i\}_{i = 1}^{\tilde{m}} := \cup_{k \neq 0} \{\xi^{(k)}_j\}_{j = 1}^{n_k}$. Consider now their duals, i.e. graded linear maps $\xi^i : \R^{(n_j)} \to \R$ of degree $|\xi^i| = -|\xi_i|$, defined by $\xi^i(\xi_k) = \delta\indices{^i_k}$, where $\delta\indices{^i_k} = 1$ for $i = k$, and $\delta\indices{^i_k} = 0$ of degree $|\xi_k| - |\xi_i|$ for $i \neq k$. In particular, for any $j \in \Z\setminus\{0\}$ there are exactly $n_{-j}$ elements $\xi^i$ of degree $j$. These will play the role of graded coordinates on $\R^{(n_j)}$. Let us now denote $m_j := n_{-j}$ for any $j \in \Z$.

    Next, we may construct a sheaf $\cifty_{(m_{j})}$ valued in $\gcas$ on the space $\R^{m_0}$ like so: for any open $U \subseteq \R^{m_0}$ let $\cifty_{(m_j)}(U)$ be a graded algebra whose elements of degree $k$ are formal power series
    \begin{equation}\label{eq1}
        f := \sum_{\mathbf{p} \in \N^{\tilde{m}}_k} f_{\mathbf{p}} (\xi^1)^{p_1}\cdots (\xi^{\tilde{m}})^{p_{\tilde{m}}} \equiv \sum_{\mathbf{p} \in \N^{\tilde{m}}_k} f_{\mathbf{p}}\xi^{\mathbf{p}},
    \end{equation}
    where $\tilde{m} = \sum_{k \neq 0}m_k$, $f_{\mathbf{p}}$ are smooth functions on $U$ and the multiindices $\mathbf{p}$ range over the set
    \begin{equation}
        \N^{\tilde{m}}_{k} := \{\mathbf{p} \in (\N_{0})^{\tilde{m}} \ : \ \sum_{i = 1}^{\tilde{m}}p_i|\xi^i| = k \ \text{and} \ p_{i} \in \{0,1\} \text{ for } |\xi^i| \text{ odd }\},
    \end{equation}
    which ensures that each summand of the series has the same degree $k$. The graded vector space structure of $\cifty_{(m_j)}(U)$ is defined component-wise, and multiplication is introduced, as one might expect, by
    \begin{equation}
         ( \sum_{{\bld{r}} \in \N_{|f|}^{\tilde{m}}} f_{\bld{r}}\, \xi^{\bld{r}} ) \cdot ( \sum_{{\bld{q}} \in \N_{|g|}^{\tilde{m}}} g_{\bld{q}}\, \xi^{\bld{q}} ) =  \sum_{{\bld{p}} \in \N_{|f| + |g|}^{\tilde{m}}} \left(f\cdot g\right)_{\bld{p}}\, \xi^{\bld{p}} ,
    \end{equation}
    where for any $\bld{p} \in \N^{\tilde{m}}_{|f| + |g|}$ there is
    \begin{equation}
        \left(f \cdot g\right)_{\bld{p}} = \sum_{\substack{\bld{r} \in \N^{\tilde{m}}_{|f|} \\ \bld{r}\leq \bld{p}}} \, \epsilon^{\bld{r}, \bld{p} - \bld{r}} \,  f_{\bld{r}} \, g_{\bld{p}-\bld{r}}.
    \end{equation}
    Here we write $\bld{r} \leq \bld{p}$ if and only if $r_\mu \leq p_\mu$ for every $\mu \in \{1, \dotsc, \tilde{m}\}$, and $\epsilon^{\bld{r}, \bld{p} - \bld{r}} \in \{-1, 1\}$ is the sign obtained by rearranging $\xi^{\bld{r}}\xi^{\bld{p} - \bld{r}}$ into $\xi^{\bld{p}}$ according to the graded commutativity rules. For finite series the multiplication reduces to ordinary multiplication of polynomials in $\xi$ with smooth function coefficients, subject to the graded commutativity. Restrictions on the presheaf $\cifty_{(m_j)}$ are given simply by
    \begin{equation}
        (\sum_{\mathbf{p} \in \N^{\tilde{m}}_{|f|}}f_{\mathbf{p}}\xi^{\mathbf{p}}\rest{)}{V} = \sum_{\mathbf{p} \in \N^{\tilde{m}}_{|f|}}\rest{(f_{\mathbf{p}})}{V}\xi^{\mathbf{p}}.
    \end{equation}
    It can be shown that $\cifty_{(m_j)}$ is indeed a sheaf valued in $\gcas$. The pair $(\R^{m_0}, \cifty_{(m_j)})$ is a graded manifold of graded dimension $(m_j) \equiv (n_{-j})$, corresponding to the graded vector space $\R^{(n_j)}$. We sometimes denote this graded manifold as $\gR^{(n_{j})}$. 
For any $U \in \Op(\R^{m_0})$ we may consider the pair $(U, \cifty_{(m_j)}|_{U}) =: U_{(m_j)}$, which is also a graded manifold of graded dimension $(m_j)_{j \in \Z}$ called a \textbf{graded domain}. This is the local model for all graded manifolds. 
\end{example}
\begin{definition}[Graded Manifold]
A \textbf{graded manifold} is a pair $\calM := (M, \cifty_\calM)$, where $M$ is a second-countable Hausdorff topological space and $\cifty_\calM$ is a sheaf on $M$ valued in $\gcas$ for which there exists a sequence of non-negative integers $(m_j)_{j \in \Z}$, $\sum_{j \in \Z} m_j =: m < + \infty$ and an open cover $\{U_\alpha\}_{\alpha \in I}$ of $M$, such that for every $\alpha \in I$,
\begin{enumerate}
\item There exists a homeomorphism $\uline{\phi_{\alpha}} : U_\alpha \to \hat{U}_{\alpha}$ for some open subset $\hat{U}_{\alpha} \subseteq \R^{m_0}$.
\item There exists a natural isomorphism $\phi_\alpha^\ast : \rest{\cifty_{(m_j)}}{\hat{U}_\alpha} \to (\uline{\phi_\alpha})_\ast \left( \rest{\cifty_\calM}{U_\alpha} \right)$. The codomain here is the pushforward sheaf $\left[(\uline{\phi_\alpha})_\ast \left( \rest{\cifty_\calM}{U_\alpha} \right)\right](\hat{V}) := \cifty_\calM (\uline{\phi_\alpha}^{-1}(\hat{V}))$,  for all $\hat{V} \in \Op(\hat{U}_\alpha)$.
\end{enumerate}
We denote $\phi_\alpha := (\uline{\phi_\alpha}, \phi_\alpha^\ast)$ and say that $\{(U_\alpha, \phi_\alpha)\}_{\alpha \in I}$ is an \textbf{atlas} for a graded manifold $\calM$. Each pair $(U_\alpha, \phi_\alpha)$ is a local \textbf{chart} for $\calM$ and we call $(m_j)$ the graded dimension, or just the \textbf{dimension}, of $\calM$. It can be shown that $\{(U_\alpha, \uline{\phi_\alpha})\}_{\alpha \in I}$ is automatically a smooth atlas for $M$, making $M$ an $m_0$-dimensional smooth manifold called the \textbf{underlying smooth manifold} of $\calM$. The sheaf $\cifty_\calM$ is called the \textbf{structure sheaf} or the \textbf{sheaf of graded functions}. Locally, we have the isomorphism $\phi_\alpha^\ast : \cifty_\calM(U_\alpha) \cong \cifty_{(m_j)}(\hat{U}_\alpha)$, where all elements have the form (\ref{eq1}). In particular, we often abuse notation and write $\phi_\alpha^\ast(\xi^\mu) =: \xi^\mu$, $\phi_\alpha^\ast(x^i) =: x^i$ and call these the \textbf{coordinates} for $\calM$ on $U_\alpha$ introduced by the map $\phi_\alpha$. Sometimes it is advantageous to denote all the coordinates, both degree zero coordinates $x^i$ and the non-zero degree coordinates $\xi^\mu$, together as $\{x^J\}_{J = 1}^m$. 

A morphism $\varphi : \calM \to \calN$ between two graded manifolds $\calM = (M, \cifty_\calM)$ and $\calN = (N, \cifty_\calN)$ is called a \textbf{graded smooth map} and is defined as a pair $\varphi = (\uline{\varphi}, \varphi^\ast)$ where $\uline{\varphi} : M \to N$ is a smooth map and $\varphi^\ast : \cifty_\calN \to \uline{\varphi}_\ast \cifty_\calM$ is morphism of sheaves valued in $\gcas$. In addition, we require that for every $m \in M$ the induced map between stalks $[f]_{\uline{\varphi}(m)} \mapsto [\varphi^\ast(f)]_m$ is a local graded ring morphism, i.e. that it preserves the maximal ideal. For this to make sense, it should be noted that any graded manifold is a graded locally ringed space. This last condition is always satisfied when $\phi^\ast$ is an isomorphism; in particular, local charts of a graded manifold are graded smooth maps. The category of graded manifolds is denoted as $\gman$. The isomorphisms in $\gman$ are also called graded diffeomorphisms.
\end{definition}

\begin{remark} \label{rem_connected}
    Recall that a graded manifold $\calM$ is called \textbf{connected}, iff its underlying manifold $M$ is connected. Here GVBs are defined to have a constant graded rank. To establish the equivalence of categories described in this paper, we either have to allow for a different graded rank on each connected component of $M$, or restrict ourselves to connected manifolds. Since the former case can be easily deduced from the latter, we follow the easier path. We will thus henceforth assume that \textit{all graded manifolds involved in this paper are connected}.
\end{remark}

\section{Sheaf Approach}
\subsection{Definitions}\label{section_sheaf}

The first definition of graded vector bundles approaches the problem from the ``sheaves of sections" direction \cite{Vysoky:2022gm}. If $\calM$ is a graded manifold, then by a \textbf{sheaf of $\cifty_\calM$-modules} we mean any sheaf $\scrS$ on $M$ valued in $\gvec$, such that
\begin{enumerate}
\item $\scrS(U)$ is a $\cifty_\calM(U)$-module, i.e. there is a graded bilinear map of degree zero $(f,s) \mapsto f\cdot s$ such that $(fg)\cdot s = f\cdot (g \cdot s)$ and $1 \cdot s = s$, for all $f,g \in \cifty_\calM (U)$ and $s \in \scrS(U)$.
\item Restriction on the sheaf $\scrS$ are compatible with those on $\cifty_\calM$, i.e. $\rest{(f \cdot s)}{V} = \rest{f}{V} \cdot \rest{s}{V}$, for all $f \in \cifty_\calM (U)$, $s \in \scrS(U)$ and $V \in \Op(U)$.
\end{enumerate}
Note that we can additionally set $ s \cdot f := (-1)^{|f||s|} f \cdot s$ so that we need not distinguish between left and right modules. Let $\scrP$ be another sheaf of $\cifty_\calM$-modules and $U \in \Op(M)$. By a graded $\cifty_\calM(U)$-linear map $\psi : \scrS(U) \to \scrP(U)$ we mean a graded linear map satisfying $\psi(f \cdot s) = (-1)^{|\psi||f|} f \cdot \psi(s)$ for any $s \in \scrS(U)$ and $f \in \cifty_\calM(U)$. A sheaf morphism $F : \scrS \to \scrP$ is called a morphism of sheaves of $\cifty_\calM$-modules if $F_U$ is $\cifty_\calM(U)$-linear map of degree zero for every $U \in \Op(M)$. 

A collection of ``global" elements $s_1, \dotsc, s_r \in \scrS(M)$ is called a finite \textbf{frame}\footnote{See the more general Definition \ref{def_frame}} for $\scrS$ if for any $s \in \scrS(M)$ there exist unique $f^1, \dotsc, f^r \in \cifty_\calM(M)$, such that
\begin{equation}
    s = \sum_{i = 1}^r f^i\cdot s_i.
\end{equation}
The existence of partition of unity for graded manifolds ensures that if $s_1, \dotsc, s_r$ is a frame for $\scrS$, then $\rest{s_1}{U}, \dotsc, \rest{s_r}{U}$ is a frame for $\scrS|_{U}$ for any open $U \subseteq M$. A sheaf of $\cifty_\calM$-modules $\scrS$ for which there exists a finite frame is called \textbf{freely and finitely generated}. Note that if $s_1, \dotsc, s_r$ is a frame for $\scrS$, then the sequence $(r_j)_{j \in \Z}$ defined as $r_j := \# \{i \in \Z \ : \ |s_i| = j\}$ is the same for any frame, see Proposition $\ref{prop_invariantRankProperty}$. The sequence $(r_j)$ is called the \textbf{graded rank} of $\scrS$.

If there exist an open cover $\{U_\alpha\}_{\alpha \in I}$ of $M$ such that $\rest{\scrS}{U_\alpha}$ is a freely and finitely generated sheaf of $\rest{\cifty_\calM}{U_\alpha}$-modules for every $\alpha \in I$, then $\scrS$ is called \textbf{locally freely and finitely generated}. Each frame for $\rest{\scrS}{U_\alpha}$ is then called a \textbf{local frame} for $\scrS$ on $U_\alpha$. If all $\rest{\scrS}{U_\alpha}$ have the same graded rank $(r_j)$, we say that $\scrS$ is of constant graded rank and this is automatically satisfied for connected $M$.

Now we must briefly discuss dual sheaves. Turns out that duals of sheaves $\cifty_\calM$-modules are quite simple: if $\scrS$ is a sheaf of $\cifty_\calM$-modules, then $\scrS^\ast(U)$ comprises graded $\cifty_\calM(U)$-linear maps from $\scrS(U)$ to $\cifty_\calM(U)$ for every $U \in \Op(M)$. The dual $\scrS^\ast$ is automatically a sheaf of $\cifty_\calM$-modules: to obtain restrictions on $\scrS^\ast$, one must make use of partition of unity for graded manifolds and the $\cifty_\calM$-module structure on $\scrS^\ast$ is obtained, as expected, by $(f \cdot \psi)(s) := f \cdot \psi(s)$. It is important to note, that if $\scrS$ is locally freely and finitely generated, then so is $\scrS^\ast$. In fact, if $s_1, \dotsc, s_r$ is a frame for $\scrS$ over $U$, then one can define the dual frame $s^1, \dotsc, s^r$ for $\scrS^\ast$ over $U$ simply by $s^i(s_j) := \delta\indices{^i_j}$ and extending through $\cifty_\calM(U)$-linearity. We see that $|s^i| = -|s_i|$ and so if $(r_j)_{j \in \Z}$ is the graded rank of $\scrS$, then $(r_{-j})_{j \in \Z}$ is the graded rank of $\scrS^\ast$. We may now give the definition of the category of GVBs defined as sheaves of modules.

\begin{definition}[Category $\svbun$] \label{def_svbun}
\begin{itemize}
    \item Let $\calM$ be a graded manifold. A locally freely and finitely generated sheaf of $\cifty_\calM$-modules of constant graded rank is called a \textbf{graded vector bundle} over $\calM$.
    \item Let $\scrS, \ \scrP$ be two GVBs over graded manifolds $\calM$ and $\calN$, respectively. A \textbf{morphism} from $\scrS$ to $\scrP$ is defined as a pair $(\varphi, \Lambda)$ where $\varphi : \calM \to \calN$ is a graded smooth map and $\Lambda : \scrP^\ast \to \varphi_\ast(\scrS^\ast)$ is a morphism of sheaves of $\cifty_\calN$-modules. The $\cifty_\calN(U)$-module structure on $[\varphi_\ast (\scrS^\ast)](U)\equiv \scrS^\ast(\uline{\varphi}^{-1}(U))$ is introduced by $h \cdot \alpha := \varphi^\ast(h)\, \alpha$, for any $\alpha \in \scrS^\ast(\uline{\varphi}^{-1}(U))$ and $h \in \cifty_\calN(U)$. The resulting category is denoted as as $\svbun$.
    \item If $\scrS$ is freely and finitely generated, we call it a \textbf{trivial GVB}.
\end{itemize}
\end{definition}

\begin{example}[Graded Tangent Bundle]\label{example_VFs}
    Vector fields on graded manifold $\calM$ can be defined as derivations of graded functions: a graded linear map $X : \cifty_\calM(U) \to \cifty_\calM(U)$ is a vector field on $\calM$ over $U \in \Op(M)$, if it satisfies the graded Leibniz rule 
\begin{equation}
    X(fg) = X(f) \, g + (-1)^{|X||f|}f \, X(g)
\end{equation}
for any $f,g \in \cifty_\calM(U)$. We denote the sheaf of vector fields on $\calM$ by $\scrX_\calM$. Let $(m_j)_{j \in \Z}$ be the graded dimension of $\calM$ and consider coordinates $\{\xi^\mu\}_{\mu = 1}^{\tilde{m}}$, $\{x^i\}_{i = 1}^{m_0}$  for $\calM$ on some $U \subseteq M$. One can define the \textbf{coordinate vector fields} $\dd{}{\xi^{\mu}}, \ \dd{}{x^i} \in \scrX_\calM(U)$ of degrees $|\dd{}{\xi^\mu}| = -|\xi^\mu|$ and $ |\dd{}{x^i}| = 0$, by requiring 
    \begin{equation}
        \dd{}{\xi^{\mu}} \xi^\nu = \delta\indices{_\mu^\nu}, \quad \dd{}{\xi^{\mu}} x^i = 0, \quad \dd{}{x^{i}}x^j = \delta\indices{_i^j}, \quad \dd{}{x^i}\xi^\mu = 0,
    \end{equation}
    and imposing the graded Leibniz rule. It can be shown that coordinate vector fields form a local frame for $\scrX_\calM$, making $\scrX_\calM$ into a graded vector bundle over $\calM$.
\end{example}

Let $\scrS \in \svbun$ be a GVB over $\calM$. For each $m \in M$, there are two associated graded vector spaces. First, one has the stalk $\scrS_{m}$ of the sheaf $\scrS$. It has a natural structure of a $\cifty_{\calM,m}$-module, where $\cifty_{\calM,m}$ is the stalk of the sheaf $\cifty_{\calM}$ at $m$. One writes $\Op_{m}(M)$ for open subsets of $M$ containing the point $m$. For each $U \in \Op_{m}(M)$ and $s \in \scrS(U)$, we write $[s]_{m} \in \scrS_{m}$ for its germ. Note that for sheaves of $\cifty_{\calM}$-modules, the ``germ map'' $\scrS(U) \rightarrow \scrS_{m}$ is surjective for any $U \in \Op_{m}(M)$, that is we can write any element of $\scrS_{m}$ as a germ of some section of $\scrS$ over $U$ for any given open neighborhood of $m$. This uses a partition of unity argument, see Corollary 3.44 in \cite{Vysoky:2022gm}.

Next, there is a \textbf{fiber} $\scrS_{(m)}$ of $\scrS$ at $m$. It can be constructed in several equivalent ways, see e.g Example 5.17 in \cite{Vysoky:2022gm}. For our purposes in this paper, we define it as 
\begin{equation} \label{eq_fiberdefinition}
    \scrS_{(m)} = \scrS_{m} / (\frJ_{m} \cdot \scrS_{m}),
\end{equation}
where the ideal $\frJ_{m} = \{ [f]_{m} \in \cifty_{\calM,m} \mid f(m) = 0 \}$ is the Jacobson radical of the stalk $\cifty_{\calM,m}$. By $f(m)$ for a graded function $f \in \cifty_\calM(U)$ we mean $f(m) := \uline{f}(m)$ where $\uline{f} \in \cifty_M(U)$ is the smooth function called the body of $f$, see \cite[Example 2.19.]{Vysoky:2022gm}. In particular, $f(m) = 0$ whenever $|f| \neq 0$.

For each $U \in \Op_{m}(M)$, we can compose the canonical map $\scrS(U) \rightarrow \scrS_{m}$ with the quotient map $\scrS_{m} \rightarrow \scrS_{(m)}$ to obtain a surjective graded linear map which assigns to each $s \in \scrS(U)$ its \textbf{value $s|_{m} \in \scrS_{(m)}$ at $m$}. It follows that $\scrS_{(m)}$ is a finite-dimensional graded vector space: if $s_{1}, \dots, s_{r} \in \scrS(U)$ is a local frame for $\scrS$ over $U \in \Op_{m}(M)$, then $s_{1}|_{m}, \dots, s_{r}|_{m}$ forms a basis for $\scrS_{(m)}$. In particular, its (graded) dimension is equal to the graded rank of $\scrS$.

It will later prove to be useful that the fiber can be equivalently described as a certain quotient of the $\cifty_{\calM}(M)$-module of global sections. 
\begin{lemma} \label{lem_fiber}
    There is a canonical isomorphism of $\cifty_{\calM}(M)$-modules
    \begin{equation}
        \scrS_{(m)} \cong \scrS(M) / (\calJ_{m} \cdot \scrS(M)),
    \end{equation}
    where $\calJ_{m} = \{ f \in \cifty_{\calM}(M) \; | \; f(m) = 0 \}$ is the ideal of functions vanishing at $m$.
\end{lemma}
\begin{proof}
    The action of $\cifty_{\calM}(M)$ on $\scrS_{(m)}$ is defined to satisfy
    \begin{equation}
        f \cdot s|_{m} = f(m) s|_{m},
    \end{equation}
    for every $s \in \scrS(M)$ and $f \in \cifty_{\calM}(M)$. Let $\natural_{m}: \scrS(M) \rightarrow \scrS(M) / (\calJ_{m} \cdot \scrS(M))$ be the canonical quotient map. The sought isomorphism $\Psi$ is defined by requiring
    \begin{equation}
        \Psi(s|_{m}) = \natural_{m}(s),
    \end{equation}
    for all $s \in \scrS(M)$. One has to check that it is well-defined. Suppose $s|_{m} = 0$, that is $[s]_{m} \in \frJ_{m} \cdot \scrS_{m}$. Since elements of $\frJ_{m}$ can be written as $[f]_{m}$ for $f \in \calJ_{m}$, it follows that $[s]_{m} = [f^{k} \cdot s_{k}]_{m}$, for some collection of $f^{k} \in \calJ_{m}$ and $s_{k} \in \scrS(M)$. There is thus some $U \in \Op_{m}(M)$, such that $s|_{U} = f^{k}|_{U} \cdot s_{k}|_{U}$. One can choose a bump function $\eta \in \cifty_{\calM}(M)$ satisfying $\supp(\eta) \subseteq  U$ and $\eta(m) = 1$ and write
    \begin{equation}
        s = \eta \cdot s + (1 - \eta) \cdot s = \eta \cdot (f^{k}|_{U} \cdot s_{k}|_{U}) + (1 - \eta) \cdot s = (\eta \cdot f^{k}) \cdot s_{k} + (1 - \eta) \cdot s,
    \end{equation}
    where we have used the usual manipulations with bump functions in the process. Since both $\eta \cdot f^{k} \in \calJ_{m}$ and $1 - \eta \in \calJ_{m}$, we have just proved that $s \in \calJ_{m} \cdot \scrS(M)$ and thus $\Psi(s|_{m}) = \natural_{m}(s) = 0$. 

    It is easy to see that $\Psi$ is $\cifty_{\calM}(M)$-linear and surjective. Finally, if $\natural_{m}(s) = 0$, one has $s \in \calJ_{m} \cdot \scrS(M)$ and thus $[s]_{m} \in \frJ_{m} \cdot \scrS_{m}$. Hence $s|_{m} = 0$ and we see that $\Psi$ is injective.
\end{proof}
\subsection{On Morphisms}
In this subsection, let us discuss some properties of morphisms of graded vector bundles. First, let us examine general morphisms introduced in Definition \ref{def_svbun}. 
\begin{proposition} \label{prop_GVBmorphisms}
    Let $\scrS, \scrP \in \svbun$ be two GVBs over graded manifolds $\calM$ and $\calN$, respectively. Then the following data are equivalent:
    \begin{enumerate}
        \item A morphism $(\varphi,\Lambda): \scrS \rightarrow \scrP$;
        \item A pair $(\varphi, \Lambda_{0})$, where $\varphi: \calM \rightarrow \calN$ is a graded smooth map and $\Lambda_{0}: \scrP^{\ast}(N) \rightarrow \scrS^{\ast}(M)$ is a degree zero graded linear map satisfying
        \begin{equation}
            \Lambda_{0}(g \cdot \psi) = \varphi^{\ast}(g) \cdot \Lambda_{0}(\psi). 
        \end{equation}
    \end{enumerate}
We will usually write just $\Lambda_{0} = \Lambda$ when there is no risk of confusion.
\end{proposition}
\begin{proof}
The relation of $\Lambda$ to $\Lambda_{0}$ is simple, $\Lambda_{0}$ is just the component $\Lambda_{N}$ of the natural transformation $\Lambda$, that is the corresponding $\cifty_{\calN}(N)$-linear morphism, where the graded $\cifty_{\calN}(N)$-module structure on $\scrS(M)$ is discussed in Definition \ref{def_svbun}. Therefore, the implication \textit{1.} $\Rightarrow$ \textit{2.} is obvious. 

The proof of \textit{2.} $\Rightarrow$ \textit{1.} is just a simple modification of the proof of Proposition 1.7 in \cite{vysoky2024graded}. Indeed, every $\cifty_{\calN}(N)$-linear map $\Lambda_{0}$ is local, that is if $\psi|_{V} = 0$ for some $V \in \Op(N)$, then $\Lambda_{0}(\psi)|_{V} = 0$. This property allows one to restrict it to a unique $\cifty_{\calN}(V)$-module morphism $\Lambda_{0}|_{V}: \scrP^{\ast}(V) \rightarrow (\varphi_{\ast}(\scrS^{\ast}))(V) \equiv \scrS^{\ast}(\ul{\varphi}^{-1}(V))$ satisfying $\Lambda_{0}|_{V}(\psi|_{V}) = \Lambda_{0}(\psi)|_{V}$ for all $\psi \in \scrP^{\ast}(N)$. It is easy to see that setting $\Lambda_{V} := \Lambda_{0}|_{V}$ for each $V \in \Op(N)$ defines a unique morphism of sheaves of $\cifty_{\calN}$-modules $\Lambda: \scrP^{\ast} \rightarrow \varphi_{\ast}(\scrS^{\ast})$ satisfying $\Lambda_{N} = \Lambda_{0}$. 
\end{proof}

This proposition shows that when dealing with GVB morphisms, it is sufficient to work with maps of modules of global sections of the respective dual sheaves. If the underlying graded manifolds coincide and $\varphi$ is just the identity, the situation is even simpler.
\begin{proposition} \label{prop_GVBsoveridentity}
    Let $\scrS, \scrP \in \svbun$ over the same graded manifold $\calM$. Then the following data are equivalent:
    \begin{enumerate}
        \item A morphism $(\1_{\calM}, \Lambda): \scrS \rightarrow \scrP$;
        \item A morphism $F: \scrS \rightarrow \scrP$ of sheaves of $\cifty_{\calM}$-modules;
        \item A morphism $F_{0}: \scrS(M) \rightarrow \scrP(M)$ of $\cifty_{\calM}(M)$-modules.
    \end{enumerate}
We usually identify all three of those and write just $F: \scrS \rightarrow \scrP$. 
\end{proposition}
\begin{proof}
The correspondence of $F$ and $\Lambda$ is obvious, they are just transpose to each other, $\Lambda = F^{T}$. Note that it is important that GVBs are reflexive, that is $(\scrS^{\ast})^{\ast} \cong \scrS$. The proof of the equivalence \textit{2.} $\Leftrightarrow$ \textit{3.} is a line-to-line analogue of the one in the previous proposition. 
\end{proof}

In the rest of this section, we will consider only GVB morphisms over the identity between two GVBs $\scrS$ and $\scrP$ over the same graded manifold $\calM$. Let $F: \scrS \rightarrow \scrP$ be a GVB morphism. For each $m \in M$ there is an induced map $F_{m}: \scrS_{m} \rightarrow \scrP_{m}$ of stalks. Since it maps a submodule $\frJ_{m} \cdot \scrS_{m}$ into $\frJ_{m} \cdot \scrP_{m}$, it induces a graded linear map $F_{(m)}: \scrS_{(m)} \rightarrow \scrP_{(m)}$ of the respective fibers. For any $U \in \Op_{m}(M)$ and $s \in \scrS(U)$, one has 
\begin{equation} \label{eq_Fmonvalueofsection}
    F_{(m)}(s|_{m}) = F_{U}(s)|_{m}.
\end{equation}
For GVBs, injectivity and surjectivity of morphisms is easy to handle. 
\begin{proposition} \label{eq_GVBinjectiveequivalent}
    Let $F: \scrS \rightarrow \scrP$ be a GVB morphism. Then the following properties are equivalent:
    \begin{enumerate}
        \item $F_{U}$ is injective for every $U \in \Op(M)$;
        \item There is an open cover $\{ U_{\alpha} \}_{\alpha \in I}$ of $M$, such that $F_{U_{\alpha}}$ is injective for all $\alpha \in I$;
        \item $F_{0}: \scrS(M) \rightarrow \scrP(M)$ is injective;
        \item $F_{m}: \scrS_{m} \rightarrow \scrP_{m}$ is injective for every $m \in M$.
    \end{enumerate}
    We will simply say that $F$ is injective.
\end{proposition}
\begin{proof}
The equivalence of these statements is valid for any sheaf of graded $\cifty_{\calM}$-modules. The proof of \textit{1} $\Rightarrow$ \textit{2} is trivial. Next, assume that \textit{2} holds. Let $F_{0}(s) = 0$ for some $s \in \scrS(M)$. But then
\begin{equation}
    0 = F_{0}(s)|_{U_{\alpha}} = F_{U_{\alpha}}(s|_{U_{\alpha}}),
\end{equation}
for all $\alpha \in I$. This forces $s|_{U_{\alpha}} = 0$ and since $\scrS$ is a sheaf, this implies $s = 0$. Hence $F_{0}$ is injective and \textit{3} holds. Next, let us assume that \textit{3} holds. Let $U \in \Op(M)$ be arbitrary, and suppose there is $s \in \scrS(U)$ such that $F_{U}(s) = 0$. For each $m \in U$, find an open subset $V \in \Op_{m}(U)$ satisfying $\ol{V} \subseteq U$. Let $\lambda \in \cifty_{\calM}(M)$ be a bump function satisfying $\supp(\lambda) \subseteq U$ and $\lambda|_{V} = 1$. Let $\ol{s} := \lambda \cdot s \in \scrS(M)$. This is a section defined by requiring
\begin{equation}
    \ol{s}|_{U} := \lambda|_{U} \cdot s, \; \; \ol{s}|_{\supp(\lambda)^{c}} := 0,
\end{equation}
for an open cover $\{ U, \supp(\lambda)^{c} \}$ of $M$. We claim that $F_{0}(\ol{s}) = 0$. This follows from 
\begin{align}
    F_{0}(\ol{s})|_{U} = & \ F_{U}( \lambda|_{U} \cdot s) = \lambda_{U} \cdot F_{U}(s) = 0, \\
    F_{0}(\ol{s})|_{\supp(\lambda)^{c}} = & \ F_{\supp(\lambda)^{c}}( \ol{s}|_{\supp(\lambda)^{c}}) = 0.
\end{align}
Since $F_{0}$ is assumed injective, this forces $\ol{s} = 0$, and thus $s|_{V} = \ol{s}|_{V} = 0$. Since we can find such $V$ for any $m \in U$, this proves that $s = 0$ and $F_{U}$ is proved injective. The implication \textit{3} $\Rightarrow$ \textit{1} is now proved. The proof of \textit{1} $\Leftrightarrow$ \textit{4} is a standard exercise in any sheaf theory book.
\end{proof}
\begin{proposition} \label{eq_GVBsurjectiveequivalent}
    Let $F: \scrS \rightarrow \scrP$ be a GVB morphism. Then the following properties are equivalent:
    \begin{enumerate}
        \item $F_{U}$ is surjective for every $U \in \Op(M)$;
        \item There is an open cover $\{ U_{\alpha} \}_{\alpha \in I}$ of $M$, such that $F_{U_{\alpha}}$ is surjective for all $\alpha \in I$;
        \item $F_{0}: \scrS(M) \rightarrow \scrP(M)$ is surjective;
        \item $F_{m}: \scrS_{m} \rightarrow \scrP_{m}$ is surjective for every $m \in M$.
    \end{enumerate}
    We will simply say that $F$ is surjective.
\end{proposition}
\begin{proof}
    Again, the proof of \textit{1} $\Rightarrow$ \textit{2} is trivial. Let us assume that \textit{2} holds. Let $t \in \scrP(M)$ be an arbitrary section. By assumption, for each $\alpha \in I$, there is a section $s_{\alpha} \in \scrS(U_{\alpha})$, such that $F_{U_{\alpha}}(s_{\alpha}) = t|_{U_{\alpha}}$. Let $\{ \rho_{\alpha} \}_{\alpha \in I}$ be a partition of unity subordinate to $\{ U_{\alpha} \}_{\alpha \in I}$, and let $s := \sum_{\alpha \in I} \rho_{\alpha} \cdot s_{\alpha}$. We claim that $F_{0}(s) = t$. To see this, note that
    \begin{equation}
        F_{0}(s) = F_{0}( \sum_{\alpha \in I} \rho_{\alpha} \cdot s_{\alpha}) = \sum_{\alpha \in I} \rho_{\alpha} \cdot F_{U_{\alpha}}(s_{\alpha}) = \sum_{\alpha \in I} \rho_{\alpha} \cdot t|_{U_{\alpha}} = t. 
    \end{equation}
    Manipulations with the partitions unity may seem a little too courageous. However, by carefully examining the definitions, they are correct, see Proposition 3.41 in \cite{Vysoky:2022gm}. Hence $F_{0}$ is surjective and the implication \textit{2} $\Rightarrow$ \textit{3} is proved. Let us assume that \textit{3} holds. Let $U \in \Op(M)$ and $t \in \scrP(U)$. For each $m \in U$, one can find $V_{(m)} \in \Op_{m}(U)$ and $t_{(m)} \in \scrP(M)$, such that $t|_{V_{(m)}} = t_{(m)}|_{V_{(m)}}$. This is Proposition 5.5 in \cite{Vysoky:2022gm}. Then use the assumption to find $s_{(m)} \in \scrS(M)$ satisfying $F_{0}(s_{(m)}) = t_{(m)}$. Let $\{ \rho_{(m)} \}_{m \in U}$ be the partition of unity subordinate to the open cover $\{ V_{(m)} \}_{m \in U}$ of $U$. Define
    \begin{equation}
        s := \sum_{m \in U} \rho_{(m)} \cdot s_{(m)}|_{U} \in \scrS(U).
    \end{equation}
    We claim that $F_{U}(s) = t$. Indeed, one finds
    \begin{equation}
    \begin{split}
        F_{U}(s) = & \ \sum_{m \in U} \rho_{(m)} \cdot F_{U}(s_{(m)}|_{U}) = \sum_{m \in U} \rho_{(m)} \cdot t_{(m)}|_{U} = \sum_{m \in U} \rho_{(m)} \cdot t_{(m)}|_{V_{(m)}} \\
        = & \ \sum_{m \in U} \rho_{(m)} t|_{V_{(m)}} = t.
    \end{split}
    \end{equation}
    We conclude that $F_{U}$ is surjective and \textit{3} $\Rightarrow$ \textit{1} is proved. The proof of \textit{1} $\Leftrightarrow$ \textit{4} is again standard, except one has to use the nontrivial fact that for any GVB morphism $F: \scrS \rightarrow \scrP$, the image presheaf $\im(F)$ is a sheaf. This follows from the fact that $\im(F) \cong \scrP / \ker(F)$ and quotients by sheaves of $\cifty_{\calM}$-submodules are sheaves. This is proved similarly to Proposition 3.46 in \cite{Vysoky:2022gm}.
\end{proof}
Both of the above two statements hold for morphisms of any sheaves of $\cifty_{\calM}$-modules. The following statement is specific for GVBs.
\begin{proposition} \label{tvrz_GVBsurjectivesplits}
Let $F: \scrS \rightarrow \scrP$ be a surjective GVB morphism. Then there exists its right inverse, that is a GVB morphism $G: \scrP \rightarrow \scrS$ satisfying 
\begin{equation}
    F \circ G = \1_{\scrP}.
\end{equation}
In other words, every short exact sequence in $\svbun$ splits. 
\end{proposition}
\begin{proof}
Let $\{ U_{\alpha} \}_{\alpha \in I}$ be an open cover of $M$, such that for every $\alpha \in I$, there is a local frame $t_{1}^{(\alpha)}, \dots, t_{q}^{(\alpha)}$ for $\scrP$ over $U_{\alpha}$. Since $F_{U_{\alpha}}: \scrS(U_{\alpha}) \rightarrow \scrP(U_{\alpha})$ is surjective, there is $s_{j}^{(\alpha)} \in \scrS(U_{\alpha})$ for each $\alpha \in I$ and $j \in \{1, \dots, q\}$, such that $F_{U_{\alpha}}(s_{j}^{(\alpha)}) = t_{j}^{(\alpha)}$. For each $\alpha \in I$, one can thus define a $\cifty_{\calM}(U_{\alpha})$-module morphism $G^{(\alpha)}: \scrP(U_{\alpha}) \rightarrow \scrS(U_{\alpha})$ by uniquely extending the formula 
\begin{equation}
    G^{(\alpha)}( t_{j}^{(\alpha)}) := s_{j}^{(\alpha)},
\end{equation}
for all $j \in \{1,\dots,q\}$. Let $\{ \rho_{\alpha} \}_{\alpha \in I}$ be a partition of unity subordinate to $\{ U_{\alpha} \}_{\alpha \in I}$. For each $t \in \scrP(M)$, define $G_{0}: \scrP(M) \rightarrow \scrS(M)$ by 
\begin{equation}
    G_{0}(t) := \sum_{\alpha \in I} \rho_{\alpha} \cdot G^{(\alpha)}(t|_{U_{\alpha}}).
\end{equation}
It is is easy to see that $G_{0}$ is a $\cifty_{\calM}(M)$-module morphism satisfying $F_{0} \circ G_{0} = \1_{\scrP(M)}$. If $G: \scrP \rightarrow \scrS$ is the corresponding morphism of sheaves, the uniqueness of the correspondences in Proposition \ref{prop_GVBsoveridentity} ensures that $F_{U} \circ G_{U} = \1_{\scrP(U)} = (\1_{\scrP})_{U}$, and the conclusion follows.
\end{proof}

Finally, in ordinary geometry, the properties of vector bundle morphisms can be examined fiber-wise, turning the geometrical statement into linear algebra. Similar observations remain valid in the graded geometry. 

\begin{proposition} \label{tvrz_fiberinjsur}
    Let $F: \scrS \rightarrow \scrP$ be a GVB morphism. Let $m \in M$ be arbitrary, and let $F_{(m)}: \scrS_{(m)} \rightarrow \scrP_{(m)}$ be the induced map of fibers.
    \begin{enumerate}
        \item $F_{(m)}$ is injective, iff there exists $U \in \Op_{m}(M)$, such that $F|_{U}$ has a left inverse.
        \item $F_{(m)}$ is surjective, iff there exists $U \in \Op_{m}(M)$, such that $F|_{U}$ has a right inverse. 
    \end{enumerate}
\end{proposition}
We prove this statement in Appendix A. This statement has an immediate corollary.
\begin{corollary} \label{cor_fiberinjsur}
    Let $F: \scrS \rightarrow \scrP$ be a GVB morphism. 
    \begin{enumerate}
        \item $F$ is surjective, iff $F_{(m)}$ is surjective for every $m \in M$;
        \item $F$ is an isomorphism, iff $F_{(m)}$ an isomorphism for every $m \in M$;
        \item $F$ is injective, if $F_{(m)}$ is injective for every $m \in M$. The converse implication is not true. 
    \end{enumerate}
\end{corollary}
\begin{proof}
When $F$ is surjective, it has a right inverse by Proposition \ref{tvrz_GVBsurjectivesplits}. Since the assignment $F \mapsto F_{(m)}$ respects compositions, it follows that $F_{(m)}$ has a right inverse, hence it is surjective for every $m \in M$. Conversely, if it is surjective for every $m \in M$, we can use Proposition \ref{tvrz_fiberinjsur} to find an open cover $\{ U_{\alpha} \}_{\alpha \in I}$ such that $F_{U_{\alpha}}$ is surjective for all $\alpha \in I$. Hence $F$ is surjective by Proposition \ref{eq_GVBsurjectiveequivalent}. 

If $F$ is an isomorphism, the inverse to each $F_{(m)}$ is $(F^{-1})_{(m)}$, so $F_{(m)}$ is an isomorphism for every $m \in M$. The converse statement is proved in the same way as in the surjective case. 

For the injectivity, we can only prove the if part. Proposition \ref{tvrz_fiberinjsur} together with Proposition \ref{eq_GVBinjectiveequivalent} shows that fiber-wise injective $F$ is injective. The converse statement does not hold even for ordinary vector bundles: Let $M = \R$, and let $\cifty_{\R}$ be the usual sheaf of smooth functions on $\R$. This is a sheaf of sections of a trivial vector bundle $\R \times \R$. Let $F_{0}: \cifty_{\R}(\R) \rightarrow \cifty_{\R}(\R)$ be defined as 
\begin{equation}
[F_{0}(f)](x) := x f(x),
\end{equation}
for all $f \in \cifty_{\R}(\R)$ and $x \in \R$. $F_{0}$ is injective, since $F_{0}(f) = 0$ implies $f(x) = 0$ for all $x \in \R \ssm \{0\}$ and thus $f = 0$ by continuity. It is not fiber-wise injective, since for each $x \in \R$, the induced linear map $F_{(x)}: \R \rightarrow \R$ takes the form $F_{(x)}y = x y$ for all $y \in \R$ and thus $F_{(0)} = 0$.
\end{proof}
\subsection{Subbundles} \label{sec_subbundles}
In this subsection, we must discuss some elementary properties of subbundles. Let $\scrS \in \svbun$ be a fixed GVB over a graded manifold $\calM$. We say that a subsheaf $\scrL \subseteq \scrS$ is a \textbf{subbundle of $\scrS$}, if 
\begin{enumerate}
    \item $\scrL(U)$ is a $\cifty_{\calM}(U)$-submodule of $\scrS(U)$ for each $U \in \Op(M)$. We also say that $\scrL$ is a \textbf{sheaf of $\cifty_{\calM}$-submodules}.
    \item There exists a fixed subset $J \subseteq \{1, \dots, r\}$, such that for each $m \in M$, there exists $U \in \Op_{m}(M)$ and a local frame $s_{1}, \dots, s_{r}$ for $\scrS$ over $U$, such that $\{ s_{i} \}_{i \in J}$ is a local frame for $\scrL$ over $U$. We say that $s_{1}, \dots, s_{r}$ is \textbf{adapted to $\scrL$}.
\end{enumerate}
Since $\scrL$ is obviously locally freely and finitely generated of a constant graded rank, it is a GVB on its own. 
\begin{proposition}
    Let $\scrL \subseteq \scrS$ be a subbundle of $\scrS$. Let $I: \scrL \rightarrow
     \scrS$ denote the canonical inclusion. 
     
     Then for each $m \in M$, the induced map $I_{(m)}: \scrL_{(m)} \rightarrow \scrS_{(m)}$ is injective. Consequently, one can canonically identify each fiber $\scrL_{(m)}$ with a subspace of $\scrS_{(m)}$. 
\end{proposition}
\begin{proof}
    Let $m \in M$. There is $U \in \Op_{m}(M)$ and a local frame $s_{1}, \dots, s_{r}$ adapted to $\scrL$. In other words, $\{ s_{i} \}_{i \in J}$ is a local frame for $\scrL$ over $U$, where $J \subseteq \{1,\dots,r\}$ is a fixed subset. It follows that $\{ s_{i}|_{m} \}_{i \in J}$ is a basis for $\scrL_{(m)}$. One has $ I_{(m)}( s_{i}|_{m}) = I_{U}(s_{i})|_{m} = s_{i}|_{m}$ for each $i \in J$. Hence $I_{(m)}$ is injective.
\end{proof}

We will need the following technical statement:
\begin{proposition}[\textbf{Local frame completion}] \label{prop_localframecompl}
    Let $a_{1}, \dots, a_{q}$ be some collection of local sections of $\scrS$ over $U \in \Op(M)$. Suppose that there is $m \in U$, such that $a_{1}|_{m}, \dots a_{q}|_{m}$ are linearly independent vectors. 
    
    Then there exists $V \in \Op_{m}(U)$ and a local frame $s_{1}, \dots, s_{r}$ for $\scrS$ over $V$, such that $s_{i} = a_{i}|_{V} $ for every $i \in \{1,\dots,q\}$.
\end{proposition}
\begin{proof}
    Note that a system $v_{1}, \dots, v_{r}$ of vectors of a graded vector space $V$ is called linearly independent, if its elements of degree $k$ form a linearly independent system of vectors in the ordinary vector space $V_{k}$, for each $k \in \Z$. Be aware that linear combinations of vectors of different degrees are not allowed.
    
    Since this is a local statement, we may assume without the loss of generality that $\scrS$ is freely and finitely generated and $U = M$. This also means that we can assume that $\scrS = \cifty_{\calM}[K]$ for some finite-dimensional graded vector space $K$, that is
    \begin{equation}
        \scrS(U) = \cifty_{\calM}(U) \otimes_{\R} K,
    \end{equation}
    for every $U \in \Op(M)$. $\cifty_{\calM}(U)$-module structure and sheaf restrictions are defined in an obvious way. See Example 2.37 and Definition 2.38 in \cite{Vysoky:2022gm}. For each $m \in M$, there is a canonical identification $\scrS_{(m)} \cong K$. By standard linear algebra, there is thus a basis $\theta_{1}, \dots, \theta_{r}$ for $K$, such that $\theta_{i} = a_{i}|_{m}$ for every $i \in \{1,\dots,q\}$. Now, for each $U \in \Op(M)$, define $F_{U}: \scrS(U) \rightarrow \scrS(U)$ as
    \begin{equation}
        F_{U}(1 \otimes \theta_{i}) := \left\{ \begin{array}{cc} a_{i}|_{U} & \text{ for } i \in \{1,\dots,q\}, \\
        1 \otimes \theta_{i}     & \text{ for } i \in \{q+1, \dots, r \}.
        \end{array} \right.
    \end{equation}
    By extending this by $\cifty_{\calM}(U)$-linearity, we obtain a morphism of $\cifty_{\calM}(U)$-modules natural in $U$, hence a GVB morphism $F: \scrS \rightarrow \scrS$. Moreover, the induced map of fibers $F_{(m)}: K \rightarrow K$ is just the identity. It follows from Proposition \ref{tvrz_fiberinjsur} that there exists $V \in \Op_{m}(M)$, such that $F|_{V}$ is an isomorphism. Finally, letting $s_{i} := F_{V}(1 \otimes \theta_{i})$ for each $i \in \{1, \dots, r\}$ then gives us a local frame $s_{1}, \dots, s_{r}$ for $\scrS$ over $V$ satisfying $s_{i} = a_{i}|_{V}$ for every $i \in \{1, \dots, q\}$.
\end{proof}

\begin{corollary}[\textbf{Completion of independent fiber vectors}] \label{cor_complindepdenent}
    Let $\{ \theta_{i} \}_{i=1}^{q}$ be a collection of linearly independent vectors of the fiber $\scrS_{(m)}$ for any given $m \in M$. 
    
    Then there exists $U \in \Op_{m}(M)$ and a local frame $s_{1}, \dots, s_{r}$ for $\scrS$ over $U$, such that $s_{i}|_{m} = \theta_{i}$ for every $i \in \{1, \dots, q \}$.
\end{corollary}
\begin{proof}
One can always find sections $a_{1}, \dots, a_{q} \in \scrS(M)$ satisfying $a_{i}|_{m} = \theta_{i}$ for each $i \in \{1, \dots, q\}$, see the paragraph under (\ref{eq_fiberdefinition}). By previous proposition, there exists $U \in \Op_{m}(M)$ and a local frame $s_{1}, \dots, s_{r}$ for $\scrS$ over $U$, such that $s_{i} = a_{i}|_{V}$ for every $i \in \{1, \dots, q \}$. In particular, one has $s_{i}|_{m} = a_{i}|_{m} = \theta_{i}$ for every $i \in \{1, \dots, q\}$. 
\end{proof}

We can now use these statements to prove that kernels and images of GVB morphisms are subbundles under some conditions.
\begin{proposition} \label{prop_kerimsubbundles}
Let $F: \scrS \rightarrow \scrP$ be a GVB morphism. Recall that $\ker(F) \subseteq \scrS$ and $\im(F) \subseteq \scrP$ are both subsheaves, see also the proof of Proposition \ref{eq_GVBsurjectiveequivalent}. 
\begin{enumerate}
    \item Let $F_{(m)}$ be injective for all $m \in M$. Then $\im(F)$ is a subbundle of $\scrP$;
    \item Let $F$ be surjective. Then $\ker(F)$ is a subbundle of $\scrS$;
    \item In general, the constant graded rank assumption on $F_{(m)}$ is not sufficient. 
\end{enumerate}
\end{proposition}
\begin{proof}
    Let us prove \textit{1} first. It is easy to see that $\im(F)$ is a sheaf of $\cifty_{\calM}$-submodules of $\scrP$. For each $m \in M$, we must find a local frame for $\scrP$ over $U \in \Op_{m}(M)$ adapted to $\im(F)$. Let $s_{1}, \dots, s_{r}$ be a local frame for $\scrS$ over some $V \in \Op_{m}(M)$, and let $a_{i} := F_{V}(s_{i}) \in \scrP(V)$ for each $i \in \{1,\dots,r\}$. Since $a_{i}|_{m} = F_{(m)}(s|_{i})$, it follows that $a_{1}|_{m}, \dots, a_{r}|_{m}$ are linearly independent. By Proposition \ref{prop_localframecompl}, there exists $U \in \Op_{m}(V)$ and a local frame $t_{1}, \dots, t_{q}$ for $\scrP$ over $U$, such that $t_{i} = a_{i}|_{U} = F_{U}(s_{i}|_{U})$ for all $i \in \{1, \dots, r\}$. $F$ is injective by Corollary \ref{cor_fiberinjsur}. It is easy to see that $t_{1}, \dots, t_{q}$ is adapted to $\im(F)$. 

    \textit{2} can be rephrased to follow from \textit{1}. Indeed, if $F$ is surjective, we can form the following short exact sequence of sheaves of $\cifty_{\calM}$-modules:
    \begin{equation} \label{eq_kerFimFproofSES}
        \begin{tikzcd}[column sep=large]
            0 \arrow{r} & \ker(F) \arrow{r}{J} & \arrow[bend left=20, dashed]{l}{T} \scrS \arrow{r}{F} & \arrow[bend left=20, dashed]{l}{S} \scrP \arrow{r}& 0 
        \end{tikzcd},
    \end{equation}
    where $J$ is just the canonical inclusion. By Proposition \ref{tvrz_GVBsurjectivesplits}, this sequence splits. There is thus a GVB morphism $S: \scrP \rightarrow \scrS$, such that $F \circ S = \1_{\scrP}$. Consequently, there is a unique morphism of sheaves of $\cifty_{\calM}$-modules $T: \scrS \rightarrow \ker(F)$ satisfying $T \circ J = \1_{\ker(F)}$, and $S \circ F + J \circ T = \1_{\scrS}$. It is an easy exercise to prove that the dashed sequence in (\ref{eq_kerFimFproofSES}) is also exact. There is thus an induced isomorphism of sheaves of $\cifty_{\calM}$-modules:
    \begin{equation}
        \ker(F) \cong \scrS / \im(S).
    \end{equation}
    Since $S$ is a right inverse to $F$, it is fiber-wise injective by Proposition \ref{tvrz_fiberinjsur}. By already proved statement, $\im(S)$ is a subbundle of $\scrS$. Now, the quotient of a GVB by a subbundle is a GVB, see Proposition 5.14. in \cite{Vysoky:2022gm}. This shows that $\ker(F)$ is a GVB. Since $J$ has a left inverse, it is also fiber-wise injective. Hence by already proved statement, $\im(J) \equiv \ker(F)$ is a subbundle of $\scrS$. 

    In ordinary geometry, both \textit{1} and \textit{2} hold whenever the rank of linear map $F_{(m)}$ is constant in $m$. In graded geometry, this is no longer true. Let $\calM$ be a graded manifold over $M = \R$ with two global coordinates $(x,\xi)$, such that $|\xi| = 1$. Let $\scrS = \scrP = \scrX_{\calM}$, see Example \ref{example_VFs}. Let us define a GVB morphism $F: \scrX_{\calM} \rightarrow \scrX_{\calM}$ by establishing a morphism $F_{0}$ of modules of global sections:
    \begin{equation}
        F_{0}(\frac{\partial}{\partial x}) := \xi \frac{\partial}{\partial \xi} , \; \; F_{0}( \frac{\partial}{\partial \xi}) := 0.
    \end{equation}
    Now, the induced fiber map $F_{(x)}$ is zero for each $x \in \R$, so its graded rank is definitely constant in $x$. The sheaf $\im(F)$ is non-trivial, but all its sections have a zero value at all points of $\R$. In particular, there cannot be a local frame for $\scrX_{\calM}$ adapted to $\im(F)$. 
\end{proof}

\section{Algebraic Approach}
\subsection{Free and Projective Graded Modules}
In this subsection, let $A \in \gcas$ be fixed. We will further restrict the choice of $A$ in a moment. For a precise definition of graded $A$-modules used here, see \S 1.4 of \cite{Vysoky:2022gm}.
\begin{definition}
    Let $P$ and $Q$ be graded $A$-modules. A graded linear map $\varphi: P \rightarrow Q$ of degree $|\varphi|$ is called \textbf{$A$-linear}, if 
    \begin{equation}
        \varphi(a \cdot p) = (-1)^{|a||\varphi|} a \cdot \varphi(p).
    \end{equation}
    If $|\varphi| = 0$, it is called a \textbf{morphism of graded $A$-modules}. Graded $A$-modules together with morphisms of graded $A$-modules form the \textbf{category of graded $A$-modules}. 
\end{definition}
As in previous sections, we will mostly omit the adjective ``graded'' to clean up the writing. 
\begin{definition}\label{def_frame}
Let $P$ be an $A$-module. We say that $B = \{ b_{\lambda} \}_{\lambda \in I}$ is a \textbf{frame for $P$}, if each element $p \in P$ can be written as $p = a^{\lambda} \cdot b_{\lambda}$ for unique $a^{\lambda} \in A$, where only finitely many of them are non-zero. We say that $P$ is a \textbf{free $A$-module}, if it has a frame.

We say that a free $A$-module is \textbf{finitely generated}, if it has a \textit{finite} frame.
\end{definition}
\begin{remark}
    We choose to use the term ``frame'' instead of the more usual ``basis''. Since all $A$-modules are also graded vector spaces, we do this to avoid the possible confusion.
\end{remark}
\begin{example}
    Let $K \in \gvec$. Set $A[K] = A \otimes_{\R} K$ and let $a \cdot (b \otimes k) := (a \cdot b) \otimes k$. This makes $A[K]$ into  an $A$-module. Let $\{ h_{\lambda} \}_{\lambda \in I}$ be some (possibly infinite) basis of $K$. Then $\{ 1 \otimes h_{\lambda} \}_{\lambda \in I}$ is a frame for $A[K]$. Hence $A[K]$ is a free $A$-module.
\end{example}
It turns out that this is a prototypical example of a free $A$-module.
\begin{proposition} \label{prop_freeisAK}
    An $A$-module $P$ is free, iff it is isomorphic to $A[K]$ for some $K \in \gvec$.
\end{proposition}
\begin{remark}\label{rmk_invariantRankProperty}
    Let $B \subseteq P$ be a frame for $P$. By a graded cardinality of $B$, we mean the sequence $(\# B_{j})_{j \in \Z}$. We say that $A$ has an \textbf{invariant graded rank property}, if for any given free $A$-module $P$, all its frames have the same graded cardinality. One can show that equivalently, $A[K] \cong A[K']$ implies $K \cong K'$. Hence for $A$ with an invariant graded rank property, every free $A$-module is isomorphic to $A[K]$ for $K \in \gvec$ unique up to an isomorphism. In fact, the graded cardinality of any its frame is the same as a (graded) dimension of $K$. Consequently, $P$ is finitely generated, if $P \cong A[K]$ for a finite-dimensional $K$. 
    
    For any free $A$-module $P \cong A[K]$, we can define its \textbf{graded rank} $\grk(P) := \gdim(K) \equiv ( \dim K_{j})_{j \in \Z}$. Note that $P$ is finitely generated, iff it has a finite graded rank. We will henceforth assume that $A$ has an invariant graded rank property
\end{remark}
\begin{proposition}\label{prop_invariantRankProperty}
Let $\calM$ be a graded manifold. Then $\cifty_\calM(M)$ has the invariant graded rank property.
\end{proposition}
\begin{proof}
    Let $A \in \gcas$ and $K \in \gvec$ be arbitrary. Let $J \subseteq A$ be any ideal. Then there is a canonical isomorphism $A[K]/(J \otimes_{\R} K) \cong (A/J)[K]$. 

    Now, any $A$-module isomorphism $\varphi: A[K] \rightarrow A[K']$ satisfies $\varphi(J \otimes_{\R} K) \subseteq J \otimes_{\R} K'$, hence it induces an isomorphism $\hat{\varphi}: (A/J)[K] \rightarrow (A/J)[K']$.

    Let $A = \cifty_\calM(M)$ and for any given $a \in M$, consider the ideal 
    \begin{equation}
        J := \{ f \in \cifty_{\calM}(M) \; | \; f(a) = 0 \}.
    \end{equation}
    Clearly $A/J \cong \R$ and since $\R[K] \cong K$, we find an isomorphism $\hat{\varphi}: K \rightarrow K'$.
\end{proof}

\begin{definition}
Let $P$ be an $A$-module. We say that $S = \{ s_{\lambda} \}_{\lambda \in I} \subseteq P$ is a \textbf{generating set for $P$}, if every element $p \in P$ can be written as $p = a^{\lambda} \cdot s_{\lambda}$ for \textit{some} $a^{\lambda} \in A$, where only finitely many of them are non-zero. 

We say that $P$ is a \textbf{finitely generated $A$-module}, if it has a \textit{finite} generating set.
\end{definition}
\begin{remark} \label{rem_aboutfrees} Let $P$ be an $A$-module. 
\begin{enumerate}
    \item There always exists some generating set for $P$ since one can take $S = P$.
    \item $P$ is free, iff we can choose an $A$-linearly independent generating subset $B$ for $P$.
    \item If $P$ is free, it has a finite generating set $S$, iff it has a finite frame $B$. This means that both notions of ``finitely generated'' coincide for free $A$-modules. Indeed, if $S = \{s_{1}, \dots, s_{n} \}$ is a finite generating set and $B$ is an arbitrary frame, one can always find a finite subset $B' \subseteq B$ which generates $P$. Indeed, every $s_{i}$ can written as a finite $A$-linear combination of some elements of $B$, so one can consider a finite set $B'_{i}$ all elements of $B$ which appear in this combination, and let $B' := \bigcup_{i=1}^{n} B'_{i}$. The converse statement is clear. 
\end{enumerate}
\end{remark}

\begin{definition}
We say that  an $A$-module $P$ is \textbf{projective}, if for any $\psi: P \rightarrow M$ and any \textit{surjective} $\varphi: Q \rightarrow M$, there exists $\ol{\psi}: P \rightarrow Q$ fitting into the commutative diagram
\begin{equation}
    \begin{tikzcd}
    & Q \arrow{d}{\varphi}\\
    P \arrow{r}{\psi} \arrow[dashed]{ur}{\ol{\psi}}& M
    \end{tikzcd}.
\end{equation}
All involved maps are assumed to be $A$-linear. $\ol{\psi}$ is called a \textbf{lift of $\psi$ along $\varphi$}. 
\end{definition}
There are several equivalent ways how to define projective $A$-modules. Let us formulate this as a proposition. We omit its proof, since it is completely analogous to the one for modules over commutative rings. 
\begin{proposition} \label{tvrz_projective}
    Let $P$ be a graded $A$-module. Then the following statements are equivalent:
    \begin{enumerate}
        \item $P$ is projective;
        \item For any surjective $A$-linear map $\varphi: Q \rightarrow P$, there exists an $A$-linear map $\sigma: P \rightarrow Q$ satisfying 
        \begin{equation}
        \varphi \circ \sigma = \1_{P}.
        \end{equation}
        \item $P$ is isomorphic to a direct summand of a free $A$-module.
        \item The functor $\ul{\Hom}(P,-)$ from the category of graded $A$-modules to itself is exact, where $\ul{\Hom}(P,Q)$ denotes the $A$-module of $A$-linear maps from $P$ to $Q$. 
    \end{enumerate}
\end{proposition}
\begin{corollary}
Every free $A$-module is projective. 
\end{corollary}
Let us summarize some convenient properties of projective $A$-modules.
\begin{proposition} \label{prop_projectiveprops}
\begin{enumerate}
    \item Let $P$ be a projective $A$-module. Then it is finitely-generated, iff it is isomorphic to a direct summand of a finitely generated free $A$-module.
    \item For any $A$-module $P$, its \textbf{dual $A$-module} is defined as $P^{\star} := \ul{\Hom}(P,A)$. If $P$ is finitely generated and projective, then so is $P^{\star}$. 
    \item If $P$ is finitely generated and projective, then $P$ is reflexive. In other words, the canonical map $P \mapsto (P^{\star})^{\star}$ is an isomorphism of $A$-modules.
\end{enumerate}
\end{proposition}
\begin{proof}
If $P$ is finitely generated, let $S = \{ s_{1}, \dots, s_{n}\}$ be some its finite generating set. Let $\R(S)$ be the free graded vector space generated by $S$, and let $F := A[\R(S)]$. This is a finitely generated free $A$-module. There is an obvious surjective $A$-linear map $\varphi: F \rightarrow P$. By Proposition \ref{tvrz_projective}, there is an injective $\sigma: P \rightarrow F$ satisfying $\varphi \circ \sigma = \1_{P}$. It follows that $F = \ker \varphi \oplus \sigma(P)$ and we find $P$ isomorphic to a direct summand $\sigma(P)$ of a finitely generated free $A$-module $F$. 

Conversely, if $P$ is isomorphic to a direct summand of a finitely generated free $A$-module $F$, there is a surjective projection $\pi_{P}: F \rightarrow P$. For any finite frame $B$ for $F$, its image $\pi_{P}(B)$ is a finite generating subset for $P$. This proves \textit{1}.

Now, if $P$ is finitely generated and projective, there is a \textit{finitely generated} free $A$-module $F$, such that $F \cong P \oplus Q$ for some $A$-module $Q$, by already proved statement \textit{1}. Since the functor $\ul{\Hom}(\cdot,A)$ preserves direct sums, we have $F^{\star} \cong P^{\star} \oplus Q^{\star}$. It is easy to see that $F^{\star}$ is again finitely generated free. Hence $P^{\star}$ is finitely generated and projective by \textit{1}.

To prove \textit{3}, recall that the canonical $A$-linear map $j_{P}: P \rightarrow (P^{\star})^{\star}$ takes $p \in P$ to an $A$-linear map $[j_{P}(p)](\xi) := (-1)^{|\xi||p|} \xi(p)$, for all $\xi \in P^{\star}$. $j_{P}$ is always injective, bud not necessarily surjective.

Since $P$ is finitely generated and projective, we may assume that there is a finitely generated free graded $A$-module $F$, such that $F = P \oplus Q$. One obtains a diagram
\begin{equation}
    \begin{tikzcd}
    0 \arrow{r} & \arrow{r}{i_{P}} P  \arrow{d}{j_{P}} & F \arrow{r}{\pi_{Q}} \arrow{d}{j_{F}} &  Q \arrow{d}{j_{Q}} \arrow{r} & 0\\
    0 \arrow{r} &  \arrow{r}{i_{P}^{TT}} (P^{\star})^{\star}  & \arrow{r}{\pi_{Q}^{TT}}(F^{\star})^{\star} & (Q^{\star})^{\star} \arrow{r} & 0 
    \end{tikzcd}.
\end{equation}
Both rows are split exact sequences, the lower one is obtained by applying the additive functor $\ul{\Hom}(\cdot,A)$ twice. For example, $i_{P}^{TT}$ is just the ``double transpose'' of the map $i_{P}$. Both squares in this diagram are commutative. Since $F$ is finitely generated and free, the map $j_{F}$ is easily seen to be an isomorphism. By a simple diagram chase together with the injectivity of $j_{Q}$, one can now easily prove that $j_{P}$ is surjective. 
\end{proof}
\begin{remark}
    Note that \textit{2} and \textit{3} are \textit{not true}, if we remove the ``finitely generated'' assumption.
\end{remark}
Let us mention one important feature of projective modules.
\begin{lemma} \label{lem_tensorproductprojective}
    Let $P$ and $Q$ be projective $A$-modules. 
    
    Then their tensor product $P \otimes_{A} Q$ is projective. 
\end{lemma}
\begin{proof}
    Let $R$ be an arbitrary $A$-module. There is a canonical isomorphism
    \begin{equation}
        \ul{\Hom}( P \otimes_{A} Q, R) \cong \ul{\Hom}(P, \ul{\Hom}(Q,R)),
    \end{equation}
    natural in $R$. This observation and Proposition \ref{tvrz_projective} imply that the functor $\ul{\Hom}(P \otimes_{A} Q, -)$ is naturally isomorphic to the composition $\ul{\Hom}(P,-) \circ \ul{\Hom}(Q,-)$ of two exact functors. Hence it is itself exact and $P \otimes_{A} Q$ is projective by Proposition \ref{tvrz_projective}. 
\end{proof}
\subsection{Definition and The Functor}
Let us start by defining the suitable category of \textit{algebraic graded vector bundles}, which we shall for the purposes of this paper denote as $\avbun$. It will contain algebraic GVBs over all graded manifolds. This will require some work on the level of their morphisms.

\begin{definition}[Category $\avbun$]\label{def_avbun}
\begin{itemize}
    \item Objects are pairs $(\calM,P)$, where $\calM$ is a graded manifold and $P$ is a \textit{finitely generated projective} $\cifty_\calM(M)$-module.
    \item Let $(\calM,P)$ and $(\calN,Q)$ be two objects. A morphism $(\calM,P) \rightarrow (\calN,Q)$ is a pair  $(\varphi, \Lambda_{0})$, where $\varphi: \calM \rightarrow \calN$ is a graded smooth map and $\Lambda_{0}: Q^{\star} \rightarrow P^{\star}$ is a degree zero $\R$-linear map satisfying the condition
    \begin{equation}
        \Lambda_{0}(g \cdot \xi) = \varphi^{\ast}(g) \cdot \Lambda_{0}(\xi),
    \end{equation}
    for all $\xi \in Q^{\star}$ and $g \in \cifty_{\calN}(N)$. 
    
    In other words, a graded algebra morphism $\varphi^{\ast}: \cifty_{\calN}(N) \rightarrow \cifty_{\calM}(M)$ induces a graded $\cifty_{\calN}(N)$-module structure on $P^{\star}$ and we require $\Lambda_{0}$ to be $\cifty_{\calN}(N)$-linear. Compositions and identity morphisms are defined in an obvious way. The resulting category is denoted as $\avbun$. 
\end{itemize}

\end{definition}

\begin{remark}
    There are some remarks in order.
    \begin{enumerate}
        \item Suppose $(\1_{\calM},\Lambda_{0}): (\calM,P) \rightarrow (\calM,Q)$. Then $\Lambda_{0}: Q^{\star} \rightarrow P^{\star}$ is always just the transpose of a unique morphism $F_{0}: P \rightarrow Q$ of $\cifty_{\calM}(M)$-modules. 
        \item $(\varphi,\Lambda_{0})$ is an isomorphism of $(\calM,P)$ and $(\calN,Q)$ in $\avbun$, iff $\varphi: \calM \rightarrow \calN$ is a graded diffeomorphism and $\Lambda_{0}: Q^{\star} \rightarrow P^{\star}$ is a $\cifty_{\calN}(N)$-module isomorphism with respect to the $\cifty_{\calN}(N)$-module structure induced on $P^{\star}$ by $\varphi^{\ast}$, see above. 
    \end{enumerate}
\end{remark}

We will now prove a sequence of lemmas aiming to the main goal - showing that the module of global sections of a sheaf-defined GVB is an example of finitely generated projective $\cifty_{\calM}(M)$-module. First, we  need to show that every graded vector bundle can be ``trivialized'' on a \textit{finite} open cover of the underlying manifold of its base graded manifold. 
\begin{lemma} \label{lem_finitetrivial}
Let $\scrS \in \svbun$ be a sheaf-defined GVB over a graded manifold $\calM$. Let $n_{0} := \dim(M)$ be the dimension of the manifold $M$ underlying $\calM$.

Then there exists a finite open cover $\{ U_{k} \}_{k=0}^{n_{0}}$, such that $\scrS|_{U_{k}}$ is freely and finitely generated. 
\end{lemma}
\begin{proof}
    The proof is a modification of the one of Lemma 7.1 in Chapter 2, \S 7 of \cite{walschap2012metric}. Since $\scrS$ is a locally freely and finitely generated sheaf of $\cifty_{\calM}$-modules, there is an open cover $\{ U_{\alpha} \}_{\alpha \in I}$ of $M$, such that $\scrS|_{U_{\alpha}}$ is freely and finitely generated. Without the loss of generality, we may assume that it is of order $n_{0} + 1$, that is each $m \in M$ lies in at most $n_{0} + 1$ its distinct open sets. This is a standard result of differential topology, see \S 50 of \cite{munkres2000topology}. Let $\{ \rho_{\alpha} \}_{\alpha \in I}$ be a partition of unity subordinate to this open cover. 

    Now, for a given $k \in \N_{0}$, let us consider a set $\calI_{k}$ of all subsets of $I$ consisting of $k+1$ elements. For any $a = \{ \alpha_{0}, \dots, \alpha_{k} \} \in \calI_{k}$, consider a subset
    \begin{equation} \label{eq_Wa(k)set}
        W_{a}^{(k)} := \{ m \in M \mid \rho_{\alpha}(m) < \min \{ \rho_{\alpha_{0}}(m), \dots, \rho_{\alpha_{k}}(m) \} \text{ for all } \alpha \notin a \}.
    \end{equation}
    We shall now prove some its properties. Let $k \in \N_{0}$ be fixed. Then
    \begin{enumerate}[(i)]
        \item $W_{a}^{(k)}$ is open for each $a \in \calI_{k}$;
        \item $W_{a}^{(k)} \cap W_{a'}^{(k)} = \emptyset$ whenever $a \neq a'$;
        \item $\scrS|_{W_{a}^{(k)}}$ is freely and finitely generated. 
    \end{enumerate}
    To see (i), let $m_{0} \in W_{a}^{(k)}$ be fixed. Since $\{ \supp \rho_{\alpha} \}_{\alpha \in I}$ is locally finite, there exists an open subset $U \in \Op_{m_{0}}(M)$, such that $U \cap \supp(\rho_{\alpha}) \ne \emptyset$ only for $\alpha$ in some finite subset $I_{0} \subseteq I$. This means that for $m \in U$, the condition in (\ref{eq_Wa(k)set}) has to be tested only for $\alpha \in I_{0}$. For each $\alpha \in I_{0} \ssm a$, one has $\rho_{\alpha}(m_{0}) < \min \{ \rho_{\alpha_{0}}(m_{0}), \dots, \rho_{\alpha_{k}}(m_{0}) \}$. This inequality is true also for all $m$ in some open neighborhood $V_{\alpha} \in \Op_{m_{0}}(U)$. By taking the intersection
    \begin{equation}
    \bigcap_{\alpha \in I_{0} \ssm a} V_{\alpha}
    \end{equation}
    we find a neighborhood of $m_{0}$ contained in $W_{a}^{(k)}$. Hence $W_{a}^{(k)}$ is open. 

    To prove (ii), let $a \neq a'$ be two elements of $\calI_{k}$. Since both sets have $k + 1$ elements, there is some $\alpha \in a$ not in $a'$ and some $\alpha' \in a'$ not in $a$. But then every $m \in W_{a}^{(k)} \cap W_{a'}^{(k)}$ would satisfy both $\rho_{\alpha}(m) < \rho_{\alpha'}(m)$ and $\rho_{\alpha'}(m) < \rho_{\alpha}(m)$, which is absurd. 

    It remains to prove (iii). Let $a := \{ \alpha_{0}, \dots, \alpha_{k} \}$. We claim that 
    \begin{equation}
        W_{a}^{(k)} \subseteq \bigcap_{j=0}^{k} \supp(\rho_{\alpha_{j}}) \subseteq \bigcap_{j=0}^{k} U_{\alpha_{j}}.
    \end{equation}
    The statement (iii) would then follow immediately from the fact that $\scrS|_{U_{\alpha}}$ is freely and finitely generated for each $\alpha \in I$. Now, suppose that there is $j \in \{0,\dots,k\}$, such that $m \notin \supp(\rho_{\alpha_{j}})$. But then $\rho_{\alpha}(m) \geq \rho_{\alpha_{j}}(m) = 0$ for all $\alpha \in I$ and thus $m \notin W_{a}^{(k)}$. This proves the first inclusion, the second one follows from the definition of partitions of unity.

    Next, for each $k \in \{0,\dots,n_{0}\}$, set
    \begin{equation}
        U_{k} := \bigcup_{a \in \calI_{k}} W_{a}^{(k)}.
    \end{equation}
    First, let us argue that $\scrS|_{U_{k}}$ is freely and finitely generated. 
    
    For each $a \in \calI_{k}$, we have a local frame $s_{1}^{(a)}, \dots, s_{r}^{(a)}$ for $\scrS$ over $W_{a}^{(k)}$. This follows from (iii). From (ii) and the fact that $\scrS$ is a sheaf, it follows that there is a unique collection $s_{1}, \dots, s_{r} \subseteq \scrS(U_{k})$, such that $s_{i}|_{W_{a}^{(k)}} = s_{i}^{(a)}$ for all $i \in \{1, \dots, r \}$ and $a \in \calI_{k}$. It is easy to see that this is a local frame for $\scrS$ over $U_{k}$, that is $\scrS|_{U_{k}}$ is freely and finitely generated. 

    It remains to prove that $\{ U_{k} \}_{k=0}^{n_{0}}$ covers $M$. Let $m \in M$. By the first paragraph of the proof, $m$ lies in at most $n_{0} + 1$ open sets of the open cover $\{ U_{\alpha} \}_{\alpha \in I}$. There is thus $k \in \{0, \dots, n_{0} \}$ and $a \in \calI_{k}$, such that $\rho_{\alpha}(m) > 0$ for all $\alpha \in a$ and $\rho_{\alpha}(m) = 0$ for all $\alpha \notin a$. But this means that $m \in W_{a}^{(k)} \subseteq U_{k}$. 
\end{proof}

\begin{lemma} \label{lem_sectionsarefingen}
Let $\scrS \in \svbun$ be a sheaf-defined GVB over a graded manifold $\calM$. 

Then $\scrS(M)$ is a finitely generated $\cifty_{\calM}(M)$-module. 
\end{lemma}
\begin{proof}
Let $\{ U_{k} \}_{k=0}^{n_{0}}$ be the open cover constructed in Lemma \ref{lem_finitetrivial}. For each $k \in \{0, \dots, n_{0}\}$, one thus has a local frame $s_{1}^{(k)}, \dots, s_{r}^{(k)}$ for $\scrS$ over $U_{k}$. Let $\{ \rho_{k} \}_{k=0}^{n_{0}}$ be a partition of unity subordinate to this open cover. For each $i \in \{1,\dots,r\}$ and $k \in \{0,\dots,n_{0}\}$, set
\begin{equation}
    \ol{s}_{i}^{(k)} := \rho_{k} \cdot s_{i}^{(k)} \in \scrS(M).
\end{equation}
For manipulations with partitions of unity in the graded setting, see \S 3.5 of \cite{Vysoky:2022gm}. 

For each $k \in \{0,\dots,n_{0}\}$, let us consider an open subset $V_{k} := \{ m \in U_{k} \mid \rho_{k}(m) > 0 \} \subseteq U_{k}$. Observe that $\{ V_{k} \}_{k=0}^{n_{0}}$ still covers $M$. Since $\sum_{k=0}^{n_{0}} \rho_{k} = 1$, for every $m \in M$, there must be $k \in \{0,\dots,n_{0}\}$, such that $\rho_{k}(m) > 0$. In particular, one has $m \in \supp(\ul{\rho_{k}}) \subseteq \supp(\rho_{k}) \subseteq U_{k}$, and thus $m \in V_{k}$. This ensures that $\rho_{k}|_{V_{k}}$ has a multiplicative inverse and thus $\ol{s}_{1}^{(k)}|_{V_{k}}, \dots, \ol{s}_{r}^{(k)}|_{V_{k}}$ defines a local frame for $\scrS$ over $V_{k}$ for each $k \in \{0,\dots,n_{0}\}$. 

We claim that the finite collection $\{ \ol{s}_{i}^{(k)} \}_{i,k}$ generates $\scrS(M)$. Let $s \in \scrS(M)$ be arbitrary. For each $k \in \{0, \dots, n_{0} \}$, there is a unique collection $\{ f^{i}_{(k)} \}_{i=1}^{r} \subseteq \cifty_{\calM}(V_{k})$, such that 
\begin{equation}
    s|_{V_{k}} = f^{i}_{(k)} \cdot \ol{s}_{i}^{(k)}|_{V_{k}}.
\end{equation}
Let $\{ \eta_{k} \}_{k=0}^{n_{0}}$ be a partition of unity subordinate to $\{ V_{k} \}_{k=0}^{n_{0}}$. One can write
\begin{equation}
    s = \sum_{k=0}^{n_{0}} \eta_{k} \cdot s|_{V_{k}} = \sum_{k=0}^{n_{0}} \eta_{k} \cdot (\sum_{i = 0}^{r} f^{i}_{(k)} \cdot \ol{s}^{(k)}_{i}|_{V_{k}} ) = \sum_{k=0}^{n_{0}} \sum_{i =0}^{r} (\eta_{k} \cdot f^{i}_{(k)}) \cdot \ol{s}^{(k)}_{i}.
\end{equation}
Note that $\eta_{k} \cdot f^{i}_{(k)} \in \cifty_{\calM}(M)$ and we have explicitly stressed the sum over $i$. This proves the claim.
\end{proof}

\begin{lemma} \label{lem_sectionsareproj}
Let $\scrS \in \svbun$ be a sheaf-defined GVB over a graded manifold $\calM$. 

Then $\scrS(M)$ is a projective $\cifty_{\calM}(M)$-module.    
\end{lemma}
\begin{proof}
    We know from Lemma \ref{lem_sectionsarefingen} that $\scrS(M)$ is finitely generated. Let $S = \{ s_{1}, \dots, s_{n} \}$ be a finite generating set for $\scrS(M)$. Let $K := \R(S)$ be the free graded vector space generated by $S$, and let $F_{0}: \cifty_{\calM}(M)[K] \rightarrow \scrS(M)$ be a morphism of $\cifty_{\calM}(M)$-modules satisfying
    \begin{equation}
        F_{0}(1 \otimes s_{i}) := s_{i},
    \end{equation}
    for all $i \in \{1, \dots, n\}$. By Proposition \ref{prop_GVBsoveridentity} and Proposition \ref{eq_GVBsurjectiveequivalent}, it defines a surjective GVB morphism $F: \cifty_{\calM}[K] \rightarrow \scrS$, see also the proof of Proposition \ref{prop_localframecompl}. Moreover, since $F$ is surjective, $\ker(F) \subseteq \cifty_{\calM}[K]$ is a subbundle, see Proposition \ref{prop_kerimsubbundles}. We thus obtain a short exact sequence of GVBs:
    \begin{equation}
        \begin{tikzcd}
            0 \arrow{r} & \ker(F) \arrow{r} & \cifty_{\calM}[K] \arrow{r}{F} & \scrS \arrow{r} & 0 
        \end{tikzcd}
    \end{equation}
    By Proposition \ref{tvrz_GVBsurjectivesplits}, this sequence splits, hence there is an isomorphism $\cifty_{\calM}[K] \cong \ker(F) \oplus \scrS$. In particular, this implies that 
    \begin{equation}
        \cifty_{\calM}(M)[K] \cong \ker(F_{0}) \oplus \scrS(M).
    \end{equation}
    This shows that $\scrS(M)$ is isomorphic to a summand of a free $\cifty_{\calM}(M)$-module. Hence it is projective by Proposition \ref{tvrz_projective}. This finishes the proof.
\end{proof}
We are now ready to prove the main statement of this subsection.
\begin{proposition} \label{prop_upsilon1}
    Let $\scrS \in \svbun$ be a sheaf-defined GVB over a graded manifold $\calM$. 
    
    Then $(\calM, \scrS(M))$ is an algebraic GVB. Moreover, the assignment $\scrS \mapsto (\calM, \scrS(M))$ defines a fully faithful functor $\Upsilon: \svbun \rightarrow \avbun$.
\end{proposition}
\begin{proof}
    The fact that $(\calM,\scrS(M)) \in \avbun$ follows from Lemma \ref{lem_sectionsarefingen} and Lemma \ref{lem_sectionsareproj}. Moreover, for any GVB morphism $(\varphi,\Lambda)$, we set $\Upsilon(\varphi,\Lambda) := (\varphi, \Lambda_{0})$, see Proposition \ref{prop_GVBmorphisms}. It is easy to see that this makes $\Upsilon$ into a functor. It follows immediately from Proposition \ref{prop_GVBmorphisms} that $\Upsilon$ is fully faithful.
\end{proof}
\begin{lemma} \label{lem_trivialfree}
    Let $\scrS \in \svbun$ be a sheaf-defined GVB over a manifold $\calM$. 

    Then $\scrS$ is trivial, iff $\scrS(M)$ is free.
\end{lemma}
\begin{proof}
    Recall that $\scrS$ is called trivial when it is freely and finitely generated. This immediately implies that $\scrS(M)$ is free. Conversely, if $\scrS(M)$ is free, it is isomorphic to $\cifty_{\calM}(M)[K]$ for some finite-dimensional graded vector space $K$, see Proposition \ref{prop_freeisAK} and Remark \ref{rem_aboutfrees}. 
    
    But $\cifty_{\calM}(M)[K]$ is a module of global sections of the sheaf $\cifty_{\calM}[K]$, see the proof of Proposition \ref{prop_localframecompl} or Example 2.37 in \cite{Vysoky:2022gm}. This sheaf of $\cifty_{\calM}$-modules is freely and finitely generated, hence a trivial GVB. The isomorphism of modules of global sections extends uniquely to a GVB isomorphism by Proposition \ref{prop_GVBsoveridentity}, that is $\scrS \cong \cifty_{\calM}[K]$. Hence $\scrS$ is also a trivial GVB. 
\end{proof}
\subsection{Serre--Swan Theorem}
Our ultimate goal is to prove that $\Upsilon$ is an equivalence of categories. In view of Proposition \ref{prop_GVBmorphisms}, it remains to prove that the functor is essentially surjective. In other words, to any algebraic GVB $(\calM, P)$, we must find a sheaf-defined GVB $\scrP$, such that $(\calM,P) \cong \Upsilon(\scrP) \equiv (\calM, \scrP(M))$. We will directly construct a sheaf $\scrP$. To do so, we need a sequence of technical results. Let us start with a simple observation.
\begin{lemma} \label{lem_localaredetbyglobal}
    Let $\scrS$ be a sheaf of $\cifty_{\calM}$-modules and let $\scrP \subseteq \scrS$ be its sheaf of $\cifty_{\calM}$-submodules. Let $s \in \scrS(U)$ for some $U \in \Op(M)$. Then the following two properties are equivalent:
    \begin{enumerate}
        \item $s \in \scrP(U)$;
        \item For each $m \in U$, there exists $V \in \Op_{m}(U)$ and $s' \in \scrP(M)$, such that $s'|_{V} = s|_{V}$.
    \end{enumerate}
\end{lemma}
\begin{proof}
    Suppose $s \in \scrP(U)$, and let $m \in U$. Choose some $V \in \Op_{m}(U)$ with $\ol{V} \subseteq U$ and a smooth bump function $\eta \in \cifty_{\calM}(M)$ supported in $U$ and satisfying $\eta|_{V} = 1$. Let $s' := \eta \cdot s$. Let $U' := M \ssm \supp(\eta)$. Then
    \begin{equation}
        s'|_{U} = \eta|_{U} \cdot s \in \scrP(U), \; \; s'|_{U'} = 0 \in \scrP(U').
    \end{equation}
    Since $\{ U, U'\}$ is an open cover of $M$ and $\scrP$ is a subsheaf, this implies that $s' \in \scrP(M)$. Moreover, one has $s'|_{V} = \eta|_{V} \cdot s|_{V} = s|_{V}$. The implication \textit{1} $\Rightarrow$ \textit{2} is proved. Conversely, suppose that the property \textit{2} holds. In particular, this implies that for each $m \in U$, one has $V \in \Op_{m}(U)$, such that $s|_{V} \in \scrP(V)$. Since $\scrP$ is a subsheaf, this proves that $s \in \scrP(U)$. This proves the implication \textit{2} $\Rightarrow$ \textit{1}.
\end{proof}
This lemma shows that local sections of a sheaf $\scrP$ of a $\cifty_{\calM}$-submodules are uniquely characterized by the property that they are locally restrictions of elements of a $\cifty_{\calM}(M)$-submodule $\scrP(M)$. We will utilize this to \textit{construct} a subsheaf $\scrP$ starting from any given submodule $P \subseteq \scrS(M)$. 
\begin{definition}
    Let $\scrS$ be a sheaf of $\cifty_{\calM}$-modules. A \textbf{support} of $s \in \scrS(M)$ is a subset
    \begin{equation}
        \supp(s) := \{ m \in M \mid [s]_{m} \neq 0 \} \subseteq M,
    \end{equation}
    where $[s]_{m}$ is the germ of the section $s$ in the stalk $\scrS_{m}$. 

    Let $\{ s_{\alpha} \}_{\alpha \in I} \subseteq \scrS(M)$ be any collection of sections. We say that it is \textbf{locally finite}, if the collection $\{ \supp(s_{\alpha}) \}_{\alpha \in I}$ is locally finite. 
\end{definition}
\begin{lemma}
    For each $s \in \scrS(M)$, the subset $\supp(s)$ is closed and $s|_{M \ssm \supp(s)} = 0$.

    For any locally finite collection $\{ s_{\alpha} \}_{\alpha \in I} \subseteq \scrS(M)$, there is a well-defined sum $\sum_{\alpha \in I} s_{\alpha} \in \scrS(M)$.
\end{lemma}
\begin{proof}
The first claim is an easy exercise. For the second claim: let $m \in M$ be arbitrary. Since $\{ s_{\alpha} \}_{\alpha \in I}$ is locally finite, one can find $U_{m} \in \Op_{m}(M)$ a finite subset $I_{m} \subseteq I$, such that $U_{m} \cap \supp(s_{\alpha}) = \emptyset$ for all $\alpha \in I \ssm I_{m}$. The sum $\sum_{\alpha \in I} s_{\alpha} \in \scrS(M)$ is then uniquely determined by the requirement
\begin{equation} \label{eq_locfinitesum}
    (\sum_{\alpha \in I} s_{\alpha})|_{U_{m}} = \sum_{\alpha \in I_{m}} s_{\alpha}|_{U_{m}}.
\end{equation}
We leave for the reader to prove that such a section exists, it is unique, and it does not depend on any choices.
\end{proof}
\begin{definition}
    Let $\scrS$ be a sheaf of $\cifty_{\calM}$-modules. We say that a $\cifty_{\calM}(M)$-submodule $P \subseteq \scrS(M)$ is \textbf{closed under locally finite sums}, if for any locally finite collection $\{ s_{\alpha} \}_{\alpha \in I} \subseteq P$, one has $\sum_{\alpha \in I} s_{\alpha} \in P$. 
\end{definition}
\begin{lemma} \label{lem_sheafglobalclosedunderlocfinsums}
    Let $\scrP \subseteq \scrS$ be a sheaf of $\cifty_{\calM}$-submodules. 
    
    Then $P := \scrP(M)$ is closed under locally finite sums. 
\end{lemma}
\begin{proof}
    This follows immediately from the formula (\ref{eq_locfinitesum}) and the fact that $\scrP$ is a subsheaf.
\end{proof}
We are now ready to prove that any submodule closed under finite sums can be uniquely extended to a sheaf of submodules. The idea is motivated by Lemma \ref{lem_localaredetbyglobal}. 
\begin{proposition} \label{prop_extensionofsubmodule}
    Let $\scrS$ be a sheaf of $\cifty_{\calM}$-modules. Suppose $P \subseteq \scrS(M)$ is a $\cifty_{\calM}(M)$-submodule. For each $U \in \Op(M)$, define a graded subset 
    \begin{equation}
    \begin{split}
        \scrP(U) := \{ s \in \scrS(U) \mid & \text{ for each } m \in U, \text{ there is } V \in \Op_{m}(U) \\
        & \text{ and } s' \in P, \text{such that  } s'|_{V} = s|_{V} \}.
    \end{split}
    \end{equation}
    Then $\scrP$ is a sheaf of $\cifty_{\calM}$-submodules of $\scrS$ and $P \subseteq \scrP(M)$. Moreover, one has $P = \scrP(M)$, iff $P$ is closed under locally finite sums.
\end{proposition}
\begin{proof}
    Let us divide the proof into several steps. 
    \begin{enumerate}
        \item $\scrP(U)$ is a $\cifty_{\calM}(U)$-submodule of $\scrS(U)$: It is easy to see that $\scrP(U)$ is a graded subspace. Let $s \in \scrP(U)$ and $f \in \cifty_{\calM}(U)$ be arbitrary. We need to check that $f \cdot s \in \scrP(U)$. For any given $m \in U$, there is $V \in \Op_{m}(U)$ and $s' \in P$, such that $s'|_{V} = s|_{V}$. 
        By shrinking $V$ if necessary, we can assume that $\ol{V} \subseteq U$. Pick a bump function $\eta \in \cifty_{\calM}(M)$ supported in $U$ and satisfying $\eta|_{V} = 1$, and let $f' := \eta \cdot f$. We have $f' \cdot s' \in P$ and $(f' \cdot s')|_{V} = (f \cdot s)|_{V}$. Hence $f \cdot s \in \scrP(U)$. 

        \item $\scrP$ is a sheaf of graded $\cifty_{\calM}$-modules: It is easy to see that $\scrP$ is a subpresheaf. Let $U \in \Op(M)$ and let $\{ U_{\alpha} \}_{\alpha \in I}$ be an open cover of $U$. To prove that it is a subsheaf, we must argue that whenever some $s \in \scrS(U)$ satisfies $s|_{U_{\alpha}} \in \scrP(U_{\alpha})$, then $s \in \scrP(U)$. Let $m \in U$ be arbitrary. There is thus some $\alpha \in I$, such that $m \in U_{\alpha}$. Since $s|_{U_{\alpha}} \in \scrP(U_{\alpha})$, there is some $V \in \Op_{m}(U_{\alpha})$ and $s' \in P$, such that $s'|_{V} = (s|_{U_{\alpha}})|_{V}$. The right-hand side can be written simply as $s|_{V}$. Since $V \in \Op_{m}(U)$ and $m \in U$ was arbitrary, this proves that $s \in \scrP(U)$. 

        \item The relation of $\scrP(M)$ to $P$: It is trivial to check that $P \subseteq \scrP(M)$. Suppose that $P$ is closed under locally finite sums. Let $s \in \scrP(M)$. For each $m \in M$, we thus have $V_{(m)} \in \Op_{m}(M)$ and $s'_{(m)} \in P$, such that $s'_{(m)}|_{V_{(m)}} = s|_{V_{(m)}}$. Let $\{ \rho_{(m)} \}_{m \in M}$ be a partition of unity subordinate to the open cover $\{ V_{(m)} \}_{m \in M}$. One can use it to write 
        \begin{equation}
            s = \sum_{m \in M} \rho_{(m)} \cdot s|_{V_{(m)}} = \sum_{m \in M} \rho_{(m)} \cdot s'_{(m)}|_{V_{(m)}} = \sum_{m \in M} \rho_{(m)} \cdot s'_{(m)}. 
        \end{equation}
        It is easy to check that $\{ \rho_{(m)} \cdot s'_{(m)} \}_{m \in M}$ is locally finite. Since $P$ is closed under locally finite sums, this shows that $s \in P$. Hence $\scrP(M) = P$. 

        The converse statement follows from Lemma \ref{lem_sheafglobalclosedunderlocfinsums}.
    \end{enumerate}
    This finishes the proof. 
\end{proof}
\begin{remark} \label{rem_sheafextension}
    Let $P \subseteq \scrS(M)$ be any $\cifty_{\calM}(M)$-submodule. Then a sheaf $\scrP \subseteq \scrS$ of $\cifty_{\calM}$-submodules satisfying $\scrP(M) = P$ exists, iff $P$ is closed under locally finite sums. Moreover, if it exists, it is unique. This follows immediately from the above proposition together with Lemma \ref{lem_localaredetbyglobal}. 
\end{remark}
Let us now turn our attention to the case where $P$ is finitely generated. 
\begin{lemma} \label{lem_wholesheaffingen}
    Let $P \subseteq \scrS(M)$ be a finitely generated $\cifty_{\calM}(M)$-submodule. Let $\scrP$ be the sheaf of $\cifty_{\calM}$-submodules obtained from $P$ as in Proposition \ref{prop_extensionofsubmodule}.

    Then $\scrP(U)$ is finitely generated for each $U \in \Op(M)$. 
\end{lemma}
\begin{proof}
    Let $S = \{ s_{1}, \dots, s_{n}\}$ be a finite generating set for $P$. Let $s \in \scrP(U)$ for some given $U \in \Op(M)$. For each $m \in U$, we can thus find $V_{(m)} \in \Op_{m}(U)$ and $s'_{(m)} \in P$, such that $s'_{(m)}|_{V_{(m)}} = s|_{V_{(m)}}$. For each $m \in U$, we can find an $n$-tuple $f_{(m)}^{1}, \dots, f_{(m)}^{n} \in \cifty_{\calM}(M)$, such that 
    \begin{equation}
        s'_{(m)} = \sum_{i=1}^{n} f^{i}_{(m)} \cdot s_{i}.
    \end{equation}
    Let $\{ \rho_{(m)} \}_{m \in U}$ be a partition of unity subordinate to the open cover $\{ V_{(m)} \}_{m \in U}$ of $U$. Then
    \begin{equation}
    \begin{split}
        s = & \ \sum_{m \in U} \rho_{(m)} \cdot s|_{V_{(m)}} = \sum_{m \in U} \rho_{(m)} \cdot s'_{(m)}|_{V_{(m)}} = \sum_{m \in U} \rho_{(m)} \cdot ( \sum_{i=1}^{n} f^{i}_{(m)}|_{V_{(m)}} \cdot s_{i}|_{V_{(m)}}) \\
        = & \ \sum_{i=1}^{n} ( \sum_{m \in U} \rho_{(m)} \cdot f^{i}_{(m)}|_{U}) \cdot s_{i}|_{U}.
    \end{split}
    \end{equation}
    But this shows that $\{ s_{1}|_{U}, \dots, s_{n}|_{U}\}$ is a finite generating set for $\scrP(U)$.
\end{proof}
\begin{corollary} \label{cor_projclosed}
    Let $P \subseteq \scrS(M)$ be a finitely generated $\cifty_{\calM}(M)$-submodule. 
    
    Then $P$ is closed under locally finite sums.
\end{corollary}
\begin{proof}
    It suffices to check that $\scrP(M) \subseteq P$, where $\scrP$ is the subsheaf obtained as in Proposition \ref{prop_extensionofsubmodule}. From the last step of the proof of the previous lemma for $U = M$, we see that every $s \in \scrP(M)$ can be written as a $\cifty_{\calM}(M)$-linear combination of elements of $P$. 
\end{proof}
As a next step, we will consider finitely generated projective modules.
\begin{proposition} \label{prop_fingenprojimpliesdirectsumsubsheaves}
    Let $P$ be a finitely generated projective $\cifty_{\calM}(M)$-module.

    Then there exists a trivial $\scrS \in \svbun$ and its sheaves $\scrP$ and $\scrQ$ of $\cifty_{\calM}$-submodules, such that $\scrS = \scrP \oplus \scrQ$, and $P \cong \scrP(M)$.
\end{proposition}
\begin{proof}
    By Proposition \ref{tvrz_projective}, $P$ is isomorphic to a direct summand of a free $\cifty_{\calM}(M)$-module $F$. By Proposition \ref{prop_projectiveprops}, we may assume that $F$ is finitely generated. Using the same argument as in the proof of Lemma \ref{lem_trivialfree}, the module $F$ is isomorphic to $\scrS(M)$ for a trivial $\scrS \in \svbun$. Without the loss of generality, we can thus assume that
    \begin{equation} \label{eq_PQdirectsum}
        \scrS(M) = P \oplus Q,
    \end{equation}
    for some projective $\cifty_{\calM}(M)$-module $Q$. Note that since $Q \cong \scrS(M)/P$, it is also finitely generated. By Proposition \ref{prop_extensionofsubmodule} together with Corollary \ref{cor_projclosed}, there are sheaves $\scrP$ and $\scrQ$ of $\cifty_{\calM}$-submodules of $\scrS$, such that $P = \scrP(M)$ and $Q = \scrQ(M)$. To finish the proof, one only has to argue that 
    \begin{equation}
    \scrS(U) = \scrP(U) \oplus \scrQ(U),
    \end{equation}
    for all $U \in \Op(M)$. Let $s \in \scrS(U)$ be arbitrary. For each $m \in U$, there exists $V_{(m)} \in \Op_{m}(U)$ and $s_{(m)} \in \scrS(M)$, such that $s_{(m)}|_{V_{(m)}} = s|_{V_{(m)}}$, see Lemma \ref{lem_localaredetbyglobal}. It follows from (\ref{eq_PQdirectsum}) that 
    \begin{equation}
        s_{(m)} = p_{(m)} + q_{(m)}
    \end{equation}
    for unique elements $p_{(m)} \in P$ and $q_{(m)} \in Q$. Let $\{ \rho_{(m)} \}_{m \in U}$ be a partition of unity subordinate to the open cover $\{ V_{(m)} \}_{m \in U}$. Then
    \begin{equation}
    \begin{split}
        s = & \ \sum_{m \in U} \rho_{(m)} \cdot s|_{V_{m}} = \sum_{m \in U} \rho_{(m)} \cdot s_{(m)}|_{V_{(m)}} = \sum_{m \in U} \rho_{(m)} \cdot ( p_{(m)}|_{V_{(m)}} + q_{(m)}|_{V_{(m)}}) \\
        = & \ \sum_{m \in U} \rho_{(m)} \cdot p_{(m)}|_{U} + \sum_{m \in U} \rho_{(m)} \cdot q_{(m)}|_{U}.
    \end{split}
    \end{equation}
    Observe that the collection $\{ \rho_{(m)} \cdot p_{(m)}|_{U} \}_{m \in U} \subseteq \scrP(U)$ is locally finite. Since $\scrP(U)$ is finitely generated by Lemma \ref{lem_wholesheaffingen}, it is closed under locally finite sums by Corollary \ref{cor_projclosed}. This implies that the first sum in $\scrP(U)$. Using the same reasoning, the second sum is in $\scrQ(U)$. This proves that 
    \begin{equation} \label{eq_SasPplusQ}
        \scrS(U) = \scrP(U) + \scrQ(U).
    \end{equation}
    Finally, note that $\scrP \cap \scrQ$ is a sheaf of $\cifty_{\calM}$-submodules of $\scrS$ satisfying $(\scrP \cap \scrQ)(M) = P \cap Q = 0$. Hence by the uniqueness assertion in Remark \ref{rem_sheafextension}, one has $\scrP \cap \scrQ = 0$ and the sum  (\ref{eq_SasPplusQ}) is direct. 
\end{proof}
Starting with a finitely generated projective $P$, we have constructed a sheaf $\scrP$ of $\cifty_{\calM}$-submodules of a trivial $\scrS \in \svbun$. It remains to prove that it is actually its subbundle, hence a GVB itself. To do so, we employ the following observation:
\begin{proposition} \label{prop_PandQaresubbundles}
    Let $\scrS \in \svbun$. Suppose that $\scrS = \scrP \oplus \scrQ$ for two its sheaves of $\cifty_{\calM}$-submodules. 

    Then both $\scrP$ and $\scrQ$ are actually subbundles of $\scrS$. 
\end{proposition}
\begin{proof}
    For each $m \in M$, we can consider graded subspaces of the fiber $\scrS_{(m)}$ obtained by collecting the values of elements of $\scrP(M)$ and $\scrQ(M)$, that is 
    \begin{equation}
        \scrP_{[m]} := \{ s|_{m} \mid s \in \scrP(M) \}, \; \; \scrQ_{[m]} := \{ s|_{m} \mid s \in \scrQ(M) \}.
    \end{equation}
    We claim that one has a direct sum decomposition
    \begin{equation} \label{eq_fiberofSasfibersofPandQ}
        \scrS_{(m)} = \scrP_{[m]} \oplus \scrQ_{[m]}.
    \end{equation}
    Now, every element of $\scrS_{(m)}$ can be written as $s|_{m}$ for some $s \in \scrS(M)$. Since $s = p + q$ for $p \in \scrP(M)$ and $q \in \scrQ(M)$, we have $s|_{m} = p|_{m} + q|_{m}$. This proves that $\scrS_{(m)} = \scrP_{[m]} + \scrQ_{[m]}$. A difficult bit is to prove that the sum is direct. Suppose that there exist $p \in \scrP(M)$ and $q \in \scrQ(M)$, such that $p|_{m} = q|_{m}$. From the proof of Lemma \ref{lem_fiber}, it follows that 
    \begin{equation}
        p - q \in \calJ_{m} \cdot \scrS(M).
    \end{equation}
    There is thus a finite collection $\{ s_{i} \}_{i=1}^{k} \subseteq \scrS(M)$ and functions $\{ f^{i} \}_{i=1}^{k} \subseteq \calJ_{m}$, such that 
    \begin{equation} \label{eq_pminusq}
        p - q = \sum_{i=1}^{k} f^{i} \cdot s_{i}.
    \end{equation}
    For each $i \in \{1, \dots, k\}$, one can write $s_{i} = p_{i} + q_{i}$, where $p_{i} \in \scrP(M)$ and $q_{i} \in \scrQ(M)$. The equation (\ref{eq_pminusq}) can be then rewritten as 
    \begin{equation}
        p - \sum_{i=1}^{k} f^{i} \cdot p_{i} = q + \sum_{i=1}^{k} f^{i} \cdot q_{i}.
    \end{equation}
    Since $\scrP(M) \cap \scrQ(M) = 0$, both sides of this equation must vanish. This immediately implies
    \begin{equation}
        p|_{m} = \sum_{i=1}^{k} f^{i}(m) \; p_{i}|_{m} = 0,
    \end{equation}
    since $f^{i} \in \calJ_{m}$ and thus $f^{i}(m) = 0$. Hence $\scrP_{[m]} \cap \scrQ_{[m]} = 0$ and the sum (\ref{eq_fiberofSasfibersofPandQ}) is indeed direct. 

    Next, we claim that the sequence $\gdim(\scrP_{[m]})$ is constant in $m$. Fix some $m_{0} \in M$. One can choose a collection $\{ s_{i} \}_{i=1}^{k} \subseteq \scrP(M)$, such that $s_{1}|_{m_{0}}, \dots, s_{k}|_{m_{0}}$ is a basis for $\scrP_{[m_{0}]}$. It follows from Proposition \ref{prop_localframecompl} that there exists $U \in \Op_{m_{0}}(M)$, such that $s_{1}|_{m}, \dots, s_{k}|_{m}$ are linearly independent elements of $\scrP_{[m]}$ for all $m \in U$. Hence $\gdim(\scrP_{[m]}) \geq \gdim(\scrP_{[m_{0}]})$ for all $m \in U$. Similarly, one can find $U' \in \Op_{m_{0}}(M)$, such that $\gdim(\scrQ_{[m]}) \geq \gdim(\scrQ_{[m_{0}]})$, for all $m \in U'$. But thanks to (\ref{eq_fiberofSasfibersofPandQ}), there is a constraint
    \begin{equation}
        \grk(\scrS) = \gdim(\scrS_{(m)}) = \gdim(\scrP_{[m]}) + \gdim(\scrQ_{[m]}),
    \end{equation}
    for all $m \in M$. This shows that $\gdim( \scrP_{[m]}) = \gdim( \scrP_{[m_{0}]})$ for all $m \in U \cap U'$. Since $m_{0} \in M$ was arbitrary and $M$ is assumed connected\footnote{In view of Remark \ref{rem_connected}, this is the only point where we actually use this assumption.}, $\gdim( \scrP_{[m]})$ is the same for all $m \in M$.

    Let $m \in M$ be arbitrary. We seek a local frame for $\scrS$ adapted to both $\scrP$ and $\scrQ$. Find a collection $\{ a_{i} \}_{i=1}^{r} \subseteq \scrS(M)$ as follows:
    \begin{enumerate}
        \item $a_{1}, \dots a_{k} \in \scrP(M)$ and their values at $m$ form a basis of $\scrP_{[m]}$;
        \item $a_{k+1}, \dots, a_{r} \in \scrQ(M)$ and their values at $m$ form a basis of $\scrQ_{[m]}$.
    \end{enumerate}
    It follows from (\ref{eq_fiberofSasfibersofPandQ}) that $a_{1}|_{m}, \dots, a_{r}|_{m}$ is a basis of the fiber $\scrS_{(m)}$. By Proposition \ref{prop_localframecompl}, we there exists $V \in \Op_{m}(M)$ and a local frame $s_{1}, \dots, s_{r}$ for $\scrS$ over $V$, such that $s_{i} = a_{i}|_{V}$ for every $i \in \{1, \dots, r\}$. Let $U \in \Op_{m}(V)$ be arbitrary. Let us argue that $s_{1}|_{U}, \dots, s_{k}|_{U}$ generates $\scrP(U)$. Let $s \in \scrP(U)$ be arbitrary. Since $s_{1}, \dots, s_{r}$ is a local frame for $\scrS$, one can write
    \begin{equation}
        s = \sum_{i=1}^{r} f^{i} \cdot s_{i}|_{U} = \sum_{i=1}^{k} f^{i} \cdot s_{i}|_{U} + \sum_{i=k+1}^{r} f^{i} \cdot s_{i}|_{U}.
    \end{equation}
    for unique functions $f_{1}, \dots, f_{r} \in \cifty_{\calM}(U)$. The first sum is in $\scrP(U)$ and the second one is in $\scrQ(U)$, that is one has 
    \begin{equation}
        s - \sum_{i=1}^{k} f^{i} \cdot s_{i}|_{U} = \sum_{i=k+1}^{r} f^{i} \cdot s_{i}|_{U} \in (\scrP \cap \scrQ)(U) 
    \end{equation}
    It follows from the assumptions that both sides of the equation must be trivial. In particular, we see that $s_{1}|_{U}, \dots, s_{k}|_{U}$ indeed generates $\scrP(U)$, hence $s_{1}, \dots, s_{r}$ is adapted to $\scrP$. Using the same arguments, it is adapted to $\scrQ$. 

    Finally, since $\gdim( \scrP_{[m]})$  and $\gdim( \scrQ_{[m]})$ are constant in $m \in M$, we can repeat this procedure consistently around each point to prove that $\scrP$ and $\scrQ$ are subbundles of $\scrS$, see the first paragraph of Subsection \ref{sec_subbundles}. This finishes the proof.
\end{proof}

We are ready to prove the main statement of this subsection.
\begin{theorem}[Serre--Swan theorem for GVBs]\label{thm_SerreSwan}
Let $(\calM,P) \in \avbun$. 

Then there exists $\scrP \in \svbun$, such that $(\calM,P) \cong (\calM, \scrP(M)) \equiv \Upsilon(\scrP)$. In particular, the functor $\Upsilon$ is essentially surjective and the categories $\svbun$ and $\avbun$ are equivalent.
\end{theorem}
\begin{proof}
    It follows from Proposition \ref{prop_fingenprojimpliesdirectsumsubsheaves} together with Proposition \ref{prop_PandQaresubbundles} that there exists $\scrP \in \svbun$, such that $P \cong \scrP(M)$. This isomorphism of $\cifty_{\calM}(M)$-modules can be interpreted as an isomorphism over $\1_{\calM}$ of $(\calM,P)$ and $\Upsilon(\scrP)$ in $\avbun$. This shows that $\Upsilon$ is essentially surjective. Since we know from Proposition \ref{prop_upsilon1} that is is fully faithful, it is an equivalence of categories. 
\end{proof}
\begin{remark}
    Note that this theorem directly contradicts Remark 1.3 in \cite{vysoky2022graded}. It turns out there was a rather silly computational error in a carefully crafted counterexample we had at a time. Its discovery drove us to the creation of this paper. 
\end{remark}
\subsection{Applications: Tensor product}
Let us utilize tools introduced in this paper to make one very useful observation. Recall that for $\scrS, \scrP \in \svbun$, their tensor product was defined in \cite{Vysoky:2022gm} as a sheafification
\begin{equation}
    \scrS \otimes_{\cifty_{\calM}} \scrP := \Sff( \scrS \otimes_{\cifty_{\calM}}^{P} \scrP),
\end{equation}
of the ``naive'' tensor product presheaf $\scrS \otimes_{\cifty_{\calM}}^{P} \scrP$. For each $U \in \Op(M)$, it is given by
\begin{equation}
    (\scrS \otimes_{\cifty_{\calM}}^{P} \scrP)(U) := \scrS(U) \otimes_{\cifty_{\calM}(U)} \scrP(U)
\end{equation}
with obvious presheaf restrictions. For tensor products of general sheaves of $\cifty_{\calM}$-modules, it may happen to \textit{not} be a sheaf. It is the aim of this subsection to show that for GVBs everything is fine. 

Recall that there is a canonical $\cifty_{\calM}$-linear presheaf morphism 
\begin{equation} \label{eq_sfftensorproducts}
    \sff: \scrS \otimes_{\cifty_{\calM}}^{P} \scrP \rightarrow \scrS \otimes_{\cifty_{\calM}} \scrP.
\end{equation}
We will prove that this is an isomorphism. Let us start with the following observation.
\begin{lemma} \label{lem_sfffortrivial}
    Suppose $\scrS$ and $\scrP$ are trivial GVBs. 
    
    Then $\sff$ is an isomorphism. 
\end{lemma}

\begin{proof}
    Without the loss of generality, we may assume that $\scrS = \cifty_{\calM}[K]$ and $\scrP = \cifty_{\calM}[L]$ for finite-dimensional graded vector spaces $K$ and $L$, respectively. See also the proof of Proposition \ref{prop_localframecompl}. It is not difficult to prove that there is an obvious $\cifty_{\calM}$-linear presheaf isomorphism
    \begin{equation}
        \cifty_{\calM}[K] \otimes^{P}_{\cifty_{\calM}} \cifty_{\calM}[L] \cong \cifty_{\calM}[K \otimes_{\R} L].
    \end{equation}
    This proves that the left-hand side is actually a sheaf, hence $\sff$ must be an isomorphism. 
\end{proof}

\begin{lemma} \label{lem_summandtrivial}
    Let $\scrS \in \svbun$. Then $\scrS$ is a direct summand of some trivial GVB $\scrF$. 
\end{lemma}
\begin{proof}
    We have shown this in the proof of Lemma \ref{lem_sectionsareproj}.
\end{proof}
The proof of the following statement is based on the one of Theorem 7.5.5 in \cite{conlon2001differentiable}.
\begin{theorem}\label{thm_tensorProd}
    Let $\scrS$ and $\scrP$ be arbitrary GVBs. Then $\sff$ is an isomorphism. 

    In other words, the ``naive'' tensor product $\scrS \otimes_{\cifty_{\calM}}^{P} \scrP$ is already a sheaf, so we can canonically identify
    \begin{equation}
        (\scrS \otimes_{\cifty_{\calM}} \scrP)(U) \cong \scrS(U) \otimes_{\cifty_{\calM}(U)} \scrP(U),
    \end{equation}
    for each $U \in \Op(M)$. 
\end{theorem}
\begin{proof}
    It follows from Lemma \ref{lem_summandtrivial} that there exist GVBs $\scrS^{\perp}$ and $\scrP^{\perp}$, such that $\scrF = \scrS \oplus \scrS^{\perp}$ and $\scrF' = \scrP \oplus \scrP^{\perp}$ are trivial GVBs. It follows from Lemma \ref{lem_sfffortrivial} that the corresponding canonical $\cifty_{\calM}$-linear presheaf morphism
    \begin{equation}
        \sff': \scrF \otimes_{\cifty_{\calM}}^{P} \scrF' \rightarrow \scrF \otimes_{\cifty_{\calM}} \scrF'
    \end{equation}
    is an isomorphism. Let $I_{1}: \scrS \rightarrow \scrF$ and $I_{2}: \scrP \rightarrow
     \scrF'$ denote the direct sum inclusions and $P_{1}: \scrF \rightarrow \scrS$ and $P_{2}: \scrF' \rightarrow \scrP$ the direct sum projections. One can form a commutative diagram
     \begin{equation} \label{eq_sffsffprimeCD}
         \begin{tikzcd}
             \scrS \otimes^{P}_{\cifty_{\calM}} \scrP \arrow{d}{I_{1} \otimes^{P} I_{2}} \arrow{r}{\sff} & \scrS \otimes_{\cifty_{\calM}} \scrP \arrow{d}{I_{1} \otimes I_{2}}\\
             \scrF \otimes^{P}_{\cifty_{\calM}} \scrF' \arrow{r}{\sff'} & \scrF \otimes_{\cifty_{\calM}} \scrF'
         \end{tikzcd},
     \end{equation}
     where $I_{1} \otimes^{P} I_{2}$ is the ``naive'' tensor product of the two sheaf morphisms defined in the obvious way, and the other vertical arrow $I_{1} \otimes I_{2}$ is in fact \textit{defined} to make this diagram commutative. The existence of such a sheaf morphism follows from the universal property of a sheafification functor.

     Now, since $P_{1} \circ I_{1} = \1_{\scrS}$ and $P_{2} \circ I_{2} = \1_{\scrP}$, it follows that $P_{1} \otimes^{P} P_{2}$ is a left inverse to $I_{1} \otimes^{P} I_{2}$, rendering it injective. Since $\sff'$ is an isomorphism, the commutativity of (\ref{eq_sffsffprimeCD}) immediately implies the injectivity of $\sff$. 
     
     Finally, by replacing (\ref{eq_sffsffprimeCD}) by a similar diagram using the projections $P_{1}$ and $P_{2}$ (with reversed vertical arrows), one can prove that $\sff$ is surjective. This finishes the proof. 
\end{proof}

\section{Geometric Approach}
\subsection{Definitions}

In this section we give the third approach to GVBs, which we call the geometric one. Unlike in previous sections, much of the work here will be done locally in coordinates. We begin by finding a suitable version of fiberwise linearity in the graded setting. Similarly as in \cite{bruceGrabowski2024}, we will start with ``fiberwise homogeneity" and show that linearity is implied.

Consider the graded manifold $\gR^{(r_j)}$ from Example \ref{example_graded_domain} with coordinates $\{x^i, \xi^\mu\} =: k^a$. One may reinterpret the addition of vectors and multiplication by a scalar $\lambda \in \R$ as graded smooth maps $\add : \gR^{(r_j)} \times \gR^{(r_j)} \to \gR^{(r_j)}$ and $ H_\lambda : \gR^{(r_j)} \to \gR^{(r_j)}$ in the following way; the underlying smooth maps $\uline{\add}$ and $\uline{H}_\lambda$ are the addition on $\R^{r_0}$ and multiplication by $\lambda$ on $\R^{r_0}$, respectively, and the pullbacks are given by
\begin{equation}\label{eq_additionMultiplication}
\add^\ast(k^a) = k_{(1)}^a + k_{(2)}^a, \qquad H_\lambda^\ast(k^a) = \lambda \, k^a.
\end{equation}
Here $\{k_{(1)}^a, k_{(2)}^b\}_{a,b = 1}^r$ are the coordinates on the product $\gR^{(r_j)} \times \gR^{(r_j)}$ arising as $k_{(\lambda)}^a := p_{\lambda}^\ast(k^a)$, where $p_{1,2}: \gR^{(r_{j})} \times \gR^{(r_{j})} \rightarrow \gR^{(r_{j})}$ are the respective projections. See \cite{Vysoky:2022gm} for details on products of graded manifolds. 

\begin{definition}\label{def_fiberwiseLin}
Consider two product manifolds $\calM \times \gR^{(r_j)}$ and $\calN \times \gR^{(s_j)}$, and let $\varphi : \calM \to \calN$ be a graded smooth map. A graded smooth map $\phi : \calM \times \gR^{(r_j)} \to \calN \times \gR^{(s_j)}$ is said to be \textbf{fiberwise linear} over $\varphi$, if
\begin{equation}\label{diag1}
\begin{tikzcd}
\calM \times \gR^{(r_j)}
\arrow[r, "{\phi}"]
\arrow[d, "{p_1}"]
	&\calN \times \gR^{(s_j)}
	\arrow[d, "{p_1}"]
	\\
\calM
\arrow[r, "{\varphi}"]
	& \calN
\end{tikzcd}
\end{equation}
commutes, and also
\begin{equation}\label{diag3}
\begin{tikzcd}
\calM \times \gR^{(r_j)}
\arrow[r, "{\phi}"]
\arrow[d, "{1 \times H_\lambda}"]
	& \calN \times \gR^{(s_j)}
	\arrow[d, "{1 \times H_\lambda}"]
	\\
\calM \times \gR^{(r_j)}
\arrow[r, "{\phi}"]
	& \calN \times \gR^{(s_j)}
\end{tikzcd}
\end{equation}  
commutes for any $\lambda \in \R$. 
\end{definition}

Here recall that when one deals with linear maps between finite-dimensional vector spaces, either two of the following three properties imply the third: smoothness, homogeneity and additivity. In ordinary differential geometry, the fact that smoothness and homogeneity imply additivity is known as Euler's Homogeneous Function Theorem, see e.g. \cite{grabowski2009higher} for a discussion of its importance and \cite[Proposition 4.4]{grabowskaGrabowski2024} for a $\Z_2^n$-graded version. In this light, we can view the following proposition as a $\Z$-graded version of the same theorem.

\begin{proposition}\label{prop_additivity}
Let $\calM, \calN$ be graded manifolds, $\R^{(r_j)}, \R^{(s_j)}$ finite-dimensional graded vector spaces, and let $\phi : \calM \times \gR^{(r_j)} \to \calN \times \gR^{(s_j)}$ be fiberwise linear over $\varphi : \calM \to \calN$. Then the diagram
\begin{equation}\label{diag2}
\begin{tikzcd}
\calM \times \gR^{(r_j)} \times \gR^{(r_j)}
\arrow[r, "{\phi \times_\varphi \phi}"]
\arrow[d, "{1 \times \add}"]
	& \calN \times \gR^{(s_j)} \times \gR^{(s_j)}
	\arrow[d, "{1 \times \add}"]
	\\
\calM \times \gR^{(r_j)}
\arrow[r, "{\phi}"]
	& \calN \times \gR^{(s_j)}
\end{tikzcd}
\end{equation}
commutes. (The meaning of $\phi \times_\varphi \phi$ is explained below.)
\end{proposition}
\begin{proof}
Note that $\phi \times_{\varphi} \phi$ denotes the unique arrow fitting into the commutative diagram
\begin{equation}
\begin{tikzcd}[row sep=scriptsize, column sep=tiny]
{\calM \times \gR^{(r_j)} \times \gR^{(r_j)}}
\arrow[rr, "{(p_1, p_3)}"]
\arrow[dd, "{(p_1, p_2)}"]
\arrow[rd, "{\phi \times_\varphi \phi}", dashed]
	&{}
		&{\calM \times \gR^{(r_j)}}
		\arrow[dd, "{p_1}" yshift = -15 pt]
		\arrow[rd, "{\phi}"]
			&{} \\
{}
	&{\calN \times \gR^{(s_j)} \times \gR^{(s_j)}}
	\arrow[rr, "{(p_1, p_3)}" xshift = -20pt, crossing over]	
		&{} 
			&{\calN \times \gR^{(s_j)}}
			\arrow[dd, "{p_1}"] \\
{\calM \times \gR^{(r_j)}}
\arrow[rr, "{p_1}" xshift = -15pt]
\arrow[rd, "{\phi}"]
	&{}
		&{\calM} 
		\arrow[rd, "{\varphi}"]
			&{} \\
{}
	&{\calN \times \gR^{(s_j)}}
	\arrow[rr, "{p_1}" ]
	\arrow[uu, "{(p_1, p_2)}"' yshift = 15pt, leftarrow, crossing over]
		&{} 
			&{\calN}
\end{tikzcd},
\end{equation}
since the squares in the front and in the back of the cube are pullback squares in the category $\gman$. In particular, the underlying smooth map is
\begin{equation}
\uline{\phi \times_\varphi \phi}\,(m,v,w) \equiv \uline{\phi} \times_{\uline{\varphi}} \uline{\phi} (m,v,w)  = (\uline{\varphi}(m), p_1(\uline{\phi}(m, v)), p_2(\uline{\phi}(m,w))).
\end{equation}
Commutativity of the diagram (\ref{diag2}) can be verified locally, thus we can assume that there are global coordinates $\{x^i\}_{i = 1}^{m_0}, \{\xi^\mu\}_{\mu= 1}^{\tilde{m}}$ and $\{y^I\}_{I = 1}^{n}$ on $\calM$ and $\calN$, respectively. Furthermore, let $\{k^a\}_{a = 1}^{r_0}, \{\kappa^\varkappa\}_{\varkappa = 1}^{\tilde{r}}$ be the coordinates on $\gR^{(r_j)}$ and $\{\ell^A\}_{A = 1}^{s}$ on $\gR^{(s_j)}$, as in Example \ref{example_graded_domain}. Note that here we distinguish between zero and non-zero degree coordinates on $\calM$ and $\gR^{(r_j)}$, but not on $\calN$ and $\gR^{(s_j)}$, which is purely for convenience. This gives rise to the global coordinates $x^i, \xi^\mu, k_{(1)}^a, \kappa_{(1)}^\varkappa, k_{(2)}^a, \kappa_{(2)}^\varkappa$ on $\calM \times \gR^{(r_j)} \times \gR^{(r_j)}$ and  $y^I, \ell_{(1)}^A, \ell_{(2)}^A$ on $\calN \times \gR^{(s_j)} \times \gR^{(s_j)}$. Let us label the two paths through the diagram as $\zeta := (1 \times \add)\circ \phi \times_{\varphi} \phi$ and $\zeta^\prime := \phi \circ (1 \times \add)$. To show that $\zeta = \zeta^\prime$, it is sufficient \cite[Theorem 3.29]{Vysoky:2022gm} to show that pullbacks of the coordinates by $\zeta$ and $\zeta^\prime$ agree. First, one finds that
\begin{align}
    \zeta^\ast y^I &= (\phi^\ast \times_{\varphi} \phi)^\ast \, (1 \times \add)^\ast \, y^I = (\phi^\ast \times_{\varphi} \phi)^\ast \,  y^I = \varphi^\ast y^I, \\
    (\zeta^\prime)^\ast y^I &= (1 \times \add)^\ast \phi^\ast y^I = (1 \times \add)^\ast \varphi^\ast y^I = \varphi^\ast y^I,
\end{align}
where we used that $\phi^\ast y^I = \varphi^\ast y^I$, which follows from the commutativity of (\ref{diag1}). Now for the ``fiber" coordinates: one can write their pullback by $\phi$ in the general form
\begin{equation}\label{eq17}
    \phi^\ast(\ell^A) = \sum_{{(\bld{p}, \bld{q})} \in \N_{|\ell^A|}^{\tilde{m} + \tilde{r}}} f^A_{\bld{p}, \bld{q}}(x, k) \, \xi^{\bld{p}} \, \kappa^{\bld{q}},
\end{equation}
for any index $A$. One then uses diagram (\ref{diag3}) to find that for every $\lambda \in \R$, 
\begin{equation}
    \phi^\ast (1 \times H_\lambda)^\ast \ell^A = \lambda \, \phi^\ast \ell^A = \sum_{{(\bld{p}, \bld{q})} \in \N_{|\ell^A|}^{\tilde{m} + \tilde{r}}}  \lambda \,f^A_{\bld{p}, \bld{q}}(x, k) \, \xi^{\bld{p}} \, \kappa^{\bld{q}},
\end{equation}
equals
\begin{equation}
    (1 \times H_\lambda)^\ast \phi^\ast \, \ell^A = (1 \times H_\lambda)^\ast \sum_{{(\bld{p}, \bld{q})} \in \N_{|\ell^A|}^{\tilde{m} + \tilde{r}}} f^A_{\bld{p}, \bld{q}}(x, k) \, \xi^{\bld{p}} \, \kappa^{\bld{q}} = \sum_{{(\bld{p}, \bld{q})} \in \N_{|\ell^A|}^{\tilde{m} + \tilde{r}}} \lambda^{w(\bld{q})} \, f^A_{\bld{p}, \bld{q}}( x, \lambda  k) \, \xi^{\bld{p}} \, \kappa^{\bld{q}},
\end{equation}
where $w(\bld{q}) =  \sum_{\varkappa}q_\varkappa$ is the weight of the multiindex $\bld{q}$. We see that for any $(\bld{p}, \bld{q}) \in \N_{|\ell^A|}^{\tilde{m} + \tilde{r}}$ and any $\lambda \in \R$ there holds
\begin{equation}\label{eq2}
    \lambda \, f^A_{\bld{p}, \bld{q}}(x, k) = \lambda^{w(\bld{q})} \, f^A_{\bld{p}, \bld{q}}(x, \lambda k).
\end{equation}
In the case $\bld{q} = \bld{0}$, differentiating $(\ref{eq2})$ with respect to $\lambda$ and setting $\lambda = 0$ yields
\begin{equation}
    f^A_{\bld{p}, \bld{0}}(x,k) = \dd{f_{\bld{p},\bld{0}}^A}{k^a}(x, 0 ) \,  k^a.
\end{equation}
In the case $ w(\bld{q}) = 1$, doing the same yields
\begin{equation}
    f^A_{\bld{p}, \bld{q}}(x,k) = f_{\bld{p}, \bld{q}}^A(x,0),
\end{equation}
and in the case $w(\bld{q})>1$, we obtain
\begin{equation}
    f^{A}_{\bld{p}, \bld{q}}(x,k) = 0.
\end{equation}
Consequently, all the terms with $w(\bld{q})>1$ in the general expression (\ref{eq17}) are zero and the pullbacks of $\ell^A$ have the form
\begin{equation}
    \phi^\ast (\ell^A) = \Phi\indices{^A_a} \, k^a + \Phi\indices{^A_\varkappa} \, \kappa^\varkappa,
\end{equation}
 where $\Phi\indices{^A_a}, \Phi\indices{^A_\varkappa} \in \cifty_\calM(M)$. Explicitly,
 \begin{equation}
     \Phi\indices{^A_a} = \sum_{\bld{p} \in \N_{|\ell^A|}^{\tilde{m}}}\dd{f^A_{\bld{p, \bld{0}}}}{k^a}(x, 0) \, \xi^\bld{p}, \qquad \Phi\indices{^A_\varkappa} = \sum_{\bld{p} \in \N^{\tilde{m}}_{|\ell^A|-|\kappa^\varkappa|}}f_{\bld{p}, \bld{q}_{\varkappa}}^A(x,0) \, \xi^\bld{p}.
 \end{equation}
here $\bld{q}_{\varkappa}$ is the multiindex with $q_\varkappa = 1$ and other entries zero. This can be cast in a more elegant form by re-labeling the coordinates $k^a, \kappa^\varkappa \equiv k^\alpha$, leading to
\begin{equation}
    \phi^\ast(\ell^A) = \Phi\indices{^A_\alpha} \, k^\alpha,
\end{equation}
for some graded functions $\Phi\indices{^A_{\alpha}}$ on $\calM$. Returning back to $\zeta$ and $\zeta^\prime$, it is now easy to see that
\begin{equation}
    \zeta^\ast(\ell^A) = (\phi \times_\varphi \phi)^\ast (1 \times \add)^\ast \ell^A  =  (\phi \times_\varphi \phi)^\ast  (\ell_{(1)}^A  + \ell^A_{(2)}) =  \Phi\indices{^A_\alpha} \, k_{(1)}^\alpha + \Phi\indices{^A_\alpha } \,k^\alpha_{(2)},
\end{equation}
equals
\begin{equation}
    (\zeta^\prime)^\ast (\ell^A) = (1 \times \add)^\ast  \phi^\ast \ell^A = (1 \times \add)^\ast (\Phi\indices{^A_\alpha} \, k^\alpha) = \Phi\indices{^A_\alpha} \, (k^\alpha_{(1)} + k^\alpha_{(2)}),
\end{equation}
as was to be shown.
\end{proof}

Let us mention that in the non-graded case the commutativity of diagram (\ref{diag2}) is equivalent to the usual definition of the fiberwise additivity of $\phi$.

\begin{definition}[Category $\gvbun$]\label{def_GVB2}
\begin{itemize}
   \item Consider graded manifolds $\calE, \calM$ and a graded smooth map $\pi : \calE \to \calM$ together with a collection of graded smooth maps $\{H_\lambda^\calE\}$, $H^\calE_\lambda : \calE \to \calE$ such that $\pi \circ H_\lambda^\calE = \pi$. We say that $(\calE, \calM, \pi, \{H_\lambda^\calE\}_{\lambda \in \R})$ is a \textbf{graded vector bundle} over $\calM$, if there exists an open cover $\{U_\alpha\}_{\alpha \in I}$ of $M$ together with a collection of graded diffeomorphisms $\psi_\alpha : \rest{\calM}{U_\alpha} \times \gR^{(r_j)} \to \rest{\calE}{\uline{\pi}^{-1}(U_\alpha)}$, for some fixed $\gR^{(r_j)}$, such that the diagrams
    \begin{equation}\label{diag_triv1}
        \begin{tikzcd}
        \rest{\calM}{U_\alpha} \times \gR^{(r_j)}
        \arrow[r, "{\psi_\alpha}"]
        \arrow[rd, "{p_1}"']
            &\rest{\calE}{\uline{\pi}^{-1}(U_\alpha)}
            \arrow[d, "{\pi}"]
            \\
        {}
            & \rest{\calM}{U_\alpha}
        \end{tikzcd},
    \end{equation}
    and
    \begin{equation}\label{diag_triv2}
        \begin{tikzcd}
            \rest{\calM}{U_\alpha} \times \gR^{(r_j)}
            \arrow[r, "{\psi_\alpha}"]
            \arrow[d, "{1 \times H_\lambda}"']
                & \rest{\calE}{\uline{\pi}^{-1}(U_\alpha)}
                \arrow[d, "{H^\calE_\lambda}"]
                \\
            \rest{\calM}{U_\alpha} \times \gR^{(r_j)}
            \arrow[r, "{\psi_\alpha}"]
                & \rest{\calE}{\uline{\pi}^{-1}(U_\alpha)}
        \end{tikzcd}
    \end{equation}
    commute. For the meaning of $H_\lambda$ recall (\ref{eq_additionMultiplication}). Each $H_\lambda^\calE$ is called a \textbf{homothety map} on $\calE$, $\gR^{(r_j)}$ is called a \textbf{typical fiber} of $\calE$ and any collection $\{U_\alpha, \psi_\alpha\}_{\alpha \in I}$ satisfying the above conditions is called a \textbf{local trivialization} for $\calE$.

    \item Let $(\calE, \calM, \pi, \{ H_\lambda^\calE \}_{\lambda \in \R})$ and $(\calF, \calN, \varpi, \{ H_\lambda^{\calF} \}_{\lambda \in \R})$ be two GVBs, $\phi: \calE \to \calF$ and $\varphi : \calM \to \calN$ two graded smooth maps. We say that $\phi$ is \textbf{fiberwise linear} over $\varphi$ if $\varpi \circ \phi = \varphi \circ \pi$ and $\phi \circ H_\lambda^\calE = H_\lambda^{\calF} \circ \phi$ for every $\lambda \in \R$. A pair $(\varphi, \phi)$ where $\phi$ is fiberwise linear over $\varphi$ is called a GVB \textbf{morphism} $(\varphi, \phi) : \calE \to \calF$. The resulting category is denoted as $\gvbun$.
\end{itemize}
    \end{definition}

\begin{remark}\label{rem_GVBStuff}
\begin{enumerate}
    \item We define GVBs through global homothety maps $\{H_\lambda^\calE\}$. An equivalent approach has been taken before in \cite{bruceGrabowski2024} using a global Euler vector field, i.e. a vector field of degree $0$, that locally takes the form
    \begin{equation}
        E|_{\uline{\pi}^{-1}(U_\alpha)} = k^a\dd{}{k^a},
    \end{equation}
    where $k^a$ are the fiber coordinates. In contrast to Definition \ref{def_linearInFibers} below, functions linear in fibers would be defined as eigenfunctions of $E$ of eigenvalue 1, i.e. satisfying $Ef = f$. It is easy to see that locally such functions are exactly functions linear in fiber coordinates.
    \item Let $\calE, \calM$ be graded manifolds, $\pi : \calE \to \calM$ a graded smooth map and assume there is an open cover $\{U_\alpha\}_{\alpha \in I}$ of $M$ together with graded diffeomorphisms $\psi_\alpha$ such that the diagram (\ref{diag_triv1}) commutes for every $\alpha \in I$. One can consider the transition morphisms  $\psi_{\beta \alpha} : \rest{\calM}{U_{\alpha\beta}}\times \gR^{(r_j)} \to \rest{\calM}{U_{\alpha\beta}} \times \gR^{(r_j)}$ ,  $\psi_{\beta \alpha} := \psi_\beta^{-1} \circ \psi_\alpha $, suitably restricted. If these transition morphisms are fiberwise linear over $\1_\calM$, then one may construct a unique homothety $H_\lambda^\calE : \calE \to \calE$ for every $\lambda \in \R$ locally via the diagram (\ref{diag_triv2}). Indeed, fiberwise linearity then ensures agreement on overlaps. This is a common way of constructing graded vector bundles.
    \item Depending on context, we will denote geometrically defined GVBs as $(\calE, \calM, \pi, H_\lambda^\calE)$, $\pi : \calE \to \calM$ or simply as $\calE$.

\end{enumerate}
\end{remark}

\begin{example}[Pullback bundle]\label{ex_pullbackBundle}
Consider a graded vector bundle $\pi : \calE \to \calN$ and some graded smooth map $\varphi : \calM \to \calN$. Since $\pi$ is a submersion, which follows from diagram (\ref{diag_triv1}), there exists the fiber product graded manifold
\begin{equation}
    \begin{tikzcd}[row sep=scriptsize, column sep=scriptsize]
    \calM \times_{\varphi, \pi} \calE
    \arrow[r, "{\hat{\varphi}}"]
    \arrow[d, "{\hat{\pi}}"]
        & \calE
        \arrow[d, "{\pi}"] \\
    \calM
    \arrow[r, "{\varphi}"]
        & \calN,
    \end{tikzcd}
\end{equation}
where $\hat{\varphi}$ and $\hat{\pi}$ are just labels for the fiber bundle projections, see \cite[Proposition 7.52.]{Vysoky:2022gm} for details. Let us denote $\varphi^\ast \calE := \calM \times_{\varphi, \pi} \calE$ and argue that $\hat{\pi} : \varphi^\ast \calE \to \calM$ always carries a canonical structure of a graded vector bundle of the same rank as $\calE$, such that $(\varphi, \hat{\varphi})$ becomes a GVB morphism. We will call $\varphi^\ast \calE$ the \textbf{pullback bundle} of $\calE$ by $\varphi$. It will be clear that in the trivially graded case it is indeed the well known pullback bundle.

For every $\lambda \in \R$ we find the homothety map $H_\lambda^{\varphi^\ast \calE} : \varphi^\ast \calE \to \varphi^\ast \calE$ as the unique arrow fitting into the commutative diagram
\begin{equation}\label{diag6}
\begin{tikzcd}
{\varphi^\ast \calE}
\arrow[rrd, bend left  , "{H_\lambda \circ \hat{\varphi}}"]
\arrow[rdd, bend right , "{\hat{\pi}}"]
\arrow[rd, dotted, "{H_\lambda^{\varphi^\ast \calE}}"]
    &{}
        &{} \\
{}
    &{\varphi^\ast \calE}
    \arrow[d, "{\hat{\pi}}"]
    \arrow[r, "{\hat{\varphi}}"]
        &{\calE}
        \arrow[d, "{\pi}"] \\    
{}
    &{\calM}
    \arrow[r, "{\varphi}"]
        &{\calN} \\
\end{tikzcd}.
\end{equation}
Clearly $\hat{\pi} \circ H_\lambda^{\varphi^\ast\calE} = \hat{\pi}$. To show that $(\varphi^\ast \calE, \calM, \hat{\pi}, \{H_\lambda^{\varphi^\ast \calE}\}_{\lambda \in \R})$ is a GVB, we must find its local trivialization. Let $\{(V_\alpha, \psi_\alpha)\}_{\alpha \in I}$ be a local trivialization of $\calE$ and denote $U_\alpha := \uline{\varphi}^{-1}(V_\alpha)$, for every $\alpha \in I$. Note that $\{U_\alpha\}_{\alpha \in I}$ is an open cover of $M$ and that
\begin{equation}
\parest{\varphi^\ast \calE}{\uline{\hat{\pi}}^{-1}(U_\alpha)} = \parest{\calM \times_{\varphi, \pi} \calE}{\uline{\hat{\pi}}^{-1}(U_\alpha)} = \rest{\calM}{U_\alpha} \times_{\rest{\varphi}{U_\alpha}, \rest{\pi}{\uline{\pi}^{-1}(V_\alpha)}} \rest{\calE}{\uline{\pi}^{-1}(V_\alpha)},
\end{equation}
for every $\alpha \in I$. This allows us to define $\hat{\psi}_\alpha : \rest{\calM}{U_\alpha} \times \gR^{(r_j)} \to \rest{\varphi^\ast \calE}{\uline{\hat{\pi}}^{-1}(U_\alpha)}$ through the diagram
\begin{equation}\label{diag4}
\begin{tikzcd}[row sep=scriptsize, column sep=scriptsize]
{\rest{\calM}{U_\alpha} \times \gR^{(r_j)}}
\arrow[rd, dotted, "{\hat{\psi}_\alpha}" yshift = -1 pt]
\arrow[r, "{\varphi \times 1}"]
\arrow[rdd, "{p_1}", bend right]
	&{\rest{\calN}{V_\alpha} \times \gR^{(r_j)}}
	\arrow[rd, "{\psi_\alpha}", bend left=10]
		&{}\\
{}
	&{\rest{\varphi^\ast \calE}{\uline{\hat{\pi}}^{-1}(U_\alpha)}}
	\arrow[r, "{\hat{\varphi}}"]
	\arrow[d, "{\hat{\pi}}"]
		&{\rest{\calE}{\uline{\pi}^{-1}(V_\alpha)}}
		\arrow[d, "{\pi}"] \\
{}
	&{\rest{\calM}{U_\alpha}}
	\arrow[r, "{\varphi}"]
		&{\rest{\calN}{V_\alpha}}
\end{tikzcd}
\end{equation}
In fact, $\hat{\psi}_{\alpha}$ is a graded diffeomorphism whose inverse can be defined via $p_1 \circ \hat{\psi}_\alpha^{-1} := \hat{\pi}$ and $p_2 \circ \hat{\psi}_\alpha^{-1} := p_2 \circ \psi_\alpha^{-1} \circ \hat{\varphi}$, which can be verified by composing the relevant diagrams and using the universal property.

We claim that $\{(U_\alpha, \hat{\psi}_\alpha)\}_{\alpha \in I}$ is a local trivialization of $\varphi^\ast \calE$. According to Definition \ref{def_GVB2} one needs to verify that $\hat{\pi} \circ \hat{\psi}_\alpha = p_1$ and that $H_\lambda^{\varphi^\ast \calE} \circ \hat{\psi}_\alpha = \hat{\psi}_\alpha \circ (1 \times H_\lambda)$. The former is merely the left part of diagram (\ref{diag4}). As for the latter, one sees that
\begin{equation}
   \hat{\varphi} \circ \hat{\psi}_\alpha \circ (1 \times H_\lambda) \circ \hat{\psi}_\alpha^{-1} = \psi_\alpha \circ (\varphi \times 1) \circ (1 \times H_\lambda) \circ \hat{\psi}_\alpha^{-1} = H_\lambda^\calE \circ \psi_\alpha \circ (\varphi \times 1) \circ \hat{\psi}_\alpha^{-1} = H_\lambda^\calE \circ \hat{\varphi},
\end{equation}
and 
\begin{equation}
\hat{\pi} \circ \hat{\psi}_\alpha \circ (1 \times H_\lambda) \circ \hat{\psi}_\alpha^{-1} = p_1 \circ (1 \times H_\lambda) \circ \hat{\psi}_\alpha^{-1} = p_1 \circ \hat{\psi}_\alpha^{-1} = \hat{\pi},
\end{equation}
hence $\hat{\psi}_\alpha \circ (1 \times H_\lambda) \circ \hat{\psi}_\alpha^{-1}$ fits as the dotted arrow in (\ref{diag6}) and so must be equal to $H_\lambda^{\varphi^\ast \calE}$. That $(\varphi, \hat{\varphi}) : \varphi^\ast \calE \to \calE$ is a GVB morphism can also be read from diagram (\ref{diag6}).

Let us show that $\varphi^\ast \calE$ satisfies the following expected universal property of the pullback bundle: let $\varpi : \calF \to \calS$ be a GVB, then for any GVB morphism $(\theta, \Theta) : \calF \to \calE$ such that the underlying graded smooth map $\theta$ factors through $\calM$ as $\theta = \varphi \circ \xi $ for some $\xi : \calS \to \calM$, there exists a unique GVB morphism $(\xi, \Xi) : \calF \to \varphi^\ast \calE$ such that $\Theta = \hat{\varphi} \circ \Xi$. Once more we use the universality of the fiber product, and find $\Xi$ as the unique graded smooth map fitting into the diagram
\begin{equation}
    \begin{tikzcd}[column sep = 15pt, row sep = 11pt]
        {\calF}
        \arrow[rrd, bend left, "{\Theta}"]
        \arrow[d, "{\varpi}"']
        \arrow[rd, dotted, "{\Xi}"]
            &{}
                &{} \\
        {\calS}
        \arrow[rd, bend right, "{\xi}"]
            &{\varphi^\ast \calE}
            \arrow[r, "{\hat{\varphi}}"]
            \arrow[d, "{\hat{\pi}}"]
                &{\calE}
                \arrow[, d, "{\pi}"]\\
        {}
            &{\calM}
            \arrow[r, "{\varphi}"]
                &{\calN}
    \end{tikzcd},
\end{equation}
which commutes due to $\Theta$ being fiberwise linear over $\theta = \varphi \circ \xi$. One need only verify that such $\Xi$ is fiberwise linear over $\xi$. Immediately from the definition follows $\hat{\pi} \circ \Xi = \xi \circ \varpi$. Next we need to show $\Xi \circ H_\lambda^\calF = H_\lambda^{\varphi^\ast \calE} \circ \Xi$, but as these are both arrow into $\varphi^\ast \calE$, it is enough to show that they agree when post-composed with $\hat{\pi}$ and with $\hat{\varphi}$. The first is simple, as
\begin{equation}
    \hat{\pi} \circ \Xi \circ H_\lambda^\calF = \xi \circ \varpi \circ H_\lambda^\calF = \xi \circ \varpi = \hat{\pi} \circ \Xi = \hat{\pi} \circ H_\lambda^{\varphi^\ast \calE} \circ \Xi.
\end{equation}
As for the second, consider
\begin{equation}
    \begin{tikzcd}
        {\calF}
        \arrow[r, "{\Xi}"]
        \arrow[d, "{H_\lambda^\calF}"]
            &{\varphi^\ast \calE}
            \arrow[r, "{\hat{\varphi}}"]
            \arrow[d, "{H_\lambda^{\varphi^\ast \calE}}"]
                &{\calE}
                \arrow[d, "{H_\lambda^\calE}"]\\
        {\calF}
        \arrow[r, "{\Xi}"]
            &{\varphi^\ast \calE}
            \arrow[r, "{\hat{\varphi}}"]
                &{\calE}
    \end{tikzcd},
\end{equation}
where the right square commutes by definition of homothety maps on $\varphi^\ast \calE$ and the outer rectangle commutes due to $(\theta, \Theta)$ being a GVB morphism. As a result, $\hat{\varphi} \circ H_{\lambda}^{\varphi^\ast \calE} \circ \Xi = \hat{\varphi} \circ \Xi \circ H_\lambda^\calF$, which is the last piece of the puzzle.
\end{example}

\subsection{Relation to the Sheaf Definition}

This subsection is set out in two parts: first, we construct for every $\calE \in \gvbun$ its sheaf of sections $\Gamma_\calE$. Then we show that this assignment can be extended to a fully faithful functor $\Gamma : \gvbun \to \svbun$.

\begin{lemma}\label{lma_homothety_restriction}
Let $\calE \in \gvbun$. Then the pullback by the homothety map $(H_\lambda^\calE)^\ast$ can be restricted to a morphism of sheaves of $\cifty_\calM$-modules $(H_\lambda^\calE)^\ast : \pi_\ast \cifty_\calE \to \pi_\ast \cifty_\calE$.
\begin{proof}
    Recall that the pushforward sheaf $\pi_\ast\cifty_\calE$ is a sheaf on $M$ given by $(\pi_\ast\cifty_\calE)(U) = \cifty_\calE(\uline{\pi}^{-1}(U))$ for any $U \in \Op(M)$ with the $\cifty_\calM(U)$-module structure introduced by $f\cdot h := \pi^\ast(f) h$ for all $f\in \cifty_\calM(U)$ and $h \in \cifty_\calE(\uline{\pi}^{-1}(U))$. For any $U \in \Op(M)$, there is $(\uline{H}_{\lambda}^{\calE})^{-1}(\uline{\pi}^{-1}(U)) = \uline{\pi}^{-1}(U)$ hence $H_\lambda^\calE$ can be restricted to a sheaf morphism $\pi_\ast\cifty_\calE \to \pi_\ast \cifty_\calE$.  One has 
    \begin{equation}
        (H_\lambda^\calE)^{\ast}(f \cdot h) = (H_\lambda^\calE)^{\ast}( \pi^\ast(f) \, h) = (H_\lambda^\calE)^{\ast}( \pi^\ast(f)) \, (H_\lambda^\calE)^{\ast}(h) = \pi^\ast(f) \, (H_\lambda^\calE)^{\ast}(h) = f \cdot (H_\lambda^\calE)^{\ast}(h),
    \end{equation}
    i.e. $H_\lambda^\calE$ is a morphism of sheaves of $\cifty_\calM$-modules.
\end{proof}
\end{lemma}

\begin{definition}\label{def_linearInFibers}
    Let $\pi : \calE \to \calM$ be a graded vector bundle, $U \in \Op(M)$. Then a graded smooth function $f \in \cifty_\calE(\pi^{-1}(U))$ is called \textbf{linear in fibers} if $(H_\lambda^\calE)^\ast(f) = \lambda \, f$ for every $\lambda \in \R$. The graded vector space of all such functions $f$ is denoted as $\clin_\calE(U)$.
\end{definition}

\begin{proposition}\label{prop_clinLFFG}
    The assignment $U \mapsto \clin_\calE(U)$ defines a locally freely and finitely generated sheaf of $\cifty_\calM$-modules. 
\end{proposition}
\begin{proof}
    Note that $\clin_\calE$ is in fact the intersection of kernel subsheaves of $\pi_\ast\cifty_\calE$,
    \begin{equation}
        \clin_\calE = \bigcap_{\lambda \in \R} \ker((H_\lambda^\calE)^\ast - \lambda).
    \end{equation}
    By Lemma \ref{lma_homothety_restriction}, $(H_\lambda^\calE)^\ast - \lambda$ is a morphism of sheaves of $\cifty_\calM$-modules from $\pi_\ast\cifty_\calE$ to itself, hence $\clin_\calE$ is a sheaf of $\cifty_\calM$-modules. One needs to verify that it is locally freely and finitely generated of a constant rank.
    Let $\{U_\alpha\}_{\alpha \in I}$ be some coordinate open cover of $\calM$ for which there exists a trivialization $\{U_\alpha, \psi_\alpha\}_{\alpha \in I}$ for $\calE$. If $\{x^A\}$ are coordinates on $\rest{\calM}{U_\alpha}$ and $\{k^a\}$ some dual basis of $\R^{(r_j)}$, then on $\rest{\calE}{\uline{\pi}^{-1}(U_\alpha)}$ we have coordinates $\{x^A, k^a\}$ of the same name induced by the graded diffeomorphism
    \begin{equation}
    \psi_\alpha : \rest{\calM}{U_\alpha} \times \gR^{(r_j)} \to \rest{\calE}{\uline{\pi}^{-1}(U_\alpha)}.
    \end{equation}
    Let us show that for any $V\subseteq U_\alpha $ the functions linear in fibers $\clin_\calE(V)$ are exactly those of the form
    \begin{equation}\label{eq_flifLoc}
    f = f_a \cdot k^a \equiv \pi^\ast(f_a) \, k^a,
    \end{equation}
    for some $f_a \in \cifty_\calM(V)$. First, any $f \in (\pi_\ast\cifty_\calE)(V)$ of the form (\ref{eq_flifLoc}) satisfies $H_\calE^{\lambda, \ast}(f) = \lambda f$ since $(H_\lambda^\calE)^\ast \circ \pi^\ast = \pi^\ast$. Second, every $f \in (\pi_\ast \cifty_\calE)(V)$ is of the general form\footnote{Strictly speaking, this is true for a local representative of $f$, that is the pullback $(\chi^{-1}_{\alpha})^\ast f$ by the (inverse of the) local chart $\chi_\alpha : \calE|_{\uline{\pi}^{-1}(U_\alpha)} \to (\hat{U}_\alpha, \cifty_{(m_j)}|_{\hat{U}_\alpha})\times(\R^{r_0}, \cifty_{(r_j)})$, where $\chi_\alpha = (\varphi_\alpha \times 1) \circ \psi_\alpha^{-1}$ and $\varphi_\alpha : \calM|_{U_\alpha} \to (\hat{U}_\alpha, \cifty_{(m_j)}|_{\hat{U}_\alpha})$ is the chart for $\calM$ at $U_\alpha$ that we implicitly work with throughout the proof. To improve legibility of the proof, we decided not to include the coordinate charts explicitly. We hope a more punctual reader will forgive this transgression.}
    \begin{equation}
    f = \sum_{\bld{p}, \bld{q}} f_{\bld{p}, \bld{q}}(x, k) \xi^{\bld{p}} \, \kappa^{\bld{q}},
    \end{equation}
    where we split the coordinates $\{x^A, k^a\}$ into those of degree zero $\{x^i, k^s\}$ and those of non-zero degree $\{\xi^\mu, \kappa^\iota\}$.
    The requirement $(H_\lambda^\calE)^\ast f = \lambda \, f$ then translates to
    \begin{equation}
    \sum_{\bld{p}, \bld{q}} \lambda^{w(\bld{q})} \, f_{\bld{p}, \bld{q}}(x, \lambda  \,  k) \, \xi^{\bld{p}} \, \kappa^{\bld{q}} = \lambda \, \sum_{\bld{p}, \bld{q}} f_{\bld{p}, \bld{q}}(x,  k)\, \xi^{\bld{p}} \, \kappa^{\bld{q}} ,
    \end{equation}
    which, as we have already seen in the proof of Proposition \ref{prop_additivity}, leads to $f$ necessarily taking the form (\ref{eq_flifLoc}). Thus, we find that $\clin_\calE$ is locally freely and finitely generated of a graded rank $(r_{-j})_{j \in \Z}$.
\end{proof}

\begin{definition}
    For any $\calE \in \gvbun$ we define its \textbf{sheaf of sections} $\Gamma_\calE$ as the dual to the sheaf of functions linear in fibers,
    \begin{equation}
        \Gamma_\calE := (\clin_\calE)^\ast.
    \end{equation}
\end{definition}

\begin{theorem}\label{prop_main}
    The assignment $\Gamma : \calE \mapsto \Gamma_\calE$ extends to a fully faithful and essentially surjective functor $\Gamma : \gvbun \to \svbun$. Consequently, the categories $\gvbun$ and $\svbun$ are equivalent.
\end{theorem}
\begin{proof}
\textit{(Functor).} We know how $\Gamma$ acts on objects of the category $\gvbun$. Now, let $\pi : \calE \to \calM$ and $\rho : \calF \to \calN$ be two GVBs, and let $(\varphi, \phi) : \calE \to \calF$ be their morphism. Let us describe the morphism $\Gamma(\varphi, \phi) : \Gamma_\calE \to \Gamma_{\calF}$ in the category $\svbun$. Recall that $\Gamma(\varphi, \phi) = (\tilde{\varphi}, \Lambda)$ where $\tilde{\varphi} : \calM \to \calN$ is a graded smooth map and $\Lambda : \Gamma_{\calF}^\ast \to \tilde{\varphi}_\ast (\Gamma_\calE^\ast)$ is a morphism of sheaves of $\cifty_{\calN}$-modules. The obvious choice for the underlying graded smooth map is $\tilde{\varphi} = \varphi$. Notice that $\Gamma_{\calF}^\ast = \clin_{\calF}$ and $\varphi_\ast (\Gamma_\calE^\ast) = \varphi_\ast \, \clin_\calE$ by the canonical identification of the double dual. Let us argue that $\phi^\ast : \cifty_{\calF} \to \uline{\phi}_\ast \cifty_\calE$ can be restricted to $\phi^\ast : \clin_{\calF} \to \varphi_\ast \clin_\calE$.

For any $U \in \Op(N)$, observe that, due to fiberwise linearity of $\phi$,
\begin{equation}
\uline{\phi}^{-1}(\uline{\rho}^{-1}(U)) = (\uline{\rho \circ \phi})^{-1}(U) = \uline{\pi}^{-1}( \uline{\varphi}^{-1}(U)).
\end{equation}
Hence $\phi^\ast_{\uline{\rho}^{ -1}(U)}$ maps from $\cifty_{\calF}(\uline{\rho}^{-1}(U))$ to $\cifty_\calE(\uline{\pi}^{-1}(\uline{\varphi}^{-1}(U)))$. Let $f \in \clin_{\calF}(U) \subseteq \cifty_{\calF}(\uline{\rho}^{-1}(U))$ be some function linear in fibers. We see that
\begin{equation}
H^{\calE, \ast}_{\lambda}(\phi^\ast(f)) = \phi^\ast (H^{\calF, \ast}_\lambda(f)) = \phi^\ast (\lambda f) = \lambda \, \phi^\ast(f),
\end{equation}
hence $\phi^\ast(f)$ is also linear in fibers. Thus $\phi^\ast_{\uline{\rho}^{-1}(U)}$ can be restricted to a graded linear map from $\clin_{\calF}(U)$ to $\varphi_\ast \clin_\calE (U)$. For any $h \in \cifty_{\calF}(U)$ and any $f \in \clin_{\calF}(U)$ we have
\begin{equation}
\phi^\ast (h \cdot f) = \pi^\ast (\varphi^\ast(h)) \ \phi^\ast(f) = h \cdot \phi^\ast(f),
\end{equation}
where we use the definition of $\cifty_{\calN}$-module structure for $\clin_{\calF}$ and $\varphi_\ast \clin_{\calE}$, respectively. We are therefore able to set
\begin{equation}
\Gamma(\varphi, \phi) := (\varphi, \phi^\ast),
\end{equation}
with $\phi^\ast$ restricted to functions linear in fibers.

\textit{(Fully faithful).} Once again consider two GVBs $\pi : \calE \to \calM$ and $\rho : \calF \to \calN$ in $\gvbun$, and some arrow $(\varphi, \Lambda) : \Gamma_\calE \to \Gamma_\calF$ in $\svbun$. We aim to show the existence of a unique arrow $(\varphi, \phi) : \calE \to \calF$ in $\gvbun$ such that $\Gamma(\varphi, \phi) = (\varphi, \Lambda)$. In other words, we ask whether for any $\cifty_\calN$-linear morphism $\Lambda : \clin_\calF \to \varphi_\ast \clin_\calE$ there exists a unique fiberwise linear graded smooth map $\phi : \calE \to \calF$ such that the pullback $\phi^\ast$ agrees with $\Lambda$ when restricted to $\clin_\calF$. Let $\{U_\alpha, \psi_\alpha\}$ be a local trivialization for $\rho : \calF \to \calN$ and denote $V_\alpha := \uline{\rho}^{-1}(U_\alpha)$ and $W_\alpha := \uline{\pi}^{-1}\left( \uline{\varphi}^{-1}(U_\alpha) \right)$. Note that $\{V_\alpha\}_{\alpha \in I}$ and $\{W_\alpha\}_{\alpha \in I}$ are open covers for $F$ and $E$, respectively. Let $\gR^{(s_j)}$ be the typical fiber of $\calF$, and let $\{k^a\}$ be the graded coordinates therein. For every $\alpha\in I$ we now define a graded smooth map 
\begin{equation}
    \hat{\phi}_{\alpha} : \rest{\calE}{W_\alpha} \to \rest{\calN}{U_\alpha} \times \gR^{(s_j)}.
\end{equation}
As an arrow to a product, it is determined by its composition with the two projections. We set
\begin{equation}\label{eq_pro1}
    p_1 \circ \hat{\phi}_\alpha := \parest{\varphi \circ \pi}{W_\alpha},
\end{equation}
 and $p_2 \circ \hat{\phi}_\alpha : \rest{\calE}{W_\alpha} \to \gR^{(s_j)}$ is given by choosing its pullback of the graded coordinates $k^a$ on $\gR^{(s_j)}$ to be
\begin{equation}
    (p_2 \circ \hat{\phi}_\alpha)^{\ast}(k^a) := \Lambda((\psi_\alpha^{-1})^\ast (p_2^\ast k^a)).
\end{equation}
Indeed, from the proof of Proposition \ref{prop_clinLFFG} we know that the fiber coordinates $k^a$ are functions linear in fibers (in fact, their local generators) and hence we may consider 
\begin{equation}
    \Lambda((\psi_\alpha^{-1})^\ast (p_2^\ast k^a)) \in \clin_\calE(\uline{\varphi}^{-1}(U_\alpha)) \subseteq \cifty_\calE(W_\alpha).
\end{equation} 
By \cite[Theorem 3.29]{Vysoky:2022gm}, this is enough information to define $p_2 \circ \hat{\phi}_\alpha$, and hence $\hat{\phi}_\alpha$ itself. We set $\phi_\alpha := \psi_\alpha^{-1} \circ \hat{\phi}_\alpha$ to obtain $\phi_\alpha : \rest{\calE}{W_\alpha} \to \rest{\calF}{V_\alpha}$. We need to see whether the graded smooth maps $\phi_\alpha, \phi_\beta$ agree on overlaps, i.e. when restricted to graded smooth maps $\rest{\calE}{W_{\alpha \beta}} \to \rest{\calF}{V_{\alpha \beta}}$. From the definition (\ref{eq_pro1}) we see that
\begin{equation}
    p_1 \circ \psi_\beta \circ \phi_\beta = \varphi \circ \pi = p_1 \circ \psi_\alpha \circ \phi_\alpha = p_1 \circ \psi_{\alpha \beta} \circ \psi_\beta \circ \phi_\alpha = p_1 \circ \psi_\beta \circ \phi_\alpha,
\end{equation}
when restricted to $W_{\alpha\beta}$. Similarly
\begin{align}
    (p_2 \circ \psi_\beta \circ \phi_\beta)^\ast k^a &= \Lambda\left((\psi_\beta^{-1})^\ast (p_2^\ast k^a)\right) = \Lambda\left((\psi_\alpha^{-1})^\ast (\psi_{\beta \alpha}^\ast (p_2^\ast k^a))\right) = \Lambda\left((\psi_\alpha^{-1})^\ast  (\Psi_{\beta \alpha})\indices{^a_b} \, p_2^\ast k^b)\right) \nonumber \\
    &= (\Psi_{\beta \alpha})\indices{^a_b} \Lambda((\psi_\alpha^{-1})^\ast(p_2^\ast k^b)) = (\Psi_{\beta \alpha})\indices{^a_b} (p_2 \circ \psi_\alpha \circ \phi_\alpha)^\ast k^b = \phi_\alpha^\ast \psi_\alpha^\ast (\Psi_{\beta \alpha})\indices{^a_b} \, p_2^\ast  k^b \nonumber \\
    &= \phi_\alpha^\ast \psi_\beta^\ast p_2^\ast k^a = (p_2 \circ \psi_\beta \circ \phi_\alpha)^\ast \, k^a,
\end{align}
hence $p_2 \circ \psi_\beta \circ \phi_\beta = p_2 \circ \psi_\beta \circ \phi_\alpha$ and so $\phi_\beta = \phi_\alpha$, considering all restricted to $\rest{\calE}{W_{\alpha \beta}}$. By \cite[Proposition 2.6]{Vysoky:2022gm}, there exists a unique $\phi : \calE \to \calF$ such that $\rest{\phi}{W_{\alpha}} = \phi_\alpha$ for every $\alpha \in I$. By construction, $\phi$ is fiberwise linear over $\varphi$, and the pullback $\phi^\ast$ agrees with $\Lambda$ on functions linear in fibers. Clearly, any $\phi$ with these properties must have the form above, therefore $\Gamma$ is fully faithful.

\textit{(Essentially surjective).} This part of the proof both relies and expands upon \cite[Section 5.5]{Vysoky:2022gm}. Let $\calM \in \gman$ and $\scrS \in \svbun$ be a locally freely and finitely generated sheaf of $\cifty_\calM$-modules of a constant rank. We need to find a GVB $\pi : \calE \to \calM$ such that $\scrS \cong \Gamma_\calE$.

Consider some at most countable coordinate open cover $\{U_\alpha\}_{\alpha \in I}$ for $\calM$ for which there also exists a collection of $\rest{\cifty_\calM}{U_\alpha}$-linear sheaf isomorphisms
\begin{equation}
\Lambda_\alpha :  \rest{\cifty_\calM}{U_\alpha} \otimes \R^{(r_j)} \to \rest{\scrS}{U_\alpha},
\end{equation}
where $(r_j)$ is the graded rank of $\scrS$. Let us denote as
\begin{equation}
\Lambda_\alpha^\lor : \rest{\cifty_\calM}{U_\alpha} \otimes \left(\R^{(r_{j})}\right)^\ast \to \rest{\scrS}{U_\alpha}^\ast,
 \end{equation}
the isomorphism defined by $\left(\Lambda^\lor_\alpha (k^a)\right)(\Lambda_\alpha(k_b)) := \delta\indices{^a_b}$, for some basis $\{k_a\}$ of $\R^{(r_j)}$ and its corresponding dual basis $\{k^a\}$. The transition morphisms $\Lambda_{\beta \alpha}^\lor := (\Lambda^\lor_\beta)^{-1} \circ \Lambda^\lor_\alpha$ are given by
\begin{equation}
\Lambda_{\beta \alpha}^\lor (k^a) = (T^\lor_{\beta \alpha})\indices{^a_b} k^b,
\end{equation}
for some graded smooth functions $(T^\lor_{\beta \alpha})\indices{^a_b} \in \cifty_\calM(U_{\alpha \beta})$. For every $\alpha, \beta \in I$, we may use this to define a graded smooth map
\begin{equation}
\psi_{\beta \alpha} : \rest{\calM}{U_{\alpha \beta}} \times \gR^{(r_j)} \to \rest{\calM}{U_{\alpha \beta}} \times \gR^{(r_j)},
\end{equation}
by
\begin{equation}\label{eq_TransMorph}
p_1 \circ \psi_{\beta \alpha} = p_1, \qquad \text{and} \qquad (p_2 \circ \psi_{\beta \alpha})^\ast\,  k^a = (T^\lor_{\alpha \beta})\indices{^a_b} \ p_2^\ast \, k^b.
\end{equation}
Recall that here $\{k^a\}$ play the role of graded coordinates on $\gR^{(r_j)}$. From \cite[Proposition 5.32.]{Vysoky:2022gm} we learn of the existence of a graded manifold $\calE$ together with a graded smooth map $\pi : \calE \to \calM$ and a collection of graded diffeomorphisms
\begin{equation}
\psi_\alpha : \rest{\calM}{U_\alpha} \times \gR^{(r_j)} \to \rest{\calE}{\uline{\pi}^{-1}(U_\alpha)},
\end{equation}
such that $\pi \circ \psi_\alpha = p_1$ and $\psi_{\beta \alpha} = \psi_\beta^{-1} \circ \psi_\alpha$. Clearly the transition diffeomorphisms $\psi_{\beta \alpha}$ are fiber-wise linear over the identity, and so, by Remark \ref{rem_GVBStuff}, $\pi : \calE \to \calM$ has a unique structure of a GVB. It remains to be seen that $\Gamma_\calE \cong \scrS$ as sheaves of $\cifty_\calM$-modules. In fact, it is easier to show that $\Gamma_\calE^\ast \cong \scrS^\ast$. Since $\Gamma_\calE^\ast = \clin_\calE$, there is an isomorphism
\begin{equation}
\Xi_\alpha : \cifty_\calM(U_\alpha) \otimes \left(\R^{(r_{j})}\right)^\ast \to \clin_\calE(U_\alpha)
\end{equation}
for every $\alpha \in I$, given by
$\Xi_\alpha(f_a k^a) := f_a \cdot (\psi_\alpha^{-1})^{\ast}\,  k^a \equiv \pi^\ast(f_a) \, (\psi_\alpha^{-1})^{\ast}\,  k^a$. Therefore we may define the isomorphism 
\begin{equation}
\rho_\alpha : \scrS(U_\alpha)^\ast \to \clin_\calE(U_\alpha),
\end{equation}
by $\rho_\alpha := \Xi_\alpha \circ (\Lambda^\lor_\alpha)^{-1}$. One needs only check that $\rho_\alpha$ agree on overlaps, and then use \cite[Proposition 2.6.]{Vysoky:2022gm}. But agree they do, as
\begin{equation}
\left(\Lambda_\alpha^\lor\right)^{-1} \left( \Lambda_\beta^\lor \left( k^a \right) \right) = \Lambda_{\alpha \beta}^\lor (k^a) = (T^\lor_{\alpha \beta})\indices{^a_b} \, k^b,
\end{equation}
and
\begin{align}
\Xi_\alpha^{-1} \left( \Xi_\beta \left( k^a \right) \right) 
&= \Xi_\alpha^{-1} \left( (\psi_\beta^{-1})^{\ast}  \left(k^a \right) \right) 
= \Xi_\alpha^{-1} \left( (\psi_\alpha^{-1})^{\ast}  \left(\psi_{\beta\alpha}^\ast\left(k^a \right) \right) \right) 
=  \Xi_\alpha^{-1} \left( (\psi_\alpha^{-1})^{\ast}  \left((T^\lor_{\alpha \beta})\indices{^a_b} \, k^b \right) \right) \nonumber \\
&= (T^\lor_{\alpha \beta})\indices{^a_b} \,  \Xi_\alpha^{-1} \left( (\psi_\alpha^{-1})^{\ast} (  k^b )\right) = (T^\lor_{\alpha \beta})\indices{^a_b} \, k^b,
\end{align}
hence $\Xi_\alpha \circ (\Lambda_\alpha^\lor)^{-1}  = \Xi_\beta \circ (\Lambda_\beta^\lor)^{-1}$, when properly restricted. As a result, $\rho_\alpha$ glue together a sheaf morphism $\rho : \scrS^\ast \to \clin_\calE$. This morphism is locally, hence also globally, a $\cifty_\calM$-linear isomorphism.
\end{proof}

\begin{remark}\label{note2}
Let us discuss the ``transition matrices" $(T^\lor_{\alpha \beta})\indices{^a_b}$ from the above proof. Consider any isomorphism
\begin{equation}
\Lambda_\alpha :  \rest{\cifty_\calM}{U_\alpha} \otimes \R^{(r_j)} \to \rest{\scrS}{U_\alpha},
\end{equation}
and the dual isomorphism
\begin{equation}
\Lambda_\alpha^\lor : \rest{\cifty_\calM}{U_\alpha} \otimes \left(\R^{(r_{j})}\right)^\ast \to \rest{\scrS}{U_\alpha}^\ast
\end{equation}
given by $\left(\Lambda^\lor_\alpha (k^a)\right)(\Lambda_\alpha(k_b)) := \delta\indices{^a_b}$, where $\{k_a\}$ is a basis of $\R^{(r_j)}$ and $\{k^a\}$ is its dual. Also write $\Lambda_{\alpha \beta} := \Lambda_\alpha^{-1} \circ \Lambda_\beta$, $\Lambda^\lor_{\alpha \beta} := (\Lambda^\lor_\alpha)^{-1} \circ \Lambda^\lor_\beta$. There is
\begin{equation}
\Lambda_{\alpha \beta}(k_a) =  (T_{\alpha \beta} )\indices{_a^b} \, k_b, \qquad \Lambda_{\alpha \beta}^\lor(k^a) = (T^\lor_{\alpha \beta})\indices{^a_b} \, k^b,
\end{equation}
for some graded functions $(T_{\alpha \beta} )\indices{_a^b}$ and $(T^\lor_{\alpha \beta})\indices{^a_b}$.
We find that
\begin{align}\label{eq16}
\delta\indices{^a_b} &= \left(\Lambda_\alpha^\lor(k^a)\right) \left(\Lambda_{\alpha}(k_b) \right) 
= \left( \Lambda_\beta^\lor \left( \Lambda_{\beta \alpha}^\lor \left(k^a \right) \right) \right) \left( \Lambda_\beta \left( \Lambda_{\beta \alpha} \left(k_b \right) \right) \right) 
=  \left( \Lambda_\beta^\lor \left( (T_{\beta \alpha}^\lor)\indices{^a_c} \, k^c \right) \right) \left( \Lambda_\beta \left( (T_{\beta \alpha})\indices{_b^d} \, k_d \right) \right) \nonumber \\
&= (-1)^{|k^c|(|k_b| - |k_d|)} (T_{\beta \alpha}^\lor)\indices{^a_c} \, (T_{\beta \alpha})\indices{_b^d} \, \left( \Lambda_\beta^\lor \left(  k^c \right) \right) \left( \Lambda_\beta \left(  k_d \right) \right) = (-1)^{|k^c|(|k_b| - |k_c|)} (T_{\beta \alpha}^\lor)\indices{^a_c} \, (T_{\beta \alpha})\indices{_b^c}.
\end{align}
But from definition there holds
\begin{align}
(T^\lor_{\alpha \gamma})\indices{^a_c} k^c &= \Lambda_{\alpha \gamma}(k^a) = \Lambda_{\alpha \beta}\left(\Lambda_{\beta \gamma}(k^a)\right)) = \Lambda_{\alpha \beta} \left( (T^\lor_{\beta \gamma})\indices{^a_b}k^b \right) = (T^\lor_{\beta \gamma})\indices{^a_b} \, \Lambda_{\alpha \beta}\left(k^b \right) \nonumber \\
&= (T^\lor_{\beta \gamma})\indices{^a_b} \, (T^\lor_{\alpha_\beta})\indices{^b_c} \, k^c,
\end{align}
for any $\alpha, \beta, \gamma$, thus
\begin{equation}
(T^\lor_{\alpha \gamma})\indices{^a_c} = (T^\lor_{\beta \gamma})\indices{^a_b} \, (T^\lor_{\alpha_\beta})\indices{^b_c},
\end{equation}
and in particular
\begin{equation}
(T^\lor_{\alpha \beta})\indices{^d_a}\, (T_{\beta \alpha}^\lor)\indices{^a_c} = \delta\indices{^d_c}.
\end{equation}
Applying $(T^\lor_{\alpha \beta})\indices{^d_a}$ from the left to both sides of (\ref{eq16}) gives
\begin{equation}\label{eq111}
(T^\lor_{\alpha \beta})\indices{^d_b} = (-1)^{|k^d|(|k_b| - |k_d|)} (T_{\beta \alpha})\indices{_b^d}.
\end{equation}
When constructing a total space of a GVB from its sheaf of sections, what we need are $(T^\lor_{\alpha \beta})\indices{^a_b}$, but what we often have are $(T_{\alpha \beta})\indices{_b^a}$. It is therefore useful to know the relation (\ref{eq111}) between the two.
\end{remark}

\begin{example}[Tangent bundle once again]

As a sheaf of modules, the tangent bundle $\TM$ was shown Example \ref{example_VFs}. Let $\calM$ be a graded manifold of dimension $(n_j)$, so the rank of $\TM$ is $(r_j) = (n_{-j})$. Let $\{U_\alpha\}_{\alpha \in I}$ be some coordinate open cover with coordinates $\{x^i_\alpha\}$ on $\rest{\calM}{U_\alpha}$ for every $\alpha \in I$.  Fix some basis $\{k_i\}$ for $\R^{(n_{-j})}$ and without the loss of generality suppose that $|x^i_\alpha| = - |k_i|$ for every $\alpha \in I$. We know that vector fields $\dd{}{x^i_{\alpha}}$ form a frame for $\scrX_\calM(U_\alpha)$. Locally, we have an isomorphism $\Lambda_\alpha : \cifty_\calM (U_\alpha) \otimes_{\R} \R^{(n_{-j})} \to \scrX_\calM (U_\alpha)$ given by
\begin{equation}
\Lambda_\alpha(k_i) = \dd{}{x^i_\alpha}.
\end{equation}
Hence we obtain
\begin{equation}
\Lambda_{\beta \alpha}(k_i) = \Lambda_\beta^{-1}(\dd{}{x_\alpha^i}) = \Lambda_\beta^{-1}(\dd{x_\beta^j}{x_\alpha^i}\dd{}{x_\beta^j}) = \dd{x_\beta^j}{x_\alpha^i} \, k_j.
\end{equation}
In the notation of Remark \ref{note2} we have 
\begin{equation}
(T_{\beta \alpha})\indices{_i^j} = \dd{x_\beta^j}{x_\alpha^i}, \qquad \text{and} \qquad (T^\lor_{\beta \alpha})\indices{^i_j} = (-1)^{|x^i|(|x^i| - |x^j|)} \dd{x_\alpha^i}{x_\beta^j},
\end{equation}
where $|x^i| = |x_\alpha^i| = |x_\beta^i|$. We see that the transition matrices are Jacobi matrices, as expected. The corresponding graded vector bundle $\pi : \rmT \calM \to \calM$ has a local trivialization
\begin{equation}
\psi_\alpha : \rest{\calM}{U_\alpha} \times \gR^{(n_{-j})} \to \rest{\rmT \calM}{\uline{\pi}^{-1}(U_\alpha)}
\end{equation}
giving rise to coordinates $x_\alpha^i \equiv (\psi_\alpha^{-1})^\ast(p_1^\ast x_\alpha^i)$ and $k_\alpha^i \equiv (\psi_\alpha^{-1})^\ast(p_2^\ast k^i)$ on $\rest{\rmT \calM}{\uline{\pi}^{-1}(U_\alpha)}$. The transition morphisms
\begin{equation}
\psi_{\beta \alpha} : \rest{\calM}{U_{\alpha\beta}} \times \gR^{(n_{-j})} \to \rest{\calM}{U_{\alpha \beta}} \times \gR^{(n_{-j})}
\end{equation}
are given by $p_1 \circ \psi_{\beta \alpha} = p_1$ and $(p_2 \circ \psi_{\beta \alpha})^\ast(k^i) = (-1)^{|x^i|(|x^i| - |x^j|)} \, \dd{x_\alpha^i}{x_\beta^j} \, p_2^\ast k^j$, see (\ref{eq_TransMorph}). The coordinates $k_\alpha^i$ transform accordingly as
\begin{align}
k_\alpha^i &= (\psi_{\alpha}^{-1})^\ast(p_2^\ast k^i) 
= (\psi_{\beta}^{-1})^\ast(\psi_{\beta\alpha}^\ast(p_2^\ast k^i))  
= (-1)^{|x^i|(|x^i| - |x^j|)} \, \dd{x_\alpha^i}{x_\beta^j} \, (\psi_{\beta}^{-1})^\ast(p_2^\ast k^j) \nonumber \\
&=  (-1)^{|x^i|(|x^i| - |x^j|)} \, \dd{x_\alpha^i}{x_\beta^j} \, k_\beta^j,
\end{align}
which is not quite the transformation of coordinate 1-forms, but there is an extra sign factor.
\end{example}

\begin{remark}
    The sheaf of 1-forms on a graded manifold $\calM$ can be defined the same way as in the non-graded case as the dual of the sheaf of vector fields $\Omega^1_\calM := \scrX_\calM^\ast$. To maintain sign consistency, the exterior derivative of a graded function is defined as
    \begin{equation}
        \mathrm{d} f (X) := (-1)^{|f||X|} X(f),
    \end{equation}
    for any $f \in \cifty_\calM(U)$, $X \in \scrX_\calM(U)$ and $U \in \Op(M)$. Consequently, coordinate 1-forms are not generally dual to their respective coordinate vector fields, but there is an extra sign. Let us note that coordinate 1-forms transform as 
    \begin{equation}
    d x^i = (-1)^{|y^j|(|x^i| + 1)} \dd{x^i}{y^j} \, dy^j =  dy^j \, \dd{x^i}{y^j}.
    \end{equation}
\end{remark}

\subsection{Degree Shifts \& Sections}
For a graded vector space $V \in \gvec$, its \textbf{degree shift} by $\ell \in \Z$ is defined simply as $(V[\ell])_k := V_{\ell + k}$, in particular $\R^{(r_j)}[\ell] = \R^{(r_{j + \ell})}$. If $V$ is a graded $A$-module for some $A \in \gcas$, then so is $V[\ell]$. The concept of a degree shift extends naturally to sheaves of graded $\cifty_\calM$-modules by $\scrS[\ell](U):= \scrS(U)[\ell]$, for any $U \in \Op(M)$. In fact, we may consider degree shifting as a functor $\delta[\ell]$ from the category of sheaves of graded $\cifty_\calM$-modules onto itself. 

If a sheaf of graded $\cifty_\calM$-modules $\scrS$ is locally freely and finitely generated with rank $(r_j)$, then its degree shift $\scrS[\ell]$ is also locally f\&fg with rank $(r_{j + \ell})_{j \in \Z}$. Hence we may think of $\delta[\ell]$ as an endofunctor in $\svbun$. Using the functor $\Gamma : \gvbun \to \svbun$ we can shift degrees in $\gvbun$ (non-canonically) via a functor $\Gamma^{-1} \circ \delta[\ell] \circ \Gamma$, where $\Gamma^{-1}$ is any functor such that $\Gamma^{-1} \circ \Gamma \cong \id_{\gvbun}$ and $\Gamma \circ \Gamma^{-1} \cong \id_{\svbun}$.

In particular, we start with $\pi : \calE \to \calM$ of rank $(r_j)$ and its local trivialization $\{(U_\alpha, \psi_\alpha)\}_{\alpha \in I}$, construct $\Gamma_\calE$, shift it to $\Gamma_\calE[\ell]$ by some $\ell \in \Z$ and then run it through the procedure in the proof of Proposition \ref{prop_main} to find the total space $\pi_{[\ell]} : \calE[\ell] \to \calM$. We find that it comes with a local trivialization $\{(U_\alpha, \psi^{[\ell]}_\alpha)\}_{\alpha \in I}$ where $\psi^{[\ell]}_\alpha : \rest{\calM}{U_\alpha} \times \gR^{(r_{j + \ell})} \to \rest{\calE[\ell]}{\uline{\pi}_{[\ell]}^{-1}(U_\alpha)}$. 

Furthermore, let $\{k_a\}$ be some basis of $\R^{(r_j)}$ giving rise to the coordinates $\{k^a\}$ on $\gR^{(r_j)}$. We may consider the basis $\{k_a^{[\ell]}\}$ of $\R^{(r_j + \ell)}$ with the same basis vectors (now with shifted degrees), yielding the coordinates $\{k_{[\ell]}^a\}$ on $\gR^{(r_{j + \ell})}$ with degrees $|k_{[\ell]}^a| = |k^a| + \ell$. It then turns out that the transition morphisms on $\calE$ and on $\calE[\ell]$ will be given by the \textit{same transition matrices}, i.e.
\begin{equation}\label{eq_muchos}
    \psi_{\alpha \beta}^\ast k^a = (S_{\alpha \beta})\indices{^a_b} k^b \implies (\psi^{[\ell]}_{\alpha \beta})^\ast k_{[\ell]}^a = (S_{\alpha \beta})\indices{^a_b} k_{[\ell]}^b,
\end{equation}
where $(S_{\alpha \beta})\indices{^a_b} \in \cifty_\calM(U_{\alpha \beta})$. This corresponds to the intuitive notion that to obtain $\calE[\ell]$ from $\calE$, one takes the fiber coordinates on $\calE$ and simply deigns to increase their degree by $\ell$, while preserving the transition matrices.

Knowing how to shift the degrees of vector bundles allows us to answer the question of sections.

\begin{remark}\label{rmk_sheaf_of_maps}
    Let $\calM$ and $\calN$ be graded manifolds, $\varphi: \calM \to \calN$ a graded smooth function. Then we may consider its restriction $\rest{\varphi}{U} : \rest{\calM}{U} \to \rest{\calN}{V}$ for any $V \in \Op(N)$ such that $\uline{\varphi}(U) \subseteq V$. In particular, we may simply consider $\rest{\varphi}{U} : \rest{\calM}{U} \to \calN$, where for any $Q \in \Op(N)$ and any $f \in \cifty_\calN(Q)$ there is $(\rest{\varphi}{U})^\ast f := \parest{\varphi^\ast f}{U_\alpha \cap \uline{\varphi}^{-1}(Q)}$. It is not difficult to see that the assignment $U \mapsto \{\varphi : \rest{\calM}{U} \to \calN\}$ is a sheaf on $M$ valued in the category of sets. 
\end{remark}

\begin{theorem}\label{thm_sections_and_sections}
Let $\pi : \calE \to \calM$ be a graded vector bundle. For every $U \in \Op(M)$ and every $\ell \in \Z$ denote as $\scrS(U)_\ell$ the set of graded smooth functions $\sigma : \rest{\calM}{U} \to \calE[\ell]$ such that $\pi_{[\ell]} \circ \sigma = \id_{\calM|_{U}}$. Then $\scrS$ has a canonical structure of a sheaf of $\cifty_\calM$-modules, isomorphic to $\Gamma_\calE$.
\end{theorem}
\begin{proof}
In view of Remark \ref{rmk_sheaf_of_maps} we find that for every $\ell \in \Z$, the assignment $U \mapsto \scrS(U)_\ell$ is a sheaf valued in $\Set$. Indeed, one must only verify that any graded smooth map $\sigma : \rest{\calM}{U} \to \calE[\ell]$ such that $\pi_{[\ell]} \circ \rest{\sigma}{U_\alpha} = \id$ for any open cover $\{U_\alpha\}_{\alpha \in I}$ satisfies $\pi_{[\ell]} \circ \sigma = \id$, but that is straightforward.

Next, let $\{U_\alpha, \psi_\alpha\}_{\alpha \in I}$ be a local trivialization for $\calE$ and $V \subseteq M$ an open subset and consider some $\sigma : \rest{\calM}{V} \mapsto \calE[|\sigma|]$. Also fix some basis $\{k_a\}$ of $\R^{(r_j)}$ where $(r_j)$ is the rank of $\calE$. For every $\alpha \in I$ denote $V_\alpha := V \cap U_\alpha$ and
\begin{equation}\label{eq4}
    \sigma_\alpha := \left(\psi_\alpha^{[ |\sigma| ]}\right)^{-1} \circ \rest{\sigma}{V_\alpha},
\end{equation}
i.e. $\sigma_\alpha : \calM|_{V_\alpha} \to \rest{\calM}{V_\alpha} \times \gR^{(r_{j + |\sigma|})}$. According to the first paragraph, $\sigma$ is fully determined by the collection $\{\sigma_\alpha\}_{\alpha \in I}$. Since $p_1 \circ \sigma_\alpha = \id$, $\sigma_\alpha$ is in turn fully determined by $p_2 \circ \sigma_\alpha$, or in other words by the graded functions 
\begin{equation}\label{eq5}
    \sigma_\alpha^a  := \sigma_\alpha^\ast k_{[|\sigma|]}^a \ \in \cifty_\calM(V_\alpha)_{|k^a| + |\sigma|}.
\end{equation}
Using (\ref{eq_muchos}), these graded functions transform as
\begin{equation}\label{eq3}
    \rest{\sigma_\alpha^a}{V_{\alpha\beta}} = \sigma_\alpha^\ast k_{[|\sigma|]}^a 
    = \left(\psi_{\alpha\beta}^{[ |\sigma| ]} \circ \rest{\sigma_\beta}{V_{\alpha\beta}}\right)^\ast k_{[|\sigma|]}^a 
    = \left(\rest{\sigma_\beta}{V_{\alpha\beta}}\right)^\ast \left( (S_{\alpha \beta})\indices{^a_b} k_{[|\sigma|]}^b \right)
    = (S_{\alpha \beta})\indices{^a_b} \,  \sigma_\beta^b |_{V_{\alpha \beta}},
\end{equation}
for $(S_{\alpha \beta})\indices{^a_b} \in \cifty_\calM(V_{\alpha \beta})$. Conversely, any collection of graded functions $\sigma_\alpha^a \in \cifty_\calM(V_\alpha)_{|k^a| + \ell}$ for the same $\ell \in \Z$ that transform according to (\ref{eq3}) defines a section map $\sigma : \rest{\calM}{V} \mapsto \calE[\ell]$ via (\ref{eq4}) and (\ref{eq5}).

We may now define the $\cifty_\calM$-module structure like so: for any $\sigma, \tau \in \scrS(V)$, $|\sigma| = |\tau|$, and any $f \in \cifty_\calM(V)$ we set 
\begin{align}
(\sigma + \tau)_\alpha^a &:= \sigma_\alpha^a + \tau_\alpha^a \\
(f \cdot \sigma)_\alpha^a &:= (-1)^{|f||\sigma|} \sigma_\alpha^a \, f. \label{eq6}
\end{align}
It is easy to see that these indeed transform according to (\ref{eq3}). The extra sign in (\ref{eq6}) is needed for associativity of the module action:
\begin{align}
    (f\cdot(g\cdot\sigma))_\alpha^a &= (-1)^{|f|(|g| + |\sigma|)} (g \cdot \sigma)_\alpha^a \, f = (-1)^{|f|(|g| + |\sigma|)} (-1)^{|g||\sigma|}\sigma_\alpha^a \, g \, f = (-1)^{(|f| + |g|)|\sigma|} \sigma_\alpha^a \, f \, g \nonumber \\
    &= ((fg)\cdot \sigma)_\alpha^a.
\end{align}
That restrictions on $\scrS$ are compatible with this $\cifty_\calM$-module structure is a straightforward verification, hence $\scrS$ is a sheaf of $\cifty_\calM$-modules.
For every $\alpha \in I$ we may find a frame for $\scrS(U_\alpha)$ as $\{\delta_{a}^{(\alpha)}\}_{a = 1}^r$, where $(\delta^{(\alpha)}_a)^b_\alpha = \delta\indices{_a^b}$. Note that $\delta_a^{(\alpha)} : \rest{\calM}{U_\alpha} \to \calE[-|k^a|]$. For every $\sigma : \rest{\calM}{U_\alpha} \to \calE[|\sigma|]$ we then have
\begin{equation}
    \sigma_\alpha^a = \delta\indices{_b^a} \, \sigma_\alpha^b = (\delta_b^{(\alpha)})^a_\alpha \, \sigma_\alpha^b = (\sigma_\alpha^b \cdot \delta_b^{(\alpha)})_\alpha^a \implies \sigma = \sigma_\alpha^b \cdot \delta_b^{(\alpha)}.
\end{equation}
Recall that $\clin_\calE(U_\alpha) \subseteq \cifty_\calE(\uline{\pi}^{-1}(U_\alpha))$ is generated by $\{(\psi_\alpha^{-1})^\ast k^a\}_{a = 1}^r$, and that $\Gamma_\calE(U_\alpha) = \clin_\calE(U_\alpha)^\ast$. Denote the dual frame as $\{k_a^{(\alpha)}\}_{a = 1}^r$. The isomorphism $\scrS(U_\alpha) \cong \Gamma_\calE(U_\alpha)$ is given by 
\begin{equation}\label{eq_loc_iso_sections}
    \delta_a^{(\alpha)} \mapsto  (-1)^{|k^a|} \, k_a^{(\alpha)}
\end{equation}
(no sum) and extending by $\cifty_\calM$-linearity. Using (\ref{eq3}) we find
\begin{equation}
    (\delta^{(\beta)}_a)^b_\alpha = (S_{\alpha \beta})\indices{^b_c} (\delta_a^{(\beta)})^c_\beta = (S_{\alpha \beta})\indices{^b_a}= (\delta_c^{(\alpha)})^b_\alpha \,(S_{\alpha \beta})\indices{^c_a} = (-1)^{(|k^c| - |k^a|) |k^c|}( (S_{\alpha \beta})\indices{^c_a} \cdot \delta^{(\alpha)}_c  )^b_\alpha
\end{equation}
that is
\begin{equation}
    \delta^{(\beta)}_a = (-1)^{(|k^b| - |k^a|) |k^b|} (S_{\alpha \beta})\indices{^b_a} \cdot \delta^{(\alpha)}_b = (-1)^{|k^b| + |k^a||k^b|} (S_{\alpha \beta})\indices{^b_a} \cdot \delta^{(\alpha)}_b
\end{equation}
on the other hand
\begin{equation}
   k^a_{(\beta)} = (\psi_\alpha^{-1})^\ast (\psi_\beta^{-1}\psi_\alpha)^\ast k^a = (S_{\beta \alpha})\indices{^a_b} \, k_{(\alpha)}^b
\end{equation}
and if 
\begin{equation}
    k_a^{(\beta)} = (S^\lor_{\beta \alpha})\indices{_a^b} \, k^{(\alpha)}_b,
\end{equation}
then
\begin{align}
    k_b^{(\alpha)} k_{(\beta)}^a &= (S^\lor_{\alpha \beta})\indices{_b^c}k_c^{(\beta)} k_{(\beta)}^a = (S_{\alpha \beta}^\lor)\indices{_b^a} \\
    &= k_b^{(\alpha)} (S_{\beta \alpha})\indices{^a_c} \, k^c_{(\alpha)} = (-1)^{|k^b|(|k^a| - |k^c|)}\, (S_{\beta\alpha})\indices{^a_c} k_b^{(\alpha)} k^c_{(\alpha)} = (-1)^{|k^b|(|k^a| - |k^b|)} \, (S_{\beta \alpha})\indices{^a_b},
\end{align}
and therefore
\begin{equation}
    k_a^{(\beta)} = (-1)^{|k^a|(|k^b| - |k^a|)} \, (S_{\alpha \beta})\indices{^b_a} \, k^{(\alpha)}_b = (-1)^{|k^a| + |k^a||k^b|} (S_{\alpha \beta})\indices{^b_a} \, k^{(\alpha)}_b.
\end{equation}
The isomorphism given locally by (\ref{eq_loc_iso_sections}) therefore agrees on overlaps, since
\begin{equation}
    \delta_a^{(\beta)} = (-1)^{|k^b| + |k^a||k^b|} (S_{\alpha \beta})\indices{^b_a} \cdot \delta^{(\alpha)}_b \mapsto (-1)^{|k^a||k^b|} (S_{\alpha \beta})\indices{^b_a}  k^{(\alpha)}_b = (-1)^{|k^a|} \, k_a^{(\beta)}.
\end{equation}
This completes the proof.
\end{proof}
\printbibliography

@Article{2011RvMaP..23..669C,
  author =        {{Cattaneo}, A.~S. and {Sch{\"a}tz}, F.},
  title =         {{Introduction to Supergeometry}},
  journal =       {Reviews in Mathematical Physics},
  year =          {2011},
  volume =        {23},
  pages =         {669-690},
  archiveprefix = {arXiv},
  eprint =        {1011.3401}
}

@Book{bartocci2012geometry,
  title =     {The Geometry of Supermanifolds},
  publisher = {Springer Science \& Business Media},
  year =      {2012},
  author =    {Bartocci, Claudio and Bruzzo, Ugo and Hern{\'a}ndez-Ruip{\'e}rez, Daniel},
  volume =    {71}
}

@Book{conlon2001differentiable,
  title =     {Differentiable manifolds},
  publisher = {Birkhäuser Boston},
  year =      {2001},
  author =    {Conlon, Lawrence}
}

@Article{fairon2017introduction,
  author =        {Fairon, Maxime},
  title =         {Introduction to graded geometry},
  journal =       {European Journal of Mathematics},
  year =          {2017},
  volume =        {3},
  number =        {2},
  pages =         {208--222},
  archiveprefix = {arXiv},
  eprint =        {1512.02810},
  publisher =     {Springer}
}

@Article{grabowski2009higher,
  author =        {Grabowski, Janusz and Rotkiewicz, Miko{\l}aj},
  title =         {Higher vector bundles and multi-graded symplectic manifolds},
  journal =       {Journal of Geometry and Physics},
  year =          {2009},
  volume =        {59},
  number =        {9},
  pages =         {1285--1305},
  archiveprefix = {arXiv},
  eprint =        {math/0702772}
}

@Article{Kontsevich:1997vb,
  author =        {Kontsevich, Maxim},
  title =         {{Deformation quantization of Poisson manifolds. 1.}},
  journal =       {Letters in Mathematical Physics},
  year =          {2003},
  volume =        {66},
  pages =         {157-216},
  archiveprefix = {arXiv},
  eprint =        {q-alg/9709040}
}

@Article{kotov2024category,
  author =        {Kotov, Alexei and Salnikov, Vladimir},
  title =         {{The category of $\mathbb{Z}$-graded manifolds: What happens if you do not stay positive}},
  journal =       {Differential Geometry and its Applications},
  year =          {2024},
  volume =        {93},
  pages =         {102109},
  archiveprefix = {arXiv},
  eprint =        {2108.13496},
  publisher =     {Elsevier}
}

@Book{manin1988gauge,
  title =     {Gauge Field Theory and Complex Geometry},
  publisher = {Springer-Verlag},
  year =      {1988},
  author =    {Manin, Y. I.},
  volume =    {289}
}

@Article{mehta2006supergroupoids,
  author =        {Mehta, Rajan Amit},
  title =         {{Supergroupoids, double structures, and equivariant cohomology}},
  year =          {2006},
  archiveprefix = {arXiv},
  eprint =        {math/0605356}
}

@Book{munkres2000topology,
  title =     {Topology},
  publisher = {Prentice Hall, Incorporated},
  year =      {2000},
  author =    {Munkres, J.R.},
  series =    {Featured Titles for Topology}
}

@Book{nestruev2003smooth,
  title =     {Smooth Manifolds and Observables},
  series = {Graduate Texts in Mathematics vol. 220},
  publisher = {Springer},
  year =      {2003},
  author =    {Nestruev, Jet}
}

@Article{severa2001some,
  author =        {\v{S}evera, Pavol},
  title =         {Some title containing the words ``homotopy'' and ``symplectic'', e.g. this one},
  year =          {2001},
  archiveprefix = {arXiv},
  eprint =        {math/0105080}
}

@Article{Voronov:2019mav,
  author =        {Voronov, Theodore Th.},
  title =         {{Graded Geometry, $Q$-Manifolds, and Microformal Geometry}},
  journal =       {Fortschritte der Physik},
  year =          {2019},
  volume =        {67},
  number =        {8-9},
  pages =         {1910023},
  archiveprefix = {arXiv},
  eprint =        {1903.02884}
}

@Article{Vysoky:2022gm,
  author =        {Vysok\'{y}, Jan},
  title =         {{Global Theory of Graded Manifolds}},
  journal =       {Reviews in Mathematical Physics},
  year =          {2022},
  volume =        {34},
  number =        {10},
  pages =         {2250035},
  archiveprefix = {arXiv},
  eprint =        {2105.02534}
}

@Article{vysoky2022graded,
  author =        {Vysok{\'y}, Jan},
  title =         {{Graded Generalized Geometry}},
  journal =       {Journal of Geometry and Physics},
  year =          {2022},
  volume =        {182},
  pages =         {104683},
  archiveprefix = {arXiv},
  eprint =        {2203.09533},
  publisher =     {Elsevier}
}

@Article{vysoky2024graded,
  author        = {Vysok{\'y}, Jan},
  journal       = {Journal of Geometry and Physics},
  title         = {{Graded Jet Geometry}},
  year          = {2024},
  pages         = {105250},
  archiveprefix = {arXiv},
  eprint        = {2311.15754},
  publisher     = {Elsevier},
}

@Book{walschap2012metric,
  title =     {{Metric Structures in Differential Geometry}},
  publisher = {Springer New York},
  year =      {2012},
  author =    {Walschap, G.},
  series =    {Graduate Texts in Mathematics}
}

@article{Balduzzi_2011,
year = {2011},
month = {mar},
publisher = {},
volume = {284},
number = {1},
pages = {012010},
author = {Balduzzi, L. and Carmeli, C. and Cassinelli, G.},
title = {{Super Vector Bundles}},
journal = {Journal of Physics: Conference Series},
}

@article{serre1955faisceaux,
  title={{Faisceaux Alg{\'e}briques Coh{\'e}rents}},
  author={Serre, Jean-Pierre},
  journal={Annals of Mathematics},
  volume={61},
  number={2},
  pages={197--278},
  year={1955},
  publisher={JSTOR}
}

@article{swan1962vector,
  title={{Vector bundles and projective modules}},
  author={Swan, Richard G},
  journal={Transactions of the American Mathematical Society},
  volume={105},
  number={2},
  pages={264--277},
  year={1962},
  publisher={JSTOR}
}

@book{deligne1999quantum,
  title={{Quantum Fields and Strings: A Course for Mathematicians}},
  author={Deligne, Pierre and Etingof, Pavel and Freed, Daniel S and Jeffrey, Lisa C and Kazhdan, David and Morgan, John W and Morrison, David R and Witten, Edward},
  volume={1},
  year={1999},
  publisher={American Mathematical Society}
}

@incollection{bruzzo1988supermanifolds,
  title={{Supermanifolds, supermanifold cohomology, and super vector bundles}},
  author={Bruzzo, Ugo},
  booktitle={{Differential Geometrical Methods in Theoretical Physics}},
  pages={417--440},
  year={1988},
  publisher={Springer}
}

@article{morye2025serre,
  title={{The Serre-Swan Theorem in supergeometry}},
  author={Morye, Archana S and Soman, Abhay and Devichandrika, V},
  archiveprefix = {arXiv},
  eprint        = {2503.01249},
  year={2025}
}

@article{morye2017serre,
  title={{The Serre-Swan Theorem for Ringed Spaces}},
  author={Morye, Archana S},
  journal={Analytic and Algebraic Geometry},
  pages={207--223},
  year={2017},
  publisher={Springer}
}

@article{bruce2016introduction,
  title={{Introduction to graded bundles}},
  author={Bruce, Andrew J and Grabowska, Katarzyna and Grabowski, Janusz},
  archiveprefix = {arXiv},
  eprint = {1605.03296},
  year={2016}
}

@misc{bruceGrabowski2024,
      title={{Principal bundles in the category of $\mathbb{Z}_2^n$-manifolds}}, 
      author={Andrew James Bruce and Janusz Grabowski},
      year={2024},
      eprint={2412.12652},
      archivePrefix={arXiv},
}

@misc{grabowskaGrabowski2024,
      title={{Homogeneity supermanifolds and homogeneous Darboux theorem}}, 
      author={Katarzyna Grabowska and Janusz Grabowski},
      year={2024},
      eprint={2411.00537},
      archivePrefix={arXiv},
}

@article{grabowskiRotkiewicz_2012,
   title={{Graded bundles and homogeneity structures}},
   volume={62},
   number={1},
   journal={Journal of Geometry and Physics},
   publisher={Elsevier BV},
   author={Grabowski, Janusz and Rotkiewicz, Mikołaj},
   year={2012},
   pages={21–36} }
\appendix
\section{Fundamental Theorem for GVB morphisms}
Let us prove Proposition \ref{tvrz_fiberinjsur}. We shall first prove its simplified version:

\begin{proposition} \label{prop_GVBfiberisoGVBiso}
Let $F: \scrS \rightarrow \scrP$ be a GVB morphism. Let $m \in M$ be a given point. Then the induced linear map $F_{(m)}: \scrS_{(m)} \rightarrow \scrP_{(m)}$ is an isomorphism, iff there exists $U \in \Op_{m}(M)$, such that $F|_{U}: \scrS|_{U} \rightarrow \scrP|_{U}$ is a GVB isomorphism.
\end{proposition}
\begin{proof}
Let us start with the only if part. Choose $U \in \Op_{m}(M)$ and local frames $s_{1}, \dots, s_{r}$ and $t_{1}, \dots, t_{q}$ for $\scrS$ and $\scrP$ over $U$, respectively. Since $\grk(\scrS) = \grk(\scrP)$ by assumption, we may assume that $r = q$ and $|s_{i}| = |t_{i}|$ for all $i \in \{1, \dots, r \}$. Hence we can write
\begin{equation}
    F_{U}(s_{i}) = \fF_{i}{}^{k} \cdot t_{k},
\end{equation}
where $\fF_{i}{}^{k}$ are functions in $\cifty_{\calM}(U)$. We look for a $\cifty_{\calM}(U)$-linear map $G_{U}: \scrP(U) \rightarrow \scrS(U)$ forming the inverse to $F_{U}$. One can write
\begin{equation}
    G_{U}(t_{k}) = \fG_{k}{}^{i} \cdot s_{i}.
\end{equation}
Let $\fF$ and $\fG$ denote the respective matrices of functions (the leftmost index is always the row index, regardless of its vertical position). The requirement of $G_{U}$ being the two-sided inverse to $F_{U}$ then translates into the pair of matrix equations
\begin{equation} \label{eq_fGinversetofF}
    \fF \fG = \f1, \; \; \fG \fF = \f1,
\end{equation}
where $\f1$ denotes the $r \times r$ unit matrix. First, let
\begin{equation}
    \fE_{i}{}^{k} := \left\{ \begin{array}{cc}
    \fF_{i}{}^{k} & \text{ if } |t_{i}| = |t_{k}| \\
    0 & \text{ otherwise},
    \end{array} \right.
\end{equation}
for all $i,k \in \{1,\dots,r\}$. Since $|\fF_{i}{}^{k}| = |t_{i}| - |t_{k}|$, the matrix $\fE$ is defined to contain precisely the degree zero elements of $\fF$. The induced map $F_{(m)}$ can be written in terms of the matrix $\fF$ and $\fE$ as 
\begin{equation} \label{eq_FmusingE}
    F_{(m)}(s_{i}|_{m}) = \fF_{i}{}^{k}(m) \cdot t_{k}|_{m} = \fE_{i}{}^{k}(m) \cdot t_{k}|_{m}.
\end{equation}
The first equality follows from (\ref{eq_Fmonvalueofsection}), the second from the fact that the value of any function of non-zero degree at any point is always zero. Since $\fE$ consists only of degree zero functions, we can unambiguously define a degree zero smooth function $\det(\fE)$. Since $F_{(m)}$ is assumed to be invertible, it follows from (\ref{eq_FmusingE}) that
\begin{equation}
    \det(\fE)(m) \neq 0.
\end{equation}
Since $\det(\fE) \in \cifty_{\calM}(U)$, this must be true also in some neighborhood of $m$. Without the loss of generality, we can assume that this neighborhood is the whole $U$. One can then use the standard algebraic formula to find the inverse of $\fE$:
\begin{equation}
    \fE^{-1} = \frac{1}{\det(\fE)} \text{adj}(\fE)
\end{equation}
This means that $\fF$ can be written as 
\begin{equation}
    \fF = \fE(\f1 + \fF'),
\end{equation}
where all non-zero entries of $\fF'$ have a \textit{non-zero} degree. It now suffices to find a matrix $\fP$ satisfying 
\begin{equation} \label{eq_fPmatrix}
    \fP (\f1 + \fF') = (\f1 + \fF') \fP = \f1,
\end{equation}
Without the loss of generality, we can assume that all matrices take values in the graded algebra $\cifty_{(m_{j})}(U)$, see Example \ref{example_graded_domain} and the notation introduced therein. We claim that $\fP$ can be found in terms of the Neumann series
\begin{equation}
    \fP = \sum_{n=0}^{\infty} (-1)^{n} \fF'^{n}.
\end{equation}
In particular, we must make sense of the infinite sum. For each $\fp \in \N^{\tilde{m}}_{k}$, define $w(\fp) := \sum_{i=1}^{\tilde{m}} p_{i}$. For each $i,k \in \{1,\dots,r\}$, the component $[\fF'^{n}]_{k}{}^{i}$ is a well-defined function of degree $|t_{k}| - |t_{i}|$ and a sum of $n$-fold products of elements of $\fF'$. Since all non-zero entries of $\fF'$ have a non-zero degree, this implies
\begin{equation}
    ([\fF'^{n}]_{k}{}^{i})_{\fp} = 0
\end{equation}
for any $\fp \in \N^{\tilde{m}}_{|t_{k}| - |t_{i}|}$ with $w(\fp) < n$. This means that, strictly speaking, we should have defined 
\begin{equation}
    (\fP_{k}{}^{i})_{\fp} := \sum_{n=0}^{w(\fp)} (-1)^{n} ([\fF'^{n}]_{k}{}^{i})_{\fp}.
\end{equation}
Next, let us calculate the components of the product $\fP \fF'$. For each $i,k \in \{1,\dots,r\}$, write $\Delta_{ki} := |t_{k}| - |t_{i}|$. Then for each $\fp \in \N^{\tilde{m}}_{\Delta_{ki}}$, one finds
\begin{equation}
   ([\fP \fF']_{k}{}^{i})_{\fp} = \sum_{\bld{r}\leq \bld{p}} \, \epsilon^{\bld{r}, \bld{p} - \bld{r}} (\fP_{k}{}^{j})_{\fr} (\fF'_{j}{}^{i})_{\fp - \fr} = \sum_{\bld{r}\leq \bld{p}} \, \epsilon^{\bld{r}, \bld{p} - \bld{r}} \{ \sum_{n=0}^{w(\fr)} (-1)^{n} ([\fF'^{n}]_{k}{}^{j})_{\fr} \} (\fF'_{j}{}^{i})_{\fp - \fr}.
\end{equation}
We have omitted the explicit writing of the fact that $\fr \in \N^{\tilde{m}}_{\Delta_{kj}}$ in the respective sum. Now, since $\fr \leq \fp$, we have also $w(\fr) \leq w(\fp)$ and we can thus replace the upper bound of each sum over $n$ by $w(\fp)$, the added terms are trivial anyways. One can thus write
\begin{equation}
    ([\fP \fF']_{k}{}^{i})_{\fp} = \sum_{n=0}^{w(\fp)} (-1)^{n} \sum_{\bld{r}\leq \bld{p}} \, \epsilon^{\bld{r}, \bld{p} - \bld{r}} ([\fF'^{n}]_{k}{}^{j})_{\fr} (\fF_{j}{}^{i})_{\fp - \fr} = \sum_{n=0}^{w(\fp)-1} (-1)^{n} ( [\fF'^{n+1}]_{k}{}^{i})_{\fp}.
\end{equation}
In the last step, we lowered the upper bound of the sum by one, since the last summand was trivial anyways. We thus have
\begin{equation}
    ([\fP( \f1 + \fF')]_{k}{}^{i})_{\fp} = \sum_{n=0}^{w(\fp)} (-1)^{n} ([\fF'^{n}]_{k}{}^{i})_{\fp} +  \sum_{n=1}^{w(\fp)} (-1)^{n-1} ( [\fF'^{n}]_{k}{}^{i})_{\fp} = (\f1_{k}{}^{i})_{\fp}.
\end{equation}
Note that one can repeat the procedure to show that $\fP$ is in fact a \textit{two-sided} inverse to $\f1 + \fF'$. In other words, both conditions in (\ref{eq_fPmatrix}) are satisfied. Now, if we define $\fG = \fP \fE^{-1}$ and $\fG' = \fE^{-1} \fP$, they satisfy the matrix equations
\begin{equation}
    \fF \fG = \f1, \; \; \fG' \fF = \f1.
\end{equation}
Since left and right inverses, if they both exist, are always equal to each other and unique, we have just found the unique solution to (\ref{eq_fGinversetofF}). Consequently $G_{U}$ is the two-sided inverse to $F_{U}$, which can be uniquely extended to the GVB morphism $G: \scrP|_{U} \rightarrow \scrS|_{U}$ by means of Proposition \ref{prop_GVBsoveridentity}, forming the two-sided inverse to $F|_{U}$. This finishes the only if part of the proof. 

Conversely, suppose that there is $U \in \Op_{m}(M)$, such that $F|_{U}$ is a GVB isomorphism. This implies that the induced map of stalks $F_{m}: \scrS_{m} \rightarrow \scrP_{m}$ is an isomorphism. This follows e.g. from Proposition \ref{eq_GVBinjectiveequivalent} together with Proposition \ref{eq_GVBsurjectiveequivalent}. But then so is the induced map of the quotients $F_{(m)}: \scrS_{m}/ (\frJ_{m} \cdot \scrS_{m}) \rightarrow \scrP_{m} / (\frJ_{m} \cdot \scrP_{m})$. This finishes the proof. 
\end{proof}

\begin{proof}[\textbf{The proof of Proposition \ref{tvrz_fiberinjsur}}]
    Let us only prove $\textit{1}$. Suppose $F_{(m)}$ is injective. Choose $U \in \Op_{m}(M)$ and local frames $s_{1} \dots, s_{r}$ and $t_{1} \dots, t_{q}$ for $\scrS$ and $\scrP$ over $U$, respectively. It follows from the assumption that $r \leq q$. We can thus write 
    \begin{equation}
        F_{U}(s_{i}) = \fF_{i}{}^{k} \cdot t_{k},
    \end{equation}
    to obtain an $r \times q$ matrix of functions $\fF$. It follows from the assumption and (\ref{eq_Fmonvalueofsection}) that $\fF_{i}{}^{k}(m)$ form a components of a matrix of the injective linear map $F_{(m)}$. But such a matrix contains an invertible $r \times r$  submatrix. We can thus relabel the frame for $\scrP$ to $s'_{1}, \dots, s'_{r}, t'_{1}, \dots, t'_{q-r}$ satisfying $|s_{i}| = |s'_{i}|$ for all $i \in \{1, \dots, r\}$ and write 
    \begin{equation}
        F_{U}(s_{i}) = \fF'_{i}{}^{k} \cdot s'_{k} + \fF''_{i}{}^{\alpha} \cdot t'_{\alpha},
    \end{equation}
    where the $r \times r$ matrix $\fF'_{i}{}^{k}(m)$ is invertible. It follows from the proof or Proposition \ref{prop_GVBfiberisoGVBiso} that this matrix has (after possibly shrinking $U$ when necessary) a two-sided inverse $\fG'$, that is 
    \begin{equation}
        \fF' \fG' = \f1, \; \; \fG' \fF' = \f1,
    \end{equation}
    where $\f1$ denotes the unit $r \times r$ matrix. We now define $\cifty_{\calM}(U)$-linear map $G_{U}: \scrP(U) \rightarrow \scrS(U)$ as 
    \begin{equation}
        G_{U}(s'_{i}) := \fG'_{i}{}^{k} s_{k}, \; \; G_{U}(t'_{\alpha}) := 0.
    \end{equation}
    It is easy to see that $G_{U}$ is a left inverse to $F_{U}$, which can be uniquely extended to a GVB morphism $G: \scrP|_{U} \rightarrow \scrS|_{U}$  by means of Proposition \ref{prop_GVBsoveridentity}, forming a left inverse to $F|_{U}$. This proves the only if part of \textit{1}. 
    
    Conversely, suppose there is $U \in \Op_{m}(M)$, such that $F|_{U}$ has the left inverse $G: \scrP|_{U} \rightarrow \scrS|_{U}$. The identity $G \circ F|_{U} = \1_{\scrS|_{U}}$ immediately implies $G_{m} \circ F_{m} = \1_{\scrS_{m}}$. This follows from the fact how the induced maps of stalks respect compositions. But from the construction of the induced maps of fibers, this identity also immediately implies $G_{(m)} \circ F_{(m)} = \1_{\scrS_{(m)}}$. This immediately forces $F_{(m)}$ to be injective and the proof of \textit{1} is finished. We omit the proof of \textit{2} since it uses very similar arguments.
\end{proof}
\end{document}